\documentclass[a4paper]{article}
\usepackage{amsmath,amsfonts,amssymb}
\usepackage{graphicx, epsfig}
\newtheorem{theorem}{Theorem}[section]

\newtheorem{lemma}[theorem]{Lemma}
\newtheorem{cor}[theorem]{Corollary}
\newtheorem{prop}[theorem]{Proposition}
\newtheorem{defn}[theorem]{Definition}
\newtheorem{ques}[theorem]{Question}
\newtheorem{conj}[theorem]{Conjecture}
\newtheorem{exam}[theorem]{Example}
\newtheorem{rem}[theorem]{Remark}

\newtheorem{principle}[theorem]{Principle}
\newtheorem{claim}[theorem]{Claim}

\def\proj{{\Bbb R}P^2} %projective plane
\def\torus{T^2} %torus
\def\sphere{S^2} %torus
\def\Klein{{\Bbb K}} %Klein bottle %maybe use \frak K
\def\Moe{\frak M}
\def\nbhd{neighborhood}

\newenvironment{proof}[1][Proof]{\textbf{#1.} }
{\hfill\rule{0.5em}{0.5em}\medskip}
\newenvironment{proof*}[1][Proof]{\textbf{#1.} }{}

\def\Kerekjarto{Ker\'ekj\'art\'o} %Introduced as it is hard to
\def\epsilon{\varepsilon}

\begin{document}

%%%%%%%%%%%%%%%%%%%%%%%OLD TITLES
\iffalse
%\title{Foliations on non-metric manifolds III: \\
%tame behaviours under small fundamental groups}

\iffalse
\title{Foliated tameness of non-metric
%manifolds III:
surfaces\\
%under
with small fundamental groups} \fi

\title{Foliated phase-transitions of
surfaces\\
by inflating the fundamental group}

\title{Phase-transitions of
foliated surfaces\\
by
%inflating
warming the fundamental group}

\title{Ebullition in
foliated surfaces\\
from
%inflating
warming the fundamental group}

\title{Ebullition of
foliated surfaces\\
when
%inflating
warming the fundamental group}

\title{Ebullition of
foliated surfaces\\
by
%inflating
warming the fundamental group}

\title{Boiling
foliated surfaces\\
by heating the fundamental group}

\title{Ebullition in
foliated surfaces\\
by
%inflating
warming the fundamental group}

\title{Ebullition in
foliated surfaces\\
by heating the fundamental group}

\title{Foliated surfaces: thermodynamics\\
 versus gravitational collapse}

\title{Foliated surfaces: ebullition\\
 versus gravitational collapse}
\fi

\title{Ebullition in foliated surfaces\\
 versus gravitational collapse}

%%%%%%%%%%%%%%%%%%%%%%%%%%%
%%%%%%%%%%%%%%%%FINAL TITLE
%%%%%%%%%%%%%%%%%%%%%%%%%%%%

\title{Ebullition in foliated surfaces\\
 versus gravitational clumping}

\author{Alexandre Gabard}
\maketitle

\newbox\quotation
\setbox\quotation\vtop{\hsize 8cm \noindent

\footnotesize {\it ``Beweisen hei{\ss}t, den Gedankengang auf
den Kopf stellen.''}
%
%\noindent
Oswald Teichm\"uller, 1939, in {\it Extremale quasikonforme
Abbildungen und quadratische Differentiale}.

%\smallskip
%{\it Ich musste die ver\"ucktesten Kurvenscharen integrieren.}

%\noindent Oswald Teichm\"uller writing to his Ph. D. student.

%\smallskip
%{\it Is there a life without a metric?}

%\noindent Misha Gromov in Spaces and questions, 1999.
}

\hfill{\hbox{\copy\quotation}}

\medskip

%%%%%%%%%%%%%%%%%%%%%%%%%%%EVEN SHORTER ABSTRACT FOR ARXIV
\newbox\abstract
\setbox\abstract\vtop{\hsize 12.2cm \noindent

\footnotesize \noindent\textsc{Abstract.} For surfaces, we
brush a reasonably sharp picture of the influence of the
fundamental group upon the  complexity of foliated-dynamics. A
metaphor emerges with phase-changes through the
solid-liquid-gaseous states. Groups of ranks $0\le r\le 1$ are
frozen with intransitivity reigning ubiquitously. When $2\le r
\le 3$, the marmalade starts its ebullition in the liquid
phase, with both regimes (intransitive or not) intermingled
after the detailed topology. Whenever $r\ge 4$, we reach the
gaseous-volatile phase, where any finitely-connected metric
surface is transitively foliated. The game extends
non-metrically, as putting to the fridge a frozen
configuration keeps it frozen. Gromov asked: {\it Is there a
life without a metric?}, yes surely but maybe only a cold
eternal one is worth living. The picture is pondered by a
scenario of gravitational collapse at the microscopic scale
(due to Baillif) dictating the foliated morphology when it
comes to surfaces of the Pr\"ufer type (like those of R.\,L.
Moore or Calabi-Rosenlicht).}

\centerline{\hbox{\copy\abstract}}

{\small \tableofcontents}

%\normalsize

\section{Introduction}\label{sec1}

First the absence of non-metrical nomenclature in our title is
somewhat misleading and intended not to discourage potential
readers cultivating a broader perspective. Some modest working
experience seems to indicate that non-metric manifolds are not
studied in autarchy from the metric ones, but rather as an
excrescence of them. Several
%metrical
truths
%permeate through
transcend the metrical barrier with more ease than our
psychological apprehension. Much of this plasticity originates
from the metrical impulse ({\it big bang}), and proving a
universal statement oft requires working out its metric
version first (reminding somehow what Cherry calls the
vertical structure of mathematics).
%In practice
By the way the jargon ``non-metric manifolds'' (like
non-Hausdorff manifolds, etc.) is often futile (at least
%inconvenient),
cumbersome), whenever causing an artificial subdivision of
statements holding true in
%the unified ``locally Euclidean''
a unified setting.
%As a consequence in most of our statement we shall
%not insist that the manifold is non-metric, for the statement
%is true as well metrically.
\iffalse Thus we shall omit the adjective ``non-metric''
whenever the omission does not lead to a faulty statement. Now
let us come
 concretely to our subject.
\fi
%
Trying to interpret some (existential?)
%slogan
incantation (of D. Hilbert) ``{\it Wir m\"ussen wissen. Wir
werden wissen.}'' the ultimate verdict would probably involve
an
%%%%%%incarnation
infiltration
%incantation
of
%such
non-metric manifolds into the real world (assuming its
existence, of course) via some geometric modelling (like
perhaps, quantum gravity of strings
%in ebullition
or cytoplasmic
%vibratory modes
vibrations of living
%objects).
beings). Needless-to-say we have
%not the
%%%sightliest
%slightliest
no serious idea  on how to work this out concretely (yet see
Section~\ref{Gravitation:sec} for some toy examples). At any
rate mathematically, it
%may seem quite
may look
%asymmetrical
%anomalous (odd at least)
puzzling that the continuum and the non-denumerable are well
tolerated in the small, yet not
%much so
very popular in the large.

This is surely a
%slightly
too severe caricature, as non-metric manifolds enjoy a
respectable theory
%(maybe even
(if not a drastic
%certain
{\it renouvellement des mat\'eriaux} in R.
Thom's prose) originating
%in
with the
%well-known
%key contributions
%and
seminal discoveries of:

$\bullet$ Cantor 1883\footnote{For references regarding the
following chain of items, see the bibliographies of
\cite{BGG1}, \cite{GabGa_2011}.}, Hausdorff 1915, Vietoris
1921, Alexandroff 1924: {\it long ray} and {\it long line}
(the first found non-metric manifolds, yet maybe not the
simplest to visualize),

$\bullet$ Pr\"ufer--Rad\'o 1922--1925, R.\,L. Moore
1930--1942, Calabi-Rosenlicht 1953: construction of perfectly
geometric non-metric surfaces, whose prototype is the
so-called {\it Pr\"ufer surface} discovered near the end of
1922,

$\bullet$ H. Kneser and son M. Kneser: classification of
Hausdorff $1$-manifolds into 4 species (or 7 in the bordered
case) (1958),
%plus many insightful studies including
%smoothings and
real-analytic structures on the long 1-manifolds (1960), and a
%strange
%unique-leaved foliation of codimension-one
%with a
%unique leaf
%foliation  of a
3-manifold foliated by a {\it unique}
%surface
2D-leaf (1960, 1962),

$\bullet$ M.\,E. Rudin 1974, Zenor 1975: first set-theoretical
independence result involving the concept of {\it perfect
normality} (question of Alexandroff-Wilder),

$\bullet$ Nyikos:
%%%1976--$\infty$
bagpipe structures 1984, smoothings of 1-manifolds 1989, long
cytoplasmic expansions
 of surfaces hybridizing
%at the
%borderline of
%the
Cantor and Pr\"ufer
%constructions
1990,

$\bullet$ Cannon 1969 \cite{Cannon_1969}: extension of Jordan
and Schoenflies to non-metric surfaces and an almost empty,
yet not completely nihilist, study of quasi-conformal
structures \`a la Gr\"otzsch 1928-Lavrentieff 1929-Ahlfors
1935-Teichm\"uller 1938,

$\bullet$ Gauld: independence result for powers, 125
equivalent criterions for the metrizability of a manifold,
phagocytosis principle \`a la Morton Brown: any countable
subset of a manifold is contained in a cell ($\approx$ chart
$\approx{\Bbb R}^n$),

$\bullet$ Baillif: homotopical aspects,

\noindent More modest recent contributions includes:

$\bullet$ Baillif-Gabard-Gauld 2008 \cite{BGG1}: foliated
rigidity in some long manifolds with a cylindrical structure
``squat$\times$long ray'',

$\bullet$ Gabard-Gauld 2010 \cite{GaGa2010}: re-exposition of
the Jordan-Schoenflies aspects of Cannon 1969,

$\bullet$ Gabard-Gauld 2011 \cite{GabGa_2011}: elementary
study of dynamical flows on surfaces mostly, yet with many
loose ends.

\iffalse Besides there are many brilliant expository sources
like Spivak, 1970, Milnor, 1970, etc. whereas several articles
by Nyikos provide a more serious
%historiography and
genealogy of the subject. \fi

The present paper is essentially a foliated-dual to the latter
article \cite{GabGa_2011}. Whereas in the flow-case {\it
dichotomy} (every Jordan curve separates) is---since
Poincar\'e-Bendixson---a clear-cut barrier
%%%impeding
to {\it transitivity} (dense trajectory), the foliated case
presents a more subtle landscape modulated by the ``size'' of
the fundamental group and its obstructive influence upon
foliated-transitivity (existence of a dense leaf). Precisely
when we navigate at low temperatures, say  $\pi_1$ of
low-ranks ($0\le r\le 3$), then when $0\le r\le 1$ the
situation is completely frozen (intransitive). As the rank
increases to $2$ or even $3$ the marmalade starts its
ebullition in the liquid phase (with pockets of intransitivity
still resisting, yet under progressively rarefying
circumstances controlled by the topology, cf.
Figure~\ref{monolith:fig}). Finally as the rank reaches values
$\ge 4$ then we live in the volatile-gaseous
%phase,
regime, where {\it any metric surface is transitive}.
Non-metric extensions take the following form:
%essentially
{\it frozen-intransitive
%%situation
configurations remain frozen} when imbedded into the cosmic
freezer
%glacial arena
of non-metric manifolds, whereas of
course the
%converse phenomenology does not hold,
reverse engineering foils, as putting something liquid or
gaseous in the non-metrical fridge may well create a frozen
lollypop. This happens for instance to the long plane ${\Bbb
L}^2$, which punctured as often as you please, still remains
intransitive, e.g. by the foliated rigidity previously
mentioned \cite{BGG1}.

The above metaphoric trichotomy (3 phases delineated by the
rank $r$ of $\pi_1$) quantifies somehow the well-known
principle that simple topology impedes complicated dynamics
both for {\it flows} (continuous ${\Bbb R}$-actions) as for
{\it foliations} (geometric structures
%locally
microscopically modelled
%over
after the
%partition
slicing of a number-space ${\Bbb R}^n$ into parallel
$p$-planes). Of course the range of the principle is
%admittedly rather modest for its validity must be
%confessed to be rather
%modestly
primarily two-dimensional. For instance $S^3\times S^3$ admits
a {\it minimal} (=all orbits dense) smooth flow (probably
rather chaotic) furnished by a non-constructive Baire type
argument of Fathi-Herman (1977).
%Already
In an earlier paper \cite[p.\,5]{GabGa_2011},
%one of us
we advanced  the naive
%wild
speculation that positive curvature
%could
obstructs the presence of a minimal flow on a closed manifold.
If true,
%this would have a rather gigantic
%double
the impact is rather gigantic: first all spheres
%would
lack
minimal flows (Gottschalk conjecture of 1958, still open) and
$S^3\times S^3$
%would
lack positive curvature (a still older
question of Heinz Hopf from the 1930's).
%, cf. e.g., Berger's Panoramic view of
%Riemannian geometry, 2002).

Back to the
%the
more down-to-earth two-dimensionality, the
prototype for the above principle is the Poincar\'e-Bendixson
theory,
%which is
primarily based on the Jordan separation theorem. The latter
holds not merely in the plane ${\Bbb R}^2$, but in any planar
(schlichtartig) surface.
%Furthermore
In fact Jordan separation holds true
%more generally in the non-metric
%case for
non-metrically in simply-connected surfaces (see Gabard-Gauld
2010 \cite{GaGa2010}, and also R.\,J. Cannon 1969
\cite{Cannon_1969}).
%which is maybe the first place where this
%fact appears in print).
%(to our knowledge).
To reach the ultimate generality one can
%merely
%take
adopt Jordan separation as an ``axiom''
%by
%defining
specifying the class of {\it dichotomic} surfaces and derive
the following (via the classical Poincar\'e-Bendixson
trapping-bag argument):

\begin{lemma}\label{Poinc-Bendix:flows} A dichotomic surface
%cannot support
%lack a transitive flow
is flow-intransitive
%(i.e., having
%(at least one dense
%trajectory).
%orbit).
(no dense orbit).
\end{lemma}

This applies for instance to the {\it doubled Pr\"ufer
surface}\footnote{
%For this surface, cf. also
%Figure~\ref{Train:fig}.
In our opinion, the best thing to do is (following e.g.,
Nyikos 1984 \cite{Nyikos84}) to define first the {\it bordered
Pr\"ufer surface} $P$ through a purely geometric process (e.g.
like in \cite{Gabard_2008} and the references therein esp.
R.\,L. Moore (1942), and Bredon's book), and then deduce
various versions via the operation of collaring, doubling or
folding; yielding resp. the classical Pr\"ufer surface $P_{\rm
collar}$ (appearing first in print in Rad\'o 1925
\cite{Rado_1925}, yet discovered accidentally near the end of
1922 by H. Pr\"ufer), the Calabi-Rosenlicht surface $2P$ (1953
\cite{Calabi-Rosenlicht_1953}), and the Moore surface
$M=P_{\rm folded}$ (1942 in print, yet discovered earlier
$\approx 1930$, cf. the historiography in \cite{GabGa_2011}).
As pointed out by Daniel Asimov, the drawback of our
terminology is that it
 boosts Pr\"ufer's credit vs. Calabi-Rosenlicht
contribution. Yet we feel that the bordered viewpoint is very
convenient, reducing to a single (geometric) process the
generating mode of all those manifolds. Maybe Pr\"ufer's short
life justifies anyway some little
%credit's
distortion.} $2P$, considered in Calabi-Rosenlicht 1953
\cite{Calabi-Rosenlicht_1953} (cf. also Figure~\ref{Train:fig}
for an intuitive picture and \cite[5.5]{GabGa_2011} for a
%detailed
proof of $2P$'s dichotomy). The same
%conclusion (of
intransitivity as (\ref{Poinc-Bendix:flows})
%)
is
%usurped
faulty when it comes to foliations. Indeed a
%classical
noteworthy example of Dubois-Violette 1949 \cite[p.\,897,
Point 4.]{Dubois-Violette_1949}, smoothly rediscovered in
Franks 1976 \cite{Franks_1976}, or Rosenberg 1983
\cite[p.\,29, V.\,Rem.\,2)]{Rosenberg_1983}, foliates the
thrice-punctured plane
%${\Bbb R}^2-\{3pts\}$
${\Bbb R}^2_{3*}$ by dense leaves.
%The example can be described has a closed
This is manufactured from a foliated disc with two thorns
singularities glued with a replica
%of itself
%via
after an irrational
%angled
rotation (Figure~\ref{Dubois:fig}). Alternatively it can be
regarded as the quotient of
%a Kronecker irrationally foliated
Kronecker's irrational winding of the torus divided by the
(hyper)-elliptic involution.
%However

\begin{figure}[h]
\centering
    \epsfig{figure=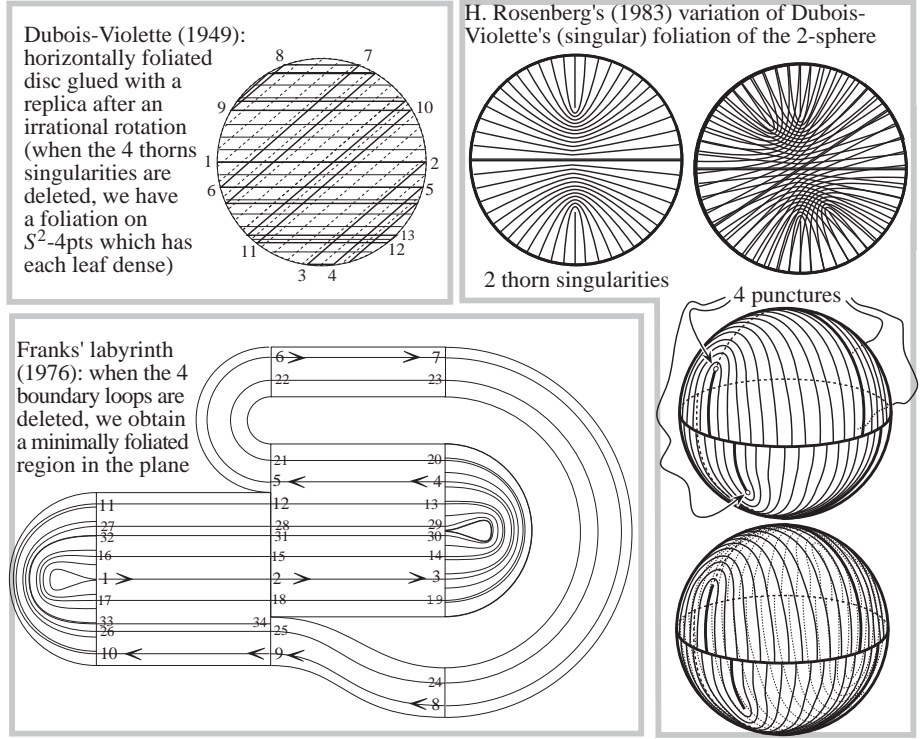,width=122mm}
  \caption{\label{Dubois:fig}
  %Artist view of
  Some labyrinths in quadruply-connected domains of the plane}
%\vskip-5pt\penalty0
\end{figure}

As recently observed by D. Gauld (in
BGG2~\cite[Prop.\,3.1]{BGG2}), the stronger
simple-connectivity impedes a dense leaf. Indeed a
%slight
localized perturbation of the foliation (akin to the closing
lemma paradigm)
%one can
creates a closed leaf (cf. Figure~\ref{David's_trick},
Case~1), which
%would
bounds a disc by Schoenflies (an absurdity
%being
as it is foliated). Here we use the universal (non-metric)
version of Schoenflies presented in \cite{GaGa2010} (also
implicit in Cannon 1969 \cite{Cannon_1969}).
%This is merely an outline of the
%precise argument which contains in fact another case (compare
%Figure~\ref{David's_trick} below).
In the light of this remark of Gauld, our immediate motivation
was two-fold:

(1) adapt the
 Haefliger-Reeb  theory (1957
\cite{Haefliger_Reeb_1957}) describing foliated structures on
the plane ${\Bbb R}^2$ to any  simply-connected (non-metric)
surfaces by a purely formal repetition of their arguments,

(2)
%trying to
exploit this general theory to deduce a somewhat rather
special result, saying that {\it surfaces whose (fundamental)
groups $\pi_1$ are infinite cyclic ${\Bbb Z}$ also lack
transitive foliations}. This should have involved the
universal cover, yet  a more Poincar\'e-Bendixson like method
turned out to be more efficient.
%(i.e., admitting at least one dense
%leaf).

Since the torus or punctured torus (with $\pi_1={\Bbb Z}^2$,
resp. $F_2$ free of rank 2) admit {\it minimal} foliations
%(i.e.,
(all leaves dense), this
%result (2)
%%would be
%gives a sharp extension of \cite[Prop.\,3.1]{BGG2} exhibiting
%
exhibits ${\Bbb Z}$ as the largest possible (fundamental)
group impeding
%the possibility of
a transitive
foliation.

During the process of aping non-metrically Haefliger-Reeb
(especially the issue that a leaf in a foliated plane divides)
we
%%derived
encountered a separation theorem generalizing the separation
by a Jordan curve (embedded circle). Specifically any
%closed proper
hypersurface which is closed as a point-set (a {\it divisor}
for short) in a
 simply-connected manifold
 (of arbitrary dimension) separates the
 %ambient
 manifold (\ref{Riemann-separation}). This
%is perhaps best
can be deduced from a trick \`a la Riemann attaching to any
divisor $H$ in a manifold a
%two-fold
double (unramified) cover polarized along $H$. Intuitively,
this covering
 %can be thought of
 %as
consists of electrically charged particles, switching their
charge signs whenever the cross the hypersurface.
 %%%It should be noticed that
({\it Warning:} This
 does not reproves the classical Jordan curve theorem
 %in the
 %Euclidean plane, yet just fail to do so,
 as our
hypersurfaces verify a local flatness condition, not a priori
known for
%an
%arbitrary
``wild'' Jordan curves in the plane, but true a posteriori via
Schoenflies.)

Then questions enchained quite naturally leading to a slightly
broader perspective which we shall now try to review. Of
course all results must be fairly classical in the metric case
(albeit as yet we were not very assiduous in locating
references). Indeed even at the metric level our exposition
contains some lacunae (maybe the most acute one being our
inaptitude to check the foliated-intransitivity of the
twice-punctured Klein bottle!), and we
%would very much
%appreciate comments from the experts in order
hope to manufacture a sharper version in the near future
(after some editorial duties).

\subsection{Overview
%of results
and methods}

%Methodologically,
Methodically, we can distinguish two trends
%which
relying either on the Schoenflies bounding disk property or on
the weaker Jordan separation. From the
%%%%dynamical
%foliated
qualitative viewpoint, the former forbids any recurrence to a
foliated chart whereas the second permits only
%very ordered
moderate recurrences
%provided one restricts attention to
for {\it oriented} foliations. Coupled with the
%%%%two-fold
double cover
%trick
induced by a non-orientable foliation
(\ref{orienting:2-fold-covering}) this allows
%in some instances
%sometimes
in some favorable situations to draw general conclusions
regarding {\it all} foliations.

\iffalse
(A) Either one uses the full strength of simple-connectivity
(and the Schoenflies theorem) or one

(B) Uses only the weaker dichotomy (i.e.
separation by Jordan curves) which albeit weaker still restricts
the foliated dynamics in the case of oriented foliations which
behave similar to flows, and to which we may apply the
Poincar\'e-Bendixson trapping argument.
\fi

%Along
The first method (mostly suggested by Gauld) gives the
following
%results in the
%(non-metrical)
repetition of Haefliger-Reeb's
%spirit:
results:
%(cross-references refer to
%the main-body of the text):

(1) {\it In a (non-metric) simply-connected surface, each leaf
appears at most once in a fixed foliated chart}
(\ref{Alex_separation})\footnote{Cross-reference to the
main-body of the text.}. Like in the metric case, this issue
is the
%corner-stone
pillar of a non-metric Haefliger-Reeb theory.
%As a
%consequence,
Consequently, the leaf-space is a (non-Hausdorff)
$1$-manifold
%, in particular any
and leaves are closed as point-sets (but of course open as
manifolds). The complete absence of recurrence forbids
%transitive foliations.
transitivity. Any leaf is locally flat and thus using
Riemann's covering trick (\ref{Rieman-polarized-cover}) it
divides the surface (\ref{Alex_separation}(d)). As a corollary
the leaf-space is a simply-connected $1$-manifold.

Via the second method based on the
weaker Jordan separation, we get the following results
using the Poincar\'e-Bendixson method:

(2) {\it In a dichotomic surface, an oriented foliation cannot
have a dense leaf, nor can a finite union of leaves be dense}
(\ref{Poinc-Bendixson_many}).  This follows by examination of
the returns of a leaf to a foliated chart, which occur in a
orderly fashion. Since foliations of  simply-connected
manifolds are orientable (\ref{orienting:2-fold-covering}),
this reproves the intransitivity of such surfaces (without
Schoenflies).
%From this
Besides,
%%%we deduce that
{\it surfaces with infinite cyclic group lack transitive
foliations} (\ref{infinite-cyclic-group}). Our proof uses
%an ad hoc
Jordan
separation
%result for
in pseudo-cylinders, namely {\it orientable surfaces with
$\pi_1={\Bbb Z}$ are dichotomic}
(\ref{dichotomy:orient_plus_inf_cycl}). (This
%separation
fails without orientability as shown
%prompted
by the M\"obius band.)
%
%Along the way of (2) we have the following conjecture:
When
%combined
married with Riemann's trick of branched coverings (familiar
in complex function theory, yet pleasant to see at work in the
foliated context), the Poincar\'e-Bendixson method gains
%even
more swing. For instance,
%shows that
{\it
 dichotomic
surfaces with $\pi_1$ free of rank $2$ lack  transitive
foliations} (\ref{dichotomic-free-of-rank-2:prop}). This
%illustrates
%prompts
shows the sharpness of Dubois-Violette's example:
%which shows that
%namely that 3
%%%three
$3$ is the minimal number of punctures in the plane
%are enough to build
to manufacture a transitive foliation (labyrinth). To complete
the picture we also notice  a non-orientable version: {\it a
non-orientable surfaces with $\pi_1$ free of rank $2$ is
foliated-intransitive}
(\ref{rank_two_non-orientable:intransitive}), plus some
sporadic obstructions in rank $3$
(\ref{intransitivity:new-obstruction}).
%Basically,
(Here is missing the issue with the Klein bottle twice
punctured, that we already confessed!)

All these results are first established metrically (with a
%special
pivotal reliance on \Kerekjarto's cylindrical ends
%theorem
(\ref{Kerkjarto:end}), as a recipe to
%permitting to
compactify
%%%%%compactificating
metric surfaces of finite-connectivity). The non-metrical
boosting involves a Lindel\"of exhaustion with calibrated
fundamental groups (\ref{calibrated-exhaustions:lemma}) (yet
another application of Schoenflies) amounting to fill in the
holes of a $\pi_1$-epimorphic subregion to adjust its group to
that of the ambient surface. The logical flattening of
%the
%full
the details frequently sidetracked us into purely topological
%lengthy
considerations, that were ultimately collected in the first
section.
%which was subsequently slightly re-lifted to give a
%presentable shape.
%Still
Despite the abundance of details,
%(some of which being still
%not completely understood by the writer),
the underlying
%phenomenology
%%%%idea
metabolism remains rather
%clear-cut:
basic:

\begin{principle}[Freudo-Lindel\"ofian Anschauung transfer--FLAT]
What\-soever you are able to see of a non-metric manifold
(which is a sort of
%elusive
quantum-plasma in ebullition) it is
%essentially
(in first approximation) its Lindel\"of subregions (those
truly accessible to the ``Anschauung'') which govern both the
qualitative ``analysis situs'' (Jordan, Schoenflies,
orientability, dichotomy, fundamental group, etc.) as well as
the foliated (or dynamical) destiny of the whole.
\end{principle}

%Sometimes
Little corrections are required when a truly non-metric
phenomenology is prompted by a particular
%geometric structure in the
manifold. Yet this is really a second strata of sophistication
not affecting
%very much
tremendously the generic
%validity
value of the first principle.

Beside those geometrical
%/combinatorial
methods, we have also
%should not forget
a point-set obstruction:

(3) A transitive one-dimensional foliation (abridged
$1$-foliation) of an $n$-manifold $M^n$ with $n\ge 2$ implies
{\it separability}\footnote{Existence of a countable dense
subset.} of the underlying $M^n$ (\ref{separability}). In fact
more is true:
%in the sense that
any chaotic behaviour of a one-dimensional leaf is caused
 by one of its metrical short end, whereas long
sides of leaves (being  sequentially-compact) are always
%``tamely''
``decently'' {\it properly} embedded (the leaf-topology
matches with the relative topology) (\ref{long-semi-leaf}).

The above results reflects so-to-speak
%so-called
qualitative
%aspects
features of foliations on some classes of topologically
%conditioned
particularized surfaces. A dual aspect is the {\it
quantitative theory}, asking for a
%(complete
classification of foliations (on a fixed manifold). This game,
which is almost always hopeless in the metric realm (e.g.,
${\Bbb R}^2$ is hard yet well-understood, $S^3$ hopeless),
turns out to be
%sometimes
%miraculously easier
much easier on some special non-metric surfaces like, e.g.,
the long plane ${\Bbb L}^2$ \cite{BGG1}. Here and on some
related surfaces one
%observes
experiments a rigidity in the large, imposing an asymptotic
leaves pattern
%leaving only
with freedom
%some free
left only on certain metric subregions, viz.  squares
transversally foliated along two opposite sides and
tangentially on the remaining
%sides.
two. By the theorem of \Kerekjarto-Whitney
\cite{Kerekjarto_1925}, \cite{Whitney33} creating for oriented
foliations compatible flows (valid only in the metric case),
%we showed that
such a square permits (up to homeomorphism) a unique foliated
extension of its boundary data. It followed in \cite{BGG1}
that ${\Bbb L}^2$ tolerates only 2 foliations up to
homeomorphism. (This is to be contrasted with the menagerie of
foliations grooving the plane ${\Bbb R}^2$.)
%Of course

A plain consequence of this rigidity is the intransitivity of
thrice-punctured long-plane
%${\Bbb L}^2-\{3pts\}$
${\Bbb L}^2_{3*}$ despite its group, $F_3$ (free of rank 3)
(\ref{puncturing}), is one
%tolerating
susceptible of complicated foliated dynamics (recall
Dubois-Violette). Hence, albeit the fundamental group has much
to say, it does not control completely the
%foliated dynamics
situation, which depends ultimately upon
%the
some finer  granularity
%of the
(encoded in the geometry of the manifold). For an even simpler
example, the bagpipe $\Lambda_{0,4}$ with orientable bag of
genus 0 with 4 contours and pipes
%homeomorphic to
modelled after the long cylinder $S^1\times {\Bbb L}_{\ge 0}$,
is a (dichotomic) surface with $\pi_1=F_3$, yet intransitive.
In fact $\Lambda_{0,4}$ cannot even be foliated  \cite{BGG1},
because any pipe acts as a black hole aspirating leaves in a
purely vertical fashion or creating many horizontal circle
leaves $S^{1}\times \{\alpha\}$. Thus an appropriate surgery
reduces one to the
%classic
compact Euler-Poincar\'e obstruction. Both examples cited are
%in fact
trivial,
%since
inasmuch as their intransitivity also  derives from their
non-separability (via (3) above).

The transitivity decision problem becomes more
%pernicious
perfidious if one
%asks
wonders about the transitivity of the separable, $M_{3*}$,
thrice-punctured Moore surface. (Recall that the Moore surface
is the folded Pr\"ufer surface, cf. Figure~\ref{Train:fig} for
an intuitive picture.) Albeit there is no universal algebraic
obstruction (recall again Dubois-Violette), we experiment a
geometric one related to the
%finer
granularity of the Moore
surface (\ref{Baillif:adapted_by_Gabard}). Indeed an argument
of M. Baillif (Buenos Aires era, near 2008--2009,
%published
under press in \cite{BGG2} and reproduced below
(\ref{Baillif:nano-black-holes})) the ``thorns'' of the Moore
surface (i.e., the folded images of the boundaries of
Pr\"ufer)
 acts as a ``continuum'' series of miniature
 black holes inveigling {\it almost all}
leaves. Precisely, {\it almost all thorns (all but at most
countably many exceptions)  are semi-leaves of any foliated
Moore surface} (\ref{Baillif:nano-black-holes}). This
%gives
implements a scenario of gravitational collapse at the
microscopic scale
%of
%nano-black holes,
%(quantum fluctuations)
in sharp contrast---but somehow dual---to the macroscopic
scale at which lives the (super-massive) black hole
%idealized
%by a
sitting at the long end of a Cantor cylinder $S^1\times {\Bbb
L}_+$ \cite{BGG1}. All these examples imaginatively suggest to
contemplate foliated structures (like lignite distributions)
merely a mean of evidencing the magneto-gravitational
($\approx$geometric, since Newton-Euler vs.
Leibniz-Euler(!)-Riemann-Einstein\footnote{Compare, e.g., the
historiography in Speiser 1927 \cite{Speiser_1927}.})
%salient
%features or
anomalies of the underlying manifold.

The above results
%are still rather fragmentary,
%for instance they
do not
%answer the question if
tell whether the  doubled Pr\"ufer surface $2P$ accepts a
labyrinth (=transitive foliation). Yet, it probably does in
view of the toy:

\begin{exam} (Dubois-Violette Pr\"uferized) \label{Dubois}
{\rm Taking Dubois-Violette's foliated disc
(Fig.\,\ref{Dubois:fig}, Rosenberg's version) and Pr\"uferize
the 2 leaved-arcs ending to the $2$ punctures to get a
bordered surface $\wp$, which glued with a copy after an
irrational twist, yields a separable surface $2\wp$
transitively foliated, which is non-metric, dichotomic, etc.
and very resemblant to $2P$.}
\end{exam}

It is not impossible (and indeed highly probable)
%, albeit we
%do not want to get our hands too dirty)
that this Pr\"uferized
surface $2\wp$ is homeomorphic to $2P$ (just ``magnify'' some
bridges). Yet this foliation
%would not be
is not minimal, and it is natural to wonder if $2P$ is
minimally foliated (more about this soon!). For flows, an easy
argument of propagation \cite{GabGa_2011} showed that
minimality forces metrisability, raising some hope to classify
all flow-minimal surfaces. Presumably not so with foliated
surfaces, compare \cite{BGG2} for some minimally foliated
non-metric surfaces. The simplest example, depicted on
Figure~\ref{Venn} below, involves a punctured torus minimally
foliated \`a la Kronecker with one of the two leaves ending to
the puncture elongated up to reach the length of Cantor's long
ray). This Pinocchio expansion near the puncture exploit a
construction of Nyikos (cf. \cite{BGG2} for more details and
the original reference).
%lying at the border line of the
%Cantor and Pr\"ufer constructions.

Now what about $2P$ being minimally foliated? Since the
fundamental group is extremely voluminous (free on a continuum
of generators),  the rank is big,
%melting
pushing the surface in the very gaseous-volatile regime where
transitivity mutates in minimality. However the real answer is
quick and easy thanks to the gravitational clumping
%%%phraseology of Penrose p.322, means masse, gros bloc, groupe
%%%bouquet d'arbre, etc... maybe clustering
%%%means also grappe and groupe massif
of Baillif, to the effect that a (finally violent)
condensation of diffuse gas must occur along the ``bridges''
(i.e. the images of the
%contours
boundaries of $P$ in $2P$ via the canonical inclusion
$P\hookrightarrow 2P$).
% are responsible of a
%special clustering,
Indeed arguing as for the Moore surface
(\ref{Baillif:nano-black-holes}), Baillif's method shows that
in any foliation of $2P$ {\it almost all} bridges are leaves
(as above this means all but countably many exceptions). Thus
we have with $2P$ (or better its Dubois-Violette model $2\wp$
(\ref{Dubois}))
%, where $\wp$ is Dubois-Violette's
%disc Pr\"uferized along 2 arcs going to the punctures)
%an
%example of
a surface which is transitive, yet not minimal. (The author
does not know if such an example exist metrically.)
%Surely or maybe not. It is
%certainly trivial, yet )

\subsection{Questions and
%possible
ramifications}

%%%%%%NOW SOLVED HENCE CENSURED
\iffalse In view of the aforementioned obstruction
(\ref{dichotomic-free-of-rank-2:prop}), any twice-punctured
simply-connected surface is intransitive under a foliation;
e.g., the twice-punctured Moore surface, $M_{2*}$. If we allow
three punctures in Moore $M$, then there is no algebraic
obstruction to transitivity, yet it is not
%completely
%trivial
easy to construct a transitive foliation. \fi

\iffalse (A similar question arise if we puncture once the
M\"obius band, then $\pi_1\approx F_2$ but non-dichotomic
since non-orientable, and we cannot apply our obstruction
(\ref{dichotomic-free-of-rank-2:prop}) to transitivity.) \fi

Here we mention a short list of questions which are probably
not structurally hard, but rather unsolved due to the
incompetence of the writer.

(1) {\it Some metrical missing links}. As just noticed what is
the simplest example of a metric surface which is transitively
foliated but not minimally. Also is the twice punctures Klein
bottle $\Klein_{2*}$ foliated-intransitive? This is actually
the only missing case to complete our picture
(Figure~\ref{monolith:fig}) classifying finitely-connected
surfaces according to their foliated-transitivity. Which
metric surfaces can be {\it biminimally foliated} (i.e. so
that each semi-leaves are dense)? Cf.
(\ref{biminimal-impeded-by-puncture}) for a partial answer.

(2) {\it Any pseudo-Moore surface foliates?}
In the case of flows, the
%engulfing
{\it phagocytosis lemma} (saying that any countable subset of
a manifold is contained in a chart)
%had
found a nice application to what we called in
\cite{GabGa_2011} the  {\it pseudo-Moore problem} (no
non-singular flow on a non-metric, simply-connected, separable
surface). Such surfaces are referred to as {\it pseudo-Moore},
with the Moore surface being the simplest prototype. In the
foliated case, it is not obvious to guess an applied avatar,
except for the over-optimistic option that all pseudo-Moore
surfaces foliate.
%Naively speaking
Recall that for flows, the Moore surface had no brush, and
this
%extended to
turned out to be the
%destiny
fate of any pseudo-Moore surface \cite{GabGa_2011}. But now
the Moore surface foliates, thus should we expect that any
pseudo-Moore surface foliates? We believe the answer is
negative, in view of Nyikos long cytoplasmic expansions (cf.
the discussion following Question
(\ref{corona-sun:question})).

(3) {\it Euler obstruction in the $\omega$-bounded case.}
Another
%%%serious
frustrating problem is what happens to the
%fundamental
Euler-Poincar\'e obstruction? Specifically we conjecture that
$\omega$-bounded\footnote{A space is  {\it $\omega$-bounded}
if countable subsets have compact closures. In the
manifold-case, this amounts to  Lindel\"of subregions having
compact closures. This concept is a non-metric avatar of
compactness, especially acute for surfaces in view of Nyikos'
bagpipe theorem.} surfaces with $\chi<0$ lack foliations
(independently of any
%specific knowledge
specification of the pipes).
%This would be an extension of the Euler-Poincar\'e
%obstruction in the compact case.
In the case of flows, it was
%rather
comparatively easy to show \cite{Gab_2011_Hairiness} that a
non-vanishing $\chi\neq 0$ obstructs
%%%the existence of a
non-singular flows (non-metric hairy-ball theorem).

(4) {\it Freeness of the fundamental group of
curves(=non-Hausdorff $1$-manifolds) and the Haefliger
twistor.}
%Still another question is
%How to
Can somebody
%formalize properly
prove the (hypothetical) Lemma~\ref{twistor} below, which
seems to be folklore since Haefliger 1955
\cite{Haefliger_1955}, and which could play a crucial r\^ole
in showing that all (non-Hausdorff) $1$-manifolds have a free
fundamental group. Prior to this we
%give a proof of
show the $\pi_1$-freeness
%of the $\pi_1$
of all open
%(non-compact,
(Hausdorff) surfaces by reduction to the metric case
(\ref{freeness-for-open-surfaces:prop}). (Hausdorffness is of
course essential, as seen by picturing flying-saucers, e.g.,
$S^1\times (\text{line with two origins})$ with $\pi_1={\Bbb
Z}^2$.)

In guise of provisory conclusion, we
%make the following
diagnostic that our
%present
understanding of foliated structures (especially on surfaces)
is slightly less sharp than the corresponding one for
dynamical flows, where deeper paradigms entered effectively
into the arena (like phagocytosis or the Euler-Poincar\'e
obstruction).
%which had a clear-cut
%incarnation.
%avatar.

\section{Topological
%preliminaries
preparations}

This section collects
%some
purely topological
%basic
results, independent of (yet related to) our foliated
investigations. The reader can skip
%this section
it
%and
referring to individual results later if necessary.
%(As the
%writer was rather short-sighted i)
%Alternatively one can ignore foliated questions and
%concentrate purely on the present section.
The Leitfaden below
is supposed to help
%not getting lost into
navigating through the menagerie of
%foundational
details.

\begin{figure}[h]
%\centering
    \hskip-20pt\epsfig{figure=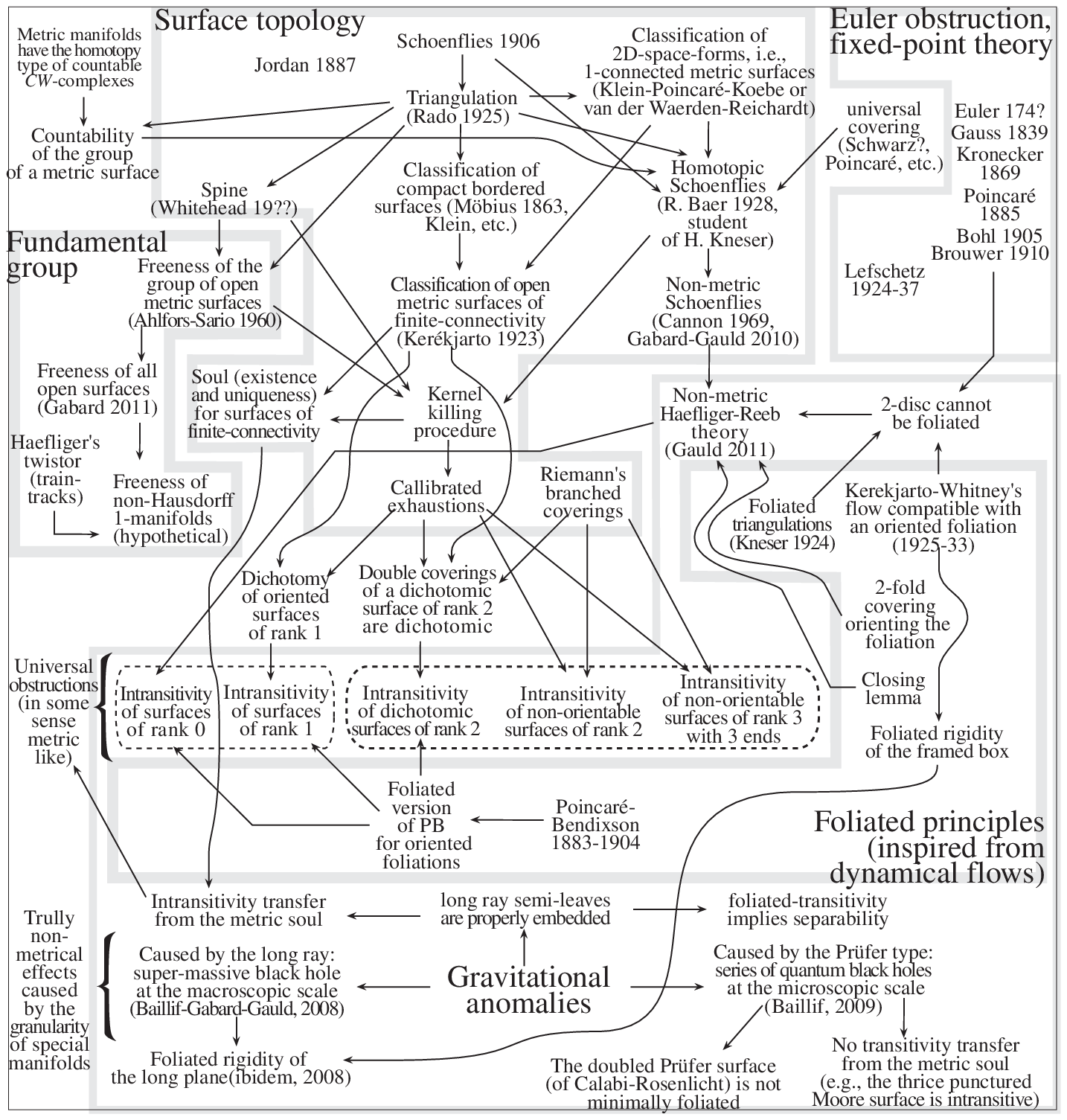,width=139.2mm}
  %\caption{\label{diagr}
  %Artist view of
  %Jordan separation in a simply-connected surface via a
  %Schoenflies trapping}
%\vskip-5pt\penalty0
\end{figure}

\subsection{Foundations (Leibniz, Euler, Gauss, Listing,
M\"obius, Riemann,  Klein, Dyck,
%%Poincar\'e,
Schoenflies,
\Kerekjarto, Rad\'o)}

We first recall without proofs (but cross-references) the key
results in the topology of the plane and surfaces. A pillar of
the theory is the following theorem often attributed to
Schoenflies (1906), albeit there are serious
function-theoretical competitors building over the Riemann
mapping theorem and conformal representation (including Osgood
1900--1913, Carath\'eodory 1912, Hilbert-Courant, etc.), not
to mention the early attempt in M\"obius 1863
\cite{Moebius_1863}:

\begin{theorem} (Schoenflies 1906)\label{Schoenflies:thm-plane}
Any Jordan curve (embedded circle) in the plane ${\Bbb R}^2$
bounds a disc.
\end{theorem}

%Using it, Rad\'o (1925) arrived at the:
\begin{proof}
Compare e.g. Siebenmann 2005 \cite{Siebenmann_2005}.
\end{proof}

\begin{lemma} (e.g., Weyl, 1913)\label{Weyl:triangulation}
A triangulated surface   is metric.
\end{lemma}

\begin{proof}
Aggregating simplices
%of the triangulation
by adjacency, the surface is expressible as a countable union
of compacta, hence Lindel\"of, so metric (Urysohn, 1925).
\end{proof}

\begin{theorem} (Rad\'o, 1925)\label{Rado:triangulation}
Any metric surface can be triangulated.
\end{theorem}

\begin{proof} The classical proof relies in principle on
Schoenflies (\ref{Schoenflies:thm-plane}), cf.
 Rad\'o 1925 \cite{Rado_1925} or
Ahlfors-Sario 1960 \cite[pp.\,105--110]{Ahlfors-Sario_1960}.
For a proof
%avoiding
circumventing Schoenflies compare Moise 1977
\cite[p.\,60]{Moise_1977}.
\end{proof}

When specialized to compact (bordered) surfaces one gets using
some combinatorial tricks the following seminal classification
theorem (initiated by pre-Morse theoretical considerations in
M\"obius 1863 \cite{Moebius_1863}, extended to the
non-orientable case in Klein's writings (partly motivated by
real algebraic curves and his paradigm of the ``Galois-Riemann
Verschmelzung'', plus apparently
 some  slight helping-hand from L. Schl\"afli). We can also
 mention  Klein's students
like Weichold 1883, and von Dyck 1888 (plus some earlier
works) . Later the combinatorial
%perspective
machine is
%strengthened
purified
%epured
in Dehn-Heegaard 1907 and
%in the US the ultimate
%polishings by
Brahana 1923.

\begin{theorem} (M\"obius 1860--63, Jordan, Klein, Dyck 1888,
Dehn-Heegaard, Brahana+Rad\'o)
\label{Moebius-Klein-classification} A compact bordered
surface is classified by
%the value of
the Euler characteristic
$\chi$, the number of contours and the indicatrix
(=orientability character). When orientable the surface is a
sphere with $g$ handles $\Sigma_g$, and otherwise it is
homeomorphic to the sphere with $g\ge 1$ cross-caps, denoted
$N_g=S^2_{gc}$.
\end{theorem}

\begin{proof} It is probably fair to qualify all early proofs
as semi-intuitive
%in the sense that
inasmuch as they required some geometric `structuration' lying
%somewhat
beyond the naked topological manifolds (those were
%in full
%generality
perhaps first defined in print in \Kerekjarto{
}1923~\cite{Kerekjarto_1923}, though the idea
%was of course in
%the air much longer ago,
is much older, e.g., Riemann 1. March\footnote{Compare, e.g.,
Speiser 1927 \cite[p.\,107-8]{Speiser_1927}} 1853--1854,
Betti, Poincar\'e 1895, Tietze 1907, Brouwer, Weyl 1913,
etc.). Thus
%formally
one first triangulates the surface with Rad\'o
(\ref{Rado:triangulation}) and then apply the combinatorial
reduction to a normal form \`a la Dehn-Heegaard, say or
alternatively do M\"obius-Morse theory. For a modern book
form, cf. e.g., Massey 1967 \cite{Massey_1967}.
\end{proof}

We now list several consequences, starting with the following,
%(which was
historically perhaps first proved via the uniformization
theorem (Klein, Poincar\'e, Koebe 1882--1907). Recall also
%(to avoid a fatal vicious circle!)
an alternative proof via triangulations and the combinatorial
%scheme
device of van der Waerden-Reichardt (ref. as in
\cite{GaGa2010} (arXiv version)):

\begin{prop}\label{uniformization} A metric simply-connected surface is either $S^2$
or the plane.
\end{prop}

\begin{proof} Cf. also  Ahlfors-Sario \cite{Ahlfors-Sario_1960}
and Massey \cite{Massey_1967}
%(outline
(in the exercise).
\end{proof}

This in turn implies first the metric-case of the following:

\begin{lemma}[Homotopic Schoenflies]
%(Schoenflies-
%(Baer 1928-Cannon 1969)
(Baer 1928, Cannon 1969) \label{Schoenflies-Baer} A
\hbox{null-}\\homotopic Jordan curve in a surface (metric or
not) bounds a disc. In particular any Jordan curve in a
simply-connected surface bounds a disc, which is unique
whenever the surface is open (equivalently not the sphere).
\end{lemma}

\begin{proof}
Via passage to the universal covering (still metric by
Poincar\'e-Volterra and the countability of the $\pi_1$
ensured by Rad\'o's triangulation (\ref{Rado:triangulation}),
or alternatively just lift the triangulation and use Weyl
(\ref{Weyl:triangulation})), we may apply in view of
(\ref{uniformization}) the classic Schoenflies theorem
(\ref{Schoenflies:thm-plane}). An argument of R. Baer, 1928
(compare e.g., \cite{GaGa2010}), shows that the bounding disc
for the lifted Jordan curve is homeomorphically projected down
in the original surface.
%as well as the first hand
%sources quoted therein esp. Epstein 1966, Cannon, 1969

The non-metric case reduces
%trivially
to the metric one, by covering the range of a null-homotopy by
a Lindel\"of subregion (as observed in Cannon 1969
\cite{Cannon_1969}).
\end{proof}

\subsection{Other gadgets: freeness of $\pi_1$
(Ahlfors-Sario) and Whitehead's spine}

%Yet to be written or killed. Yet the main difficulty is to
%find a ref. for the spine (a discussion is perhaps in Massey
%\cite{Massey_1967})

\begin{lemma}\label{Ahlfors-Sario:freeness:lemma} The
fundamental group of an open metric surface is free on
countably many generators.
\end{lemma}

\begin{proof} Cf.
Ahlfors-Sario 1960 \cite[\S 44A., p.\,102]{Ahlfors-Sario_1960}
or Massey
%%%'s book
1967 \cite{Massey_1967}.
\end{proof}

Using Whitehead's spine we get the stronger assertion:

\begin{lemma} \label{Whitehead-spine}
Any open metric surface retracts by deformation onto a
subgraph of the $1$-skeleton of any of its triangulation. In
particular it is homotopy equivalent to a (countable) graph.
\end{lemma}

\begin{proof} The theory of the spine originates in Whitehead
1939, cf. also Massey's book 1967 \cite{Massey_1967} for a
discussion.
\end{proof}

\subsection{Indicatrix and orientability (Gauss, Listing,
M\"obius, Klein, Schl\"afli, etc.)}

Those classical notions (originating with the discovery
%($\approx
(circa 1860) of the {\it M\"obius band} involving a
%well-known
well-documented($\pm$)
%minor
question of priority between
%two direct friends, viz.
%well-known
close colleagues, namely Gauss and Listing) is clearly
independent of a metric and makes sense for all manifolds.
Several viewpoints are possible (combinatorial vs. naked
TOP-manifolds). A first ``naked'' aspect is to define the {\it
indicatrix} (or even better the {\it orientation covering}):

\begin{lemma}\label{indicatrix-orient-covering} Given a manifold $M$, one can propagate ``local
orientations'' around loops to obtain a morphism $\pi_1(M)\to
\{\pm 1\}$ (called the indicatrix). The latter is in fact just
the monodromy of $M_{\circlearrowleft}\to M$ the
%$2$-fold
double orientation
%covering of $M$
cover given by de-doubling points by their two possible local
orientations.
%about the point.
Being purely local, the construction  works for all locally
Euclidean spaces even without Hausdorff proviso, and being
%purely
perfectly intrinsic it
%prompts
has the following:

(Naturality) If $L\subset M$ is a subregion of the manifold
$M$, then its orientation covering $L_{\circlearrowleft}\to L$
is just the restriction of that of $M$ to $L$.
\end{lemma}

\begin{proof} It boils down to define ``local orientations'',
cf. e.g. Dold's Algebraic Topology.
\end{proof}

A manifold is said to be {\it orientable} if its indicatrix is
trivial (equivalently, if its orientation covering is
trivial).
%Thus clearly:

\begin{lemma}\label{orientability:heredity}
$\bullet$ (Heredity) Any subregion of an orientable manifold
is orientable.

\noindent $\bullet$ (Transfer) A manifold all of whose
Lindel\"of subregions are orientable is orientable.
\end{lemma}

\begin{proof} The hereditary claim reduces to the fact
that  triviality of a covering is preserved by restricting to
subregions of the base. The transfer claim requires a little
argument. If not orientable the given manifold, $M$, has a
non-trivial orientation-covering, which is therefore
connected. Thus there is a path in $M_{\circlearrowleft}$
connecting the two points lying above the (arbitrarily) fixed
basepoint of $M$. This path, being compact,  is contained in
some Lindel\"of subregion, which up to taking a Lindel\"of
exhaustion of the base $M$ can be assumed to be the inverse
image of a Lindel\"of subregion $L$ of $M$.
%But then
By naturality (\ref{indicatrix-orient-covering}) $L$ is
non-orientable, violating the assumption.
\end{proof}

Another common definition of orientability (of a manifold of
any dimensionality) is that any embedded circle has a trivial
tubular neighbourhood. Yet we probably want to exclude wild
knots. Several respectable theories (PL, DIFF, etc.) explain
how to tame
%the
wildness, yet as we are primarily concerned with the 2D-case
there is an intrinsic weapon namely Schoenflies
(\ref{Schoenflies:thm-plane}) and a resulting tubular
neighbourhood theory (cf. e.g., Siebenmann 2005
\cite{Siebenmann_2005}) permitting to circumvent any
specialisation to such
%ancillary tamifying
structures (whose existence is rather
%%questionable
weak in the non-metric context and as we know even for compact
manifolds not universally available as soon as the dimension
is $\ge 4$).

\begin{lemma}\label{orientability-in-terms-of-Jordan-curves}
A surface is orientable iff any Jordan curve has a trivial
tubular neighbourhood. In particular puncturing finitely many
points in a surface does not affect the indicatrix.
\end{lemma}

\begin{proof} $[\Rightarrow]$ Let $J$ be a Jordan curve in the
surface $M$, and let $T$ be its tubular \nbhd, which is an
${\Bbb R}$-bundle over the circle $S^1$. By
%the elementary
classical bundle theory there is only two such bundles: the
trivial one and a twisted one (the open M\"obius band). The
latter option is precluded by heredity
(\ref{orientability:heredity}).

$[\Leftarrow]$ The converse looks more tricky, and we are only
able to perform a reduction to the metric case. Let $L$ be a
Lindel\"of subregion of $M$. Then clearly the assumption of
triviality of Jordan \nbhd{s} holds in $L$ as well, thus by
the metric case of the lemma, $L$ is orientable. By transfer
(\ref{orientability:heredity}) it follows that $M$ is
orientable.

{\it Metric case (outline).} The proof in the metric case
works maybe as follows: fix a triangulation and subdivide
barycentrically until all 2-simplexes lye in charts. Then
local orientations takes a more down-to-earth interpretation
as the borders of those simplices. Now the Jordan triviality
assumption specialized to combinatorial loops ensure that
there is a coherent way to orient simplices in the
combinatorial sense, implying the topological sense of
(\ref{indicatrix-orient-covering}). (Exercise: find a
reference where this is properly done, e.g. M\"obius 1865,
Weyl 1913, etc.)

For the last clause, just observe that if the original surface
is non-orientable, then it contains a M\"obius band and the
punctures can be performed outside of it.
\end{proof}

\subsection{Dichotomy (Leibniz,
%Kaestner,
K\"astner, Bolzano, Jordan, Veblen)}

After a long series of precursors (and successors), Jordan
(1887) showed that any embedded circle in the plane
disconnects the plane in two components. This motivates the
following jargon (borrowed from O. H\'ajek \cite{Hajek_1968}):

\begin{defn} {\rm A surface is {\it dichotomic} if any Jordan
curve (=embedded circle) divides the surface.
%In the metric
%case,
More common
%also referred to as
synonyms are {\it planar} (or {\it schlichtartig}), yet both
%terminologies
sound too restrictive when it comes to allow non-metric
surfaces.}
\end{defn}

Using homology and  the five lemma, one can show (cf.
\cite[5.3, 5.4]{GabGa_2011}):

\begin{lemma}\label{dicho:hered-transfer}

$\bullet$ (Heredity) Any subregion of a dichotomic surface is
dichotomic.

\noindent $\bullet$ (Transfer) A surface all of whose
Lindel\"of subregions are dichotomic is dichotomic.
\end{lemma}

\begin{lemma}\label{dicho-implies-orientable}
A dichotomic surface is orientable.
\end{lemma}

\begin{proof} In view of the orientability criterion
%in terms of
%Jordan curves
(\ref{orientability-in-terms-of-Jordan-curves}),
let $J$ be a Jordan curve in the surface $M$, and let $T$ be a
tube around it. By heredity of dichotomy
(\ref{dicho:hered-transfer}) the latter is dichotomic, hence
cannot be the M\"obius band.
\end{proof}

%%%%%%OLD PROOF
\iffalse
\begin{proof} By definition, orientability of a manifold
(any dimensionality) means that any embedded circle has a
trivial tubular neighbourhood. So let $J$ be a Jordan curve in
$M^2$, and $T$ be its tubular neighbourhood, which is an
${\Bbb R}$-bundle over $J$. Since any (open) subregion of a
dichotomic surface is dichotomic
%%(cf. \cite{GabGa_2011})
(\ref{dicho:hered-transfer}), the tube $T$ is dichotomic.
%, hence divided by $J$
Hence $T$ cannot be the unique twisted bundle (M\"obius band)
and so is the trivial bundle.
\end{proof}
\fi

%Saxo

%Then notice the:
While the converse of (\ref{dicho-implies-orientable}) is not
true (e.g., torus),
%yet becomes so
it is sometimes:

\begin{lemma}\label{dichotomy:orient_plus_inf_cycl}
An orientable surface with infinite cyclic group is
dichotomic.
\end{lemma}

%We differ the proof to later.

\def\Si{\Sigma}

\begin{proof}
%[Proof of Lemma~\ref{dichotomy:orient_plus_inf_cycl}]
(The following argument is homological, so rather algebraic;
for a more
%another
geometric proof using Schoenflies, see
Remark~\ref{new-proof}.) Let $J$ be a Jordan curve in the
surface $\Si$. We can assume that $J$ is not null-homotopic,
since otherwise Jordan separation is obvious as $J$ bounds a
disc (\ref{Schoenflies-Baer}). We fix $T$ a tubular
neighbourhood of $J$, which is trivial, i.e. $T\approx {\Bbb
S}^1\times {\Bbb R}$ (since $\Si$ is orientable).
%Since we want
To show that $\Si-J$ is disconnected
%it is natural to
we examine the homology exact sequence of the pair $(\Si,
\Si-J)$ written down as the third line of the diagram below.
Just above it we have the sequence of the tube pair $(T,T-J)$,
which we embed as the complement of the poles of the
$2$-sphere (denoted $S$) while mapping $J$ to the equator.
This gives us the first line which is the sequence of the pair
$(S,S-J)$. By naturality all squares are commutative.

\iffalse
\def\ziehen{\hskip-10pt}
$$
\begin{matrix}
&H_2(T)\ziehen     &\rightarrow\ziehen &H_2(T,T-J)\ziehen
&\rightarrow\ziehen& H_1(T-J)\ziehen   &\rightarrow\ziehen
&H_1(T)\ziehen     &\rightarrow\ziehen &H_1(T,T-J)\ziehen
&\rightarrow\ziehen& H_0(T-J)\ziehen   &\rightarrow\ziehen
&H_0(T)\ziehen     &\rightarrow\ziehen & H_0(T,T-J)=0  \cr

&\downarrow & &\downarrow & &\downarrow& &\downarrow&
&\downarrow & &\downarrow & &\downarrow & &\downarrow \cr

&H_2(\Si)\ziehen    &\rightarrow\ziehen &H_2(\Si,\Si-J)\ziehen
&\rightarrow\ziehen & H_1(\Si-J)\ziehen &\rightarrow\ziehen
&H_1(\Si)\ziehen    &\rightarrow\ziehen &H_1(\Si,\Si-J)\ziehen
&\rightarrow\ziehen & H_0(\Si-J)\ziehen &\rightarrow\ziehen
&H_0(\Si)\ziehen    &\rightarrow\ziehen & H_0(\Si,\Si-J)=0
\end{matrix}
$$
\fi

\begin{figure}[h]
%\centering
    \hskip-22pt\epsfig{figure=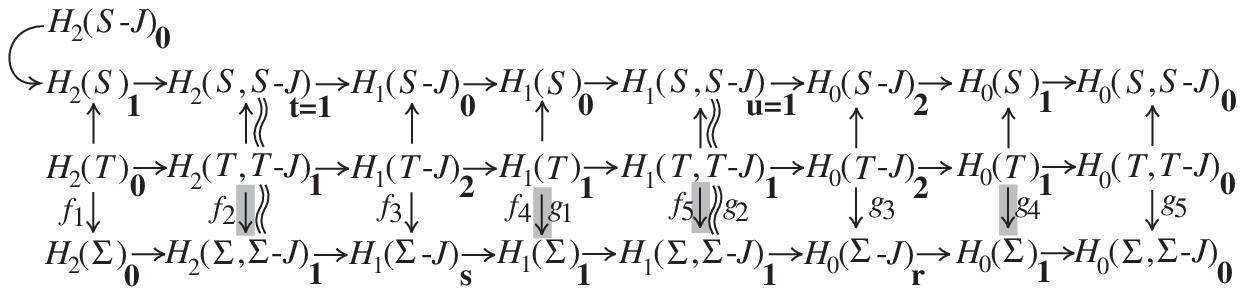,width=135mm}
  %\caption{\label{diagr}
  %Artist view of
  %Jordan separation in a simply-connected surface via a
  %Schoenflies trapping}
\vskip-5pt\penalty0
\end{figure}
%Further we have
\noindent The excision isomorphisms are denoted by
%the
vertical equivalence symbols. Boldface ``{\bf 0}''
%zero
symbols indicate trivial groups, while other bold indices
indicate the rank of the corresponding abelian group. Looking
at the first line we find that $t=1$ and $u=1$, which values
propagates downstairs by the excision isomorphisms. Next the
group $H_2(\Si)=0$ is trivial, because $\Si$ is an open
$2$-manifold and the postulated fundamental group ${\Bbb Z}$
does not occur among the list of closed surfaces. (We used
implicitly the vanishing of the top-dimensional homology $H_n$
of Hausdorff $n$-manifolds, compare e.g. Samelson 1965
\cite{Samelson_1965-homology}.) Thus  by exactness of the
bottom line, we have $1-s+1-1+r-1=0$, i.e. $r=s$, provided all
ranks are finite. For this we apply the {\it five lemma}
saying that if the diagram of abelian groups has exact rows
and each square is commutative:
\def\ziehen{\hskip-6pt}
$$
\begin{matrix}
&C_1\ziehen &\rightarrow\ziehen&C_2\ziehen
&\rightarrow\ziehen& C_3\ziehen &\rightarrow\ziehen&
C_4\ziehen &\rightarrow\ziehen& C_5 \cr

&\quad\;\;\downarrow  f_1 &&\quad\;\;\downarrow  f_2 &
&\quad\;\;\downarrow f_3 & &\quad\;\;\downarrow f_4 &
&\quad\downarrow f_5 \cr

&D_1\ziehen &\rightarrow\ziehen &D_2\ziehen
&\rightarrow\ziehen & D_3\ziehen &\rightarrow\ziehen&
D_4\ziehen &
 \rightarrow\ziehen& D_5
\end{matrix}
$$
\vskip-10pt \noindent Then

(1) if $f_2$ and $f_4$ are onto and $f_5$ injective, then
$f_3$ is onto.

(2) if $f_2$ and $f_4$ are injective and $f_1$ is onto, then
$f_3$ is injective.

%%%Part (1) applied to our $f_i$ shows that $f_3$ is onto, so
%%%%$s\le 2$.
\noindent Part (1) does not apply to the $f_i$ (we do not know
$f_4$ to be onto), but it applies to the $g_i$ showing that
$g_3$ is onto, so $r\le 2$. Since the group indexed by $s$ is
squeezed in an exact sequence with zeros extremities, it has
finite rank as well.  Now part (2)
%of the five lemma
applies to the $f_i$ (but not to the $g_i$!), thus $f_3$ is
injective and $s\ge 2$,
%%%%%(an equality indeed),
so $r\ge 2$ as we knew $r=s$. This completes the proof.
%%%
\iffalse
\begin{rem}
 {\rm In fact above we used that $f_4$ is onto since $J$
 is not null-homotopic, to get the finiteness of $s$. However}
\end{rem}
\fi
\end{proof}

%%%%%CENSURED
\iffalse
\begin{rem}
 {\rm Maybe pushing the argument forward, one can show that
 if $J$ is not null-homotopic then it has to represent the
 generator of $H_1(\Si)$. This amounts to show that $f_4$ is onto.
 This follows if $f_3$ is onto and $g_3$ injective. The latter
 follows as $g_3$ has free abelian source and range. But the former
 ontoness of $f_3$ looks not obvious... }
\end{rem}
\fi

%Saxoo

\subsection{Riemann's branched coverings}

The following
%well-known
mechanism originating in complex
function theory (Riemann's Thesis 1851) will later find a
pleasant application to foliated structures:

\begin{lemma}\label{Riemann:branched-cover}
Given a finite $d$-sheeted covering $\Sigma\to F_{n*}$ of a
punctured surface $F$, there is a canonical recipe to fill
over the punctures to deduce a branched covering $\Sigma^{*}
\to F$ whose total space is a surface. Moreover if the
(unpunctured) surface $F$ is compact, then so is $\Sigma^*$,
and their Euler characteristics are related by the so-called
Riemann-Hurwitz formula:
\begin{equation}\label{Riemann-Hurwitz:gnal-case}
\chi(\Sigma^*)=d\chi(S^2)-\deg(R),
\end{equation}
where $\deg(R)$ is the ramification counted with multiplicity.
Further orientability of $F$ transfers to $\Sigma^{*}$.
\end{lemma}

\begin{proof} If we look at a ``pierced neighbourhood''
$U$ of a puncture $p\in F-F_{n*}$ topologically like ${\Bbb
C}^*$ (punctured complex plane) we obtain a covering
$p^{-1}(U)\to U$. Since $\pi_1(U)$ is ${\Bbb Z}$, the
coverings of $U$ are completely classified,
%by the
%list of
being the mappings $z\mapsto z^k$ (from ${\Bbb C}^*$ to
itself) for some integer $k\ge 1$. So there is a natural way
to fill over the punctures (Riemann's trick) to obtain
$\Sigma^*\to F$ a branched covering of degree $d$ whose total
space  $\Sigma^{*}$ is a surface.

The Riemann-Hurwitz formula follows by a Euler characteristic
counting. Triangulate $F$
%(assumed metric)
so that punctures are vertices, and lift simplices to
$\Sigma^{*}$ and count the alternating sum of those, which
behaves multiplicatively up to the correction effected by
ramification.

The assertion regarding orientability can be checked
combinatorially, or by noticing that puncturing does not
affect orientability. Hence $F$ orientable implies $F_{n*}$
orientable (\ref{orientability:heredity}), and in turn the
covering $\Sigma$ is orientable, and finally $\Sigma^*$ is
orientable. (Little exercises.)
\end{proof}

\subsection{Dichotomic coverings (via branched coverings)}

The following specialization of (\ref{Riemann:branched-cover})
will
%later
be useful
%in understanding
for the sharpness of Dubois-Violette's labyrinths (i.e., 3 is
the minimal number of punctures required in the plane to
construct a transitive foliation):

\begin{lemma}\label{dicho-covering-is-dicho}
The total space of a
%%%$2$-fold
double covering $p\colon \Sigma\to M$ of a dichotomic surface
$M$ with $\pi_1(M)=F_2$ is itself dichotomic.
\end{lemma}

\begin{rem}{\rm The result is sharp as
%exemplified
shown by the standard branched covering $T^2\to S^2$ ramified
at $4$ points: divide the torus by the (holomorphic)
involution $z\mapsto-z$, or, rotate by $180^0$ a Euclidean
model of the torus in revolution.}
\end{rem}

\begin{proof} We
%shall
first establish the metric case and then
%explain how to
boost the result beyond the metrical barrier via the
%standard
usual exhaustion method. The metric case
 involves the trick of
 %Riemann
 branched coverings (\ref{Riemann:branched-cover}).
%familiar in Riemann surfaces theory (which is pleasant
%to see at work in the foliated context, as we shall see later).

{\bf Metric case.} By (a special case (\ref{Kerekjarto:dicho})
of) \Kerekjarto's classification (\ref{Kerkjarto:end}), $M$ is
homeomorphic to
%$S^2-\{3pt\}$
$S^2_{3*}$ (sphere with $3$ punctures), and we compactify $M$
to the sphere $S^2$ by adding 3 points.
%(compare
%(\ref{Kerekjarto:dicho}) below for more details).
%
\iffalse
If we look at a ``neighbourhood'' $U$ of a puncture
$p\in S^2-M$ topologically like ${\Bbb C}^*$ (punctured
complex plane) we obtain a covering $p^{-1}(U)\to U$. Since
the $\pi_1(U)$ is ${\Bbb Z}$ its coverings are completely
classified,
%by the
%list of
being the mappings $z\mapsto z^k$ (from ${\Bbb C}^*$ to
itself) for some integer $k\ge 1$. So there is a natural way
to \fi By filling over the punctures (Riemann's trick) we
obtain $\Sigma^*\to S^2$ a branched covering of degree 2. The
space $\Sigma^{*}$ is a surface which is compact, borderless
and orientable ($M$ being dichotomic, hence orientable
(\ref{dicho-implies-orientable}), thus so
%is $M$ as well as
is $\Sigma^{*}$). By Riemann-Hurwitz
%(Euler characteristic
%counting)
we have
\begin{equation}\label{Riemann-Hurwitz}
\chi(\Sigma^*)=2\chi(S^2)-\deg(R),
\end{equation}
where $\deg(R)$ is the ramification counted with multiplicity.
Since the degree of the map is  2 there is only simple
ramification so that $\deg(R)$ is just the cardinality of the
branched points. In our situation, $\deg(R)\le 3$ and since
$\chi(\Sigma^*)=2-2g$ where $g$ is the genus, we deduce that
$g=0$. By classification (\ref{Moebius-Klein-classification})
$\Sigma^*$ is the sphere, which is dichotomic by the Jordan
curve theorem. Thus $\Sigma$ is dichotomic as well by heredity
(\ref{dicho:hered-transfer}).
%q.e.d.

{\bf Non-metric case.} We choose an exhaustion
$M=\bigcup_{\alpha<\omega_1} M_{\alpha}$ by Lindel\"of
subregions $M_{\alpha}$ and we may arrange
%that
$\pi_1(M_{\alpha}) \approx F_2$. Such a ``calibration'' of the
fundamental group is
%explained
justified in
%the
%technical
Lemma~\ref{calibrated-exhaustions:lemma} below. Since $M$ is
dichotomic, the $M_\alpha$ are also dichotomic
%(for this
%heredity cf. \cite[5.3]{GabGa_2011})
(\ref{dicho:hered-transfer}). Thus $\Sigma_{\alpha}
:=p^{-1}(M_{\alpha})\to M_{\alpha}$ is dichotomic as well by
the metric case. \iffalse Now recall from
\cite[5.4]{GabGa_2011} that dichotomy is
%super
anti-hereditary  in the sense that a surface all of whose
Lindel\"of subregions are dichotomic is dichotomic. \fi Now
given $L$ a Lindel\"of subregion of $\Sigma$, there is some
$\alpha$ such that $L\subset \Sigma_{\alpha}$. By heredity $L$
is dichotomic, and the
%Lindel\"ofian
transfer~(\ref{dicho:hered-transfer}) completes the proof.
\end{proof}

\begin{rem}\label{new-proof} {\rm This argument reproves
(\ref{dichotomy:orient_plus_inf_cycl}), i.e. {\it dichotomy of
orientable surfaces
%$M$
with infinite cyclic group}. Let us carry out this simple
exercise.

\smallskip
\begin{proof} [Another proof of~\ref{dichotomy:orient_plus_inf_cycl}] By
(\ref{calibrated-exhaustions:lemma}) we have an exhaustion
$M=\bigcup_{\alpha<\omega_1} M_{\alpha}$ by Lindel\"of (hence
metric) subregions with $\pi_1(M_\alpha)\approx{\Bbb Z}$.
Since $M$ is orientable, so are the $M_\alpha$ which are
therefore open cylinders (again by an appropriate special case
of \Kerekjarto~(\ref{Kerkjarto:end})), hence in particular
dichotomic. By
%super-heredity
transfer
%\cite[5.4]{GabGa_2011}
(\ref{dicho:hered-transfer}) it is enough to show that any
Lindel\"of subregion $L$ of $M$ is dichotomic, and so is the
case by heredity
%\cite[5.3]{GabGa_2011}
(\ref{dicho:hered-transfer}) because $L$ is contained in some
$M_\alpha$, which is dichotomic.
\end{proof}
}
\end{rem}

%Exhaustion with calibrated
\subsection{Calibrating the fundamental group (Cannon)}

We now check the pivotal lemma about exhaustions respecting
the fundamental group (which is yet another consequence of
Schoenflies going back to R.\,J. Cannon 1969
\cite[p.\,98]{Cannon_1969}, ``fill in the holes'' argument).
First we show a kernel killing procedure.
%Beware that
{\it Warning:} our clumsy(?) proof uses beside Schoenflies,
some other gadgets like the freeness of the fundamental group
of open metric surfaces ({\ref{Ahlfors-Sario:freeness:lemma}})
%(cf. e.g., Ahlfors-Sario
%\cite{Ahlfors-Sario_1960}),
plus the stronger theory of J.\,H.\,C. Whitehead's spine
(\ref{Whitehead-spine}) telling that such surfaces retract by
deformation to a countable graph---referred to as `the' {\it
spine}.

\begin{lemma}\label{killing:kernel} (Kernel killing procedure)
Given  a Lindel\"of subregion $L$ in a surface $M$ so that the
natural map $\pi_1(L)\to \pi_1(M)$ is epimorphic,  there is a
larger Lindel\"of subregion $L'\supset L$ such that
$\pi_1(L')\to \pi_1(M)$ is isomorphic.
\end{lemma}

\begin{proof}
If the natural morphism $j\colon\pi_1(L)\to \pi_1(M)$ is not
injective, then for any element of the kernel we have a
shrinking homotopy whose compact range may be covered by
finitely many charts which aggregated to $L$ gives some
$L^{*}$. Since $L$ is metric, its group $\pi_1(L)$ is
countable, and we need only iterate countably many times the
procedure, thereby conserving Lindel\"ofness for the enlarged
$L^{*}$. It may seem that $\pi_1(L^{*})\to\pi_1(M)$ is now
isomorphic. However when killing an element of the kernel may
well accidentally create a
%new
parasite ``handle'' or ``connectivity'', jeopardizing the
%whole process.
desideratum.
%So the
%game looks like an endless cowboys
%mutinery
%fighting and it is not clear that the process terminates.
%
\iffalse Maybe the difficulty is the same as the one
encountered by Mathieu in the bridge paper and which he could
ultimately solve thanks to the Schoenflies technology. \fi
%%%%%%%%%
The trick is to
%using
take advantage of some geometric topology \`a la Schoenflies,
to kill (or better plumb) the holes in a
%clever
surgical way, without generating new ones by inadvertence.
%%%%%%Imagine
Suppose first that
%any
an element in $\ker j$ is represented by a Jordan curve, which
being null-homotopic in $M$ bounds a disc in $M$
(\ref{Schoenflies-Baer}), which aggregated to $L$
%that disc which
kill one holes without creating new ones.
%This looks fine
%except that it is not true that
Unfortunately, not all element of the $\pi_1$ of a surface are
representable by  Jordan curves (e.g., the
%double of the
generator squared
%of
in the group of a punctured plane is not
Jordan-representable). By carefully selecting who to kill in
the kernel, namely the primitive elements
%and not their multiples
(yet not their proper powers),
% in which
%case there is an hope to find
as the former admit Jordan representants (cf.
(\ref{Jordan-representant}) below) completes the
%kernel
%killing
procedure. As countably many discs are aggregated
Lindel\"ofness is preserved.
%Observe
Also
%that
killing the primitive elements of the kernel
suffices to kill the whole kernel. Indeed the latter is a
subgroup of $\pi_1(L)$ which is known to be free when $L$ is
an open surface, hence free as well and therefore generated by
its primitive elements. Of course assuming $L$ open is not
expensive, since otherwise $L$ is compact, hence clopen, so
identic to $M$ and the lemma is trivially true.
\end{proof}

\begin{lemma}\label{calibrated-exhaustions:lemma}
A surface $M$ with
%free fundamental group
%$F_r$ of finite rank $r$,
finitely (or countably) generated fundamental group has an
exhaustion  by Lindel\"of
%open sets
subregions $M_\alpha$ such that the morphisms
$\pi_1(M_{\alpha})\to \pi_1(M)$ induced-by-inclusions are
isomorphic for all $\alpha$.
\end{lemma}

\begin{proof}
Choose
%$a_1, \dots ,a_r$
a finite (or countable)
%basis
generating system of $\pi_1(M)$, and representing loops $c_i$.
%with $c_i\in a_i$.
Each $c_i\colon [0,1] \to M$ is a continuous map with
$c_i(0)=c_i(1)= \star$ the basepoint of $M$. Cover randomly
the range of the $c_i$ by charts to get a Lindel\"of
%connected
%open set
subregion $L_0$. By construction $\pi_1(L_0)\to \pi_1(M)$ is
epimorphic, and by kernel killing (\ref{killing:kernel}) we
find $M_0:=L_0'$ with the required properties of
Lindel\"ofness and incompressibility. Then aggregate randomly
countably many new charts to get $L_1\supset M_0$ and again
%by
kernel killing $\pi_1(L_1)\to \pi_1(M)$ gives $M_1$.
%that remains to be clarified.
Transfinite induction
%this
completes the proof by defining $M_{\lambda}$ to be a kernel
killing enlargement of $L_{\lambda}=\bigcup_{\alpha<\lambda}
M_{\alpha}$ whenever $\lambda$ is a limit ordinal.
\end{proof}

%Hence to
%%complete
%finish the above proof we need, alas, another lemma:

\begin{lemma}\label{Jordan-representant} In the
fundamental group of
an (open) metric surface $M$ any primitive element
%(which is
(i.e., not a proper power)
%can be represented
is representable by a Jordan curve.
\end{lemma}

\begin{proof}
Our argument is not very intrinsic relying on
%%some
combinatorial methods. (Is there  an argument via the
universal covering?)
%%%%%%RADO NOT NEEDED
\iffalse With Rad\'o (1925) (\ref{Rado:triangulation}), first
triangulate the surface. \iffalse (Rad\'o 1925
\cite{Rado_1925}, also Ahlfors-Sario
\cite[p.\,105--110]{Ahlfors-Sario_1960}, or Moise
\cite[p.\,60]{Moise_1977}). \fi\fi
%Then by simplicial
%approximation find a simplicial representant of the homotopy
%class of the given loop. Finally
Via Whitehead's
%%%theory of the
spine
(\ref{Whitehead-spine}),
%to show that
$M$ retracts by deformation $M\to \Gamma$ to a countable graph
$\Gamma$. By the primitivity assumption, the loop pushed in
the spine
%graph
is homotopic to a simple loop (imagine an edge in the bouquet
of circles resulting by collapse of a maximal tree), hence
representable by a Jordan curve in the graph, so a fortiori in
the surface $M$.
%
\iffalse Alternatively argue with the universal covering which
is the plane. The given element $\gamma\in \pi_1(M)$ induces a
deck-translation of $\widetilde{M}$. \fi
%
%
%
%
\end{proof}

\subsection{Puncturing and cross-capping (Cro-Magnon, von Dyck)}

This section gives  algebraic arguments
%corroborating
for two intuitively obvious issues:
%(details can safely be
%avoided
% without lost of understanding):

\begin{lemma}\label{puncturing}
Puncturing an open surface adds one free generator to
%%its
the fundamental group.
\end{lemma}

\begin{proof}
%[Proof of \ref{puncturing}]
If $S$ is open, any puncture increases by one the rank of the
$H_1$. This follows e.g. by writing the exact sequences of the
pairs $(S,S_{*}=S-1pt)$ and $(U,U_{*})$, where $U$ is a chart
containing the puncture:
%
\iffalse {\small
$$
H_2(S)\to H_2(S,S_{*})\to H_1(S_*) \to H_1(S)\to
H_1(S,S_{*})\to H_0(S_*) \to H_0(S)\to H_0(S,S_{*})=0
$$}\fi
%
%
\begin{figure}[h]
%\centering
    \hskip-5pt\epsfig{figure=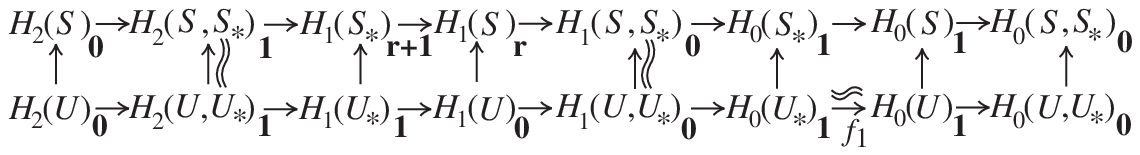,width=122mm}
  %\caption{\label{diagr}
  %Artist view of
  %Jordan separation in a simply-connected surface via a
  %Schoenflies trapping}
\vskip-10pt\penalty0
\end{figure}

\noindent Hence if $\pi_1(S)$ is free of rank $r$, then
$\pi_1(S_*)$ being free by
({\ref{freeness-for-open-surfaces:prop}}) is free of rank
$r+1$. In particular if $S$ is $1$-connected and open, then
$\pi_1(S-k\, pts)$ is $F_k$, free of rank $k$.
%
%
%%%%%%CENSURED
\iffalse {\small
%
One could hope that Seifert-van Kampen give a more direct
proof of the assertion: \iffalse that a single puncture in an
open surface increases the rank of the $\pi_1$ by one unit.
More precisely\fi   if $S$ is an open surface with
$\pi_1(S)=F_r$, then the once punctured surface $S_*$ has
$\pi_1(S_{*})=F_{r+1}$. Choose $U$ a chart about the puncture;
then $S=S_*\cup U$ and $S_*\cap U=U_*$.
%Hence
By Seifert-van Kampen $\pi_1(S)\approx
\pi_1(S_*)\ast_{\pi_1(U_*)} \pi_1(U)$ (amalgamated product),
from where the assertion
%is
looks more-or-less clear. However
%that
this argument does not capture the assumption that $S$ is
open. This explains why we opted for the first homological
argument.
%
}\fi
\end{proof}

Likewise cross-capping has the same impact on the fundamental
group:
%as puncturing:

\begin{lemma}\label{cross-capping} Cross-capping an
open surface adds one free generator to
%%%its
the fundamental group.
\end{lemma}

\begin{proof} Recall
the cross-capping operation (von Dyck, 1888) of a surface $M$
amounting to identify diametrically opposite points at the
border of a embedded compact disc $D\subset M$. Denote $M_c$
the cross-capped surface. Choose $U$ a \nbhd{ }of the
cross-cap, which is homeomorphic to an (open) M\"obius band.
We have $M_c=U \cup M_{\star}$, where $M_{\star}=M-D$. The
Mayer-Vietoris sequence: {\small
$$
\underbrace{H_2(M_c)}_{0}\to \underbrace{H_1(U \cap
M_{\star})}_{\Bbb Z}\to H_1(U)\oplus H_1( M_{\star})\to
H_1(M_c) \to H_0(U \cap M_{\star}) \to H_0(U) \oplus H_0(
M_{\star}),
$$}
\!\!whose last arrow is injective, truncates as $H_1(M_c)\to
0$. The first group is trivial since $M_c$ is open. It follows
that the rank of $H_1(M_c)$ equals that of $H_1(M_{\star})$,
which is one more than that
%the one
of $H_1(M)$ (puncturing a compact disc amounts to puncture a
point and use (\ref{puncturing})). The claim now follows by
the freeness of the $\pi_1$
(\ref{freeness-for-open-surfaces:prop}).
\end{proof}

\subsection{Deleting a closed long ray
(indicatrix and $\pi_1$-invariance)}

The {\it closed long ray} ${\Bbb L}_{\ge 0}$ is the unique
bordered non-metric (Hausdorff) 1-manifold.

\begin{lemma}\label{deleting:a_closed_long_ray} Given a closed long ray $L$ embedded in a surface
$M$, the surface $M-L$ slitted along $L$ has the same $\pi_1$.
In fact the natural morphism $\pi_1(M-L)\to \pi_1(M)$ is
isomorphic. Further the orientability character (indicatrix)
of $M$ and $M-L$
%coincide.
are the same. Finally the same holds for the ends-number.
\end{lemma}

\begin{proof} If $M$ is orientable, then so is $M-L$ by (\ref{orientability:heredity}).
Conversely if $M$ is not orientable, there is a one-sided
Jordan curve $J$ in $M$
(\ref{orientability-in-terms-of-Jordan-curves}). We can find a
tube neighbourhood $T$ for a sub-arc $A\approx [0,1]$ of $L$
such that $A\supset J \cap L$. $T$ is homeomorphic to a
rectangle and we have a homeomorphism of triads $(T,L\cap T,
A)\approx ([-1,2]\times[-1,1], [0,2]\times \{0\},[0,1]\times
\{0\})$. Using a self-homeomorphism of the rectangle which is
the identity on the boundary and  pushing the arc outside
itself (and outside $L\cap T$),
%we see that
its extension to $M$ (by the identity outside $T$) yields a
``finger move''  pushing $J$ (plus its M\"obius tubular
\nbhd{} $N$) outside $L$ (enlarge $A$ if necessary so that
$A\supset N \cap L$). So the configuration $(N,J)$ is pushed
into $M-L$ showing its non-orientability.

The assertion regarding the $\pi_1$ is proved by the same
``finger move'' trick. Indeed given a loop in $M$ we may push
it into $M-L$ (noticing that the finger move homeomorphism is
isotopic to the identity). Thus the natural morphism
$\pi_1(M-L)\to \pi_1(M)$ is onto. To get its injectivity,
assume that $[c]$ is in the kernel. So $c$ is loop in $M-L$
which is null-homotopic in $M$. Since the range of the
homotopy is a compactum $K$ we can find a subarc $A$ of $L$
large enough as to contain $K\cap L$, and a finger move push
this outside $L$ and produce a null-homotopy for $c$ ranging
through $M-L$.
\end{proof}

%\section{Freeness of the fundamental group of
%very low-dimensional manifolds}

%V

\subsection{Finitely-connected surfaces, cylinder ends (\Kerekjarto)}

With
%%out worrying about
loose conventions, we could define the {\it connectivity} of a
surface as the rank of its $H_1$ with integer coefficients.
%(if you want
%(to follow the classical Riemann-Betti convention augment this
%by one unit).
This conflicts slightly with the classical Riemann-Betti
convention, where simple-connectivity really corresponds to
rank zero (not one!). So eventually  just
%speak about the {\it
%rank} as being the first Betti number.
employ:

\begin{defn} {\rm The {\it rank} of a surface (metric or not)
is its first Betti number, i.e. the rank of the first
(singular) homology group $H_1(M,{\Bbb Z})$. When finite,
%we
say the surface to be {\it of  finite-connectivity}.}
\end{defn}

The following trick of \Kerekjarto{ }1923
\cite{Kerekjarto_1923} (only a baby case of his more general
classification of all open metric surfaces)
%(maybe also in
%Richards, 1963)
is quite foundational (equivalent to the classification of
$1$-connected metric surfaces) and pivotal subsequently:
%, yet a useful
%recasting of it:

\begin{theorem} (\Kerekjarto{ }1923)\label{Kerkjarto:end}
A metric surface of finite-connectivity is homeomorphic to a
finitely-punctured closed surface. The latter closed model is
%unambiguously
uniquely defined, and consequently
%any metric
%surface of finite-connectivity
%such
finitely-connected surfaces are classified by the connectivity
(=rank of $\pi_1$), the number of ends $\varepsilon$ and the
indicatrix.
%(orientability or not).
 In particular, open metric surfaces of finite-connectivity
possess an end neighbourhood homeomorphic to a punctured
plane.
\end{theorem}

\def\ende{\varepsilon}
\def\rk{{\rm rk}}

\begin{proof} (Via the classification of $1$-connected
surfaces (\ref{uniformization}), and some dirty
tricks.---Hausdorff would say: {\it Ich mache Komplexe mit
Komplexen!}) If the surface $M$ is compact there is nothing to
prove (\ref{Moebius-Klein-classification}).
%the surface.
Otherwise,  $\pi_1(M)$ is free
(\ref{Ahlfors-Sario:freeness:lemma}), and  $H_1$ has finite
rank. Fix a finite generating system $a_1,\dots,a_r $ of
$\pi_1(M)$. By regular \nbhd{ }theory,  any compactum in a
PL-manifold is contained in a finite bordered sub-manifold.
(This goes back to Whitehead, cf. e.g. Rourke-Sanderson as
quoted in Nyikos \cite{Nyikos84}.) This applies to metric
surfaces by Rad\'o's
%(1925)'s
triangulations (\ref{Rado:triangulation}). Representing the
$a_i$ by loops $c_i$, we may cover the ranges of the $c_i$ by
a compact bordered subsurface $W\subset M$. By construction
$\varphi\colon\pi_1(W)\to \pi_1(M)$ is epimorphic. Using the
kernel killing procedure (\ref{killing:kernel}), one can
arrange $\varphi$ to be isomorphic (controlling compactness as
we only need to kill finitely many primitive elements of the
kernel). (Following Nyikos \cite{Nyikos84}, we could say that
$W$ is a bag for $M$.)

\begin{claim}\label{claim:incompressible-implies-countour_equal-ends} Let $W$ have $n$ contours (=boundary components), then
%it is
%easy to
we claim (and prove clumsily below) that $M-W$ has also $n$
components $\ende_i$, whose closures $\overline{\ende_i}$ are
non-compact bordered surfaces with one contour.
\end{claim}

\noindent{\bf Proof of Claim.} Indeed since $W$ is a bordered
surface, each contour of $\partial W$ has a collar and
therefore is two-sided in $M$. Choose the collar-sides lying
outside $W$. If two contours get connected outside $W$ in $M$,
then we can construct (by aggregating an outer connection with
an inner connection inside $W$) a loop in $M$ whose
intersection number (in homology mod 2) with both contours is
$1$, and therefore which cannot be homotoped into $W$
(violating the surjectivity of $\varphi$). Thus $M-W$ has at
least $n$ components (and of course cannot have more). Further
each residual piece $\overline{\ende_i}$ cannot be compact,
otherwise by classification
(\ref{Moebius-Klein-classification}) jointly with Seifert-van
Kampen some alteration at the $\pi_1$-level would be detected.
For instance if $\overline{\ende_i}$ is a disc, some loop in
$W$ trivializes in $M$, violating the injectivity of
$\varphi$, and if $\overline{\ende_i}$ is a complicated
pretzel with one contour, then again by looking at an
appropriate intersection number (with a fixed curve) gives a
loop in $M$ which cannot be homotoped to $W$.

\smallskip

\noindent{\bf Back to the proof of \ref{Kerkjarto:end}.} When
capping-off $\overline{\ende_i}$ by a disc we obtain the open
surface $\overline{\ende_i}_{cap}$, which punctured is
homeomorphic to $\ende_i$. Thus by (\ref{puncturing}), the
rank $\rk H_1(\ende_i)= \rk H_1(\overline{\ende_i}_{cap})+1$.
Aggregating just one piece, say $\overline{\varepsilon}$, of
the decomposition $M=W\cup (\bigcup_{i=1}^n
\overline{\ende_i}_{cap})$,  the Mayer-Vietoris sequence
{\small
$$
\underbrace{H_2(W\cup \overline{\ende})}_{0}\to
\underbrace{H_1(W \cap \overline{\ende})}_{\Bbb Z}\to
H_1(W)\oplus H_1( \overline{\ende})\to H_1(W\cup
\overline{\ende}) \to H_0(W \cap \overline{\ende}) \to H_0(W)
\oplus H_0( \overline{\ende}),
$$}
\!\!whose the last arrow is injective,  truncates as
$H_1(W\cup \overline{\ende})\to 0$. So $\rk H_1(W\cup
\overline{\ende})= \rk H_1(W) + \rk H_1(\overline{\ende})-1$.
Hence by induction, $\rk H_1(W\cup
\bigcup_{i}\overline{\ende_{i}})= \rk H_1(W) + \sum_{i=1}^n\rk
H_1(\overline{\ende_{i}})-n$.
%But
By construction $\rk H_1(W)= \rk H_1(M)$,
%hence
so $\sum_{i=1}^n\rk H_1(\overline{\ende_{i}})=n$. Since a
bordered surface with $H_1=0$ having a unique compact contour
is
%the disc
compact  (cf. \cite[Lemma~10]{GaGa2010}) it follows from
(\ref{claim:incompressible-implies-countour_equal-ends}) that
$\rk H_1(\overline{\ende_{i}})\neq 0$. Hence $\rk
H_1(\overline{\ende_{i}})=1$ for all $i$, and $\rk
H_1(\overline{\ende_i}_{cap})=0$. Thus by freeness
(\ref{Ahlfors-Sario:freeness:lemma}),
$\pi_1(\overline{\ende_i}_{cap})=0$, and by the classification
(\ref{uniformization}) it follows that
$\overline{\ende_i}_{cap}\approx{\Bbb R}^2$. It remains now
only to compactify each end $\varepsilon_i$ by adding the
point at infinity giving us the searched closed surface. This
shows the first clause.

The third (last) clause is a trivial consequence of the first.

Finally, the second clause follows from the classification of
closed surfaces (\ref{Moebius-Klein-classification}). Indeed
if $M$ has $n$ ends it is---by the first clause---homeomorphic
to
%$\Si_{g,n*}$ (the sphere with $g$ handles and $n$ punctures)
%if orientable and to $N_{g,n*}$ the sphere with $g$ cross-caps
%and $n$ punctures, if non-orientable.
$F_{n*}$, a closed model $F$ affected by $n$ punctures.
 Comparing the
characteristic of $M$ with that of ``its'' closed model $F$,
we have the following; e.g., remove from $F$ small discs about
the $n$ punctures to get a bordered surface $W$, to which $M$
retracts by deformation (hence $\chi(M)=\chi(W)$), and which
has lost $n$ $2$-simplices
%%%less than
w.r.t. $F$ (hence $\chi(W)=\chi(F)-n$):
\begin{equation}
\chi(M)=1-b_1+b_2=\chi(F)-n.
%=\begin{cases} 2-2g-n  &
%\text{(orientable case)}\cr
% 2-g-n & \text{(non-orientable case)}
% \end{cases}
\end{equation}
%%%%%%%%%%%%%%%%%%COMMENT SAXO (=WRITER)
%%%%%%%This looks obvious combinatorially,
%%%%%%%yet it can be proved by hand by induction
%%%%%%%via puncturings and starting
%%%%%%%with a single puncture and distinguishing
%%%%%%%the orientable case (4g-gon model)
%Since $M$ is open $b_2=0$ and $\pi_1$ is free so that $b_1$ is
%the connectivity.
Now assuming (an unfortunate)
%a visual
general collapse of our optical
%senses,
systems, with our brain-memories only able to remind from $M$
its numerical invariants of connectivity $b_1$ (rank),
indicatrix and ends-number $n$,
%one of the two above
%formulae
the above formula (where $b_2=0$ as soon as $M$ is open)
determines
%%$g$, hence
$\chi(F)$ (of ``the'' compact model) uniquely. Since
puncturing finitely many points does not affect the indicatrix
of a surface
(\ref{orientability-in-terms-of-Jordan-curves})---though it
may well do so for a non-Hausdorff curve (e.g., {\it
lasso}!)---the indicatrix of $F$ is
%uniquely defined
prescribed by that of $M$. Thus by the compact classification
(\ref{Moebius-Klein-classification}) the topology of $F$ is
unambiguously determined, and so is the topological type of
$M$. (Recall that the group of self-homeomorphisms of a
manifold acts transitively over finite configurations for any
prescribed cardinality.)
\end{proof}

Here are two examples that will play a special  r\^ole in the
foliated sequel:

\begin{lemma}\label{Kerekjarto:dicho}
A dichotomic metric surface with $\pi_1=F_2$ is $S^2_{3*}$
(thrice punctured sphere).
\end{lemma}

\begin{proof}
Having finite-connectivity, the surface $M$ is, by
\Kerekjarto{ }(\ref{Kerkjarto:end}), a punctured closed
surface $F_{n*}$. Since dichotomic implies orientable
(\ref{dicho-implies-orientable}), the closed model $F$ is
orientable
%(if not, $F$ contains a M\"obius band and puncture
%out of it, to transfer it in $M$),
(\ref{orientability-in-terms-of-Jordan-curves}) hence
$F\approx \Sigma_g$ (sphere with $g$ handles). Since
$\pi_1=F_2$ is not the group of a closed surface, $M$ is open,
and so $b_2=0$. Thus $\chi=1-b_1=2-2g-n$. Since $b_1=2$, we
have $2g+n=3$ implying (as $g,n\ge 0$) that $(g,n)=(0,3)$ or
$(1,1)$; the latter option being precluded by dichotomy.
\end{proof}

\begin{lemma}\label{Kerekjarto:non-orient}
A non-orientable metric surface with $\pi_1=F_2$ is
$\proj_{**}$ or $\Klein_*$ (twice-punctured projective plane
or once-punctured Klein bottle).
\end{lemma}

\begin{proof}
Being of finite-connectivity, the surface $M$ is, by
\Kerekjarto{ }(\ref{Kerkjarto:end}), a punctured closed
surface $F_{n*}$. Since orientability is hereditary to
subregions (\ref{orientability:heredity}), the closed model
$F$ is non-orientable, hence $F\approx S_{gc}$ (sphere with
$g\ge 1$ cross-caps). Since $\pi_1=F_2$ is not the group of a
closed surface, $M$ is open, and so $b_2=0$. Thus
$\chi=1-b_1=2-g-n$. Since $b_1=2$, we have $g+n=3$ implying
(as $g\ge 1$, $n\ge 1$) that $(g,n)=(1,2)$ or $(2,1)$.
\end{proof}

\subsection{The soul of a
non-metric finitely-connected surface
%of finite-connectivity
(\Kerekjarto, Nyikos)}

One can imagine that any surface of finite-connectivity has a
metric soul capturing its salient topological features and
outside which nothing more happens. This reminds
%certainly
the
%lovely
%philosophy
phraseology  ``{\it the garbage must cease}'' coined in Nyikos
1984 \cite{Nyikos84}.
%In the metric case this is
%essentially tautological yet the substance of \Kerekjarto's
%cylindrical ends theorem.
In the $\omega$-bounded case (which
implies finite-connectivity cf. e.g.,
\cite{Gab_2011_Hairiness}) the above
%philosophy
desideratum is a weak form of the bagpipe theorem of Nyikos
1984 \cite{Nyikos84}. Thus
%idea of
the present soul is
%%%essentially
merely a non-metric version of \Kerekjarto{ }cylindrical ends
theorem (\ref{Kerkjarto:end}) as well as an extension of
Nyikos' bagpipe (at any rate a typically Hungarian
%story).
endeavor.)

%We formalise the idea of capturing
As we are doing 2D-topology, the God-given recipe to capture
metrically  the whole topology is to impose
%some
``incompressibility'' at the fundamental group level:

\begin{defn}\label{soul:def} A soul
$S$ for a (non-metric) surface $M$ of
finite-connectivity is a metric
%subsurface
subregion $S\subset M$ such that the morphism induced by
inclusion $\varphi\colon\pi_1(S)\to \pi_1(M)$ is isomorphic.
%(By a subsurface we mean here an open
%connected subset, so that $S$  is likewise a borderless
%surface.)
\end{defn}

{\it Existence} of a soul is immediate from the kernel killing
procedure (\ref{killing:kernel}), and the interesting issue is
%the question of its
{\it uniqueness} (up to homeomorphism):
% We show
%uniqueness via the:

\begin{theorem}\label{soul:uniqueness} Let $M$ be a (non-metric)
surface of finite-connectivity. Then the three characteristic
invariants $(\chi, \varepsilon, a)$ (viz. Euler
%-Poincar\'e
characteristic, number of ends and indicatrix $a=0,1$ whether
orientable or not) of a soul are uniquely defined by the whole
surface $M$,
%%%matching
coinciding with
%the corresponding
%invariants
those of $M$. Consequently:

{\rm (a)} The topological type of a soul is uniquely defined
and
%called
referred to as the soul of the finitely-connected surface $M$
(apply \Kerekjarto{ }{\rm (\ref{Kerkjarto:end})}).

{\rm (b)} Any finitely-connected surface has a finite number
of ends (an issue not completely obvious a priori).
\end{theorem}

\begin{proof} If $M$ is compact this adds nothing new to the
classical classification (\ref{Moebius-Klein-classification}).
So assume $M$ open and then $b_2=0$ (by vanishing of the
top-dimensional homology, cf. e.g., Samelson
\cite{Samelson_1965-homology}), so that $\chi=1-b_1$. Hence
the knowledge of $\chi$ is equivalent to that of the
connectivity $b_1$. Hence the matching of $\chi$ is immediate
from the {\it soul}-definition (\ref{soul:def}).

The equality of the indicatrix (telling us orientability) is
evident as well. For instance one can use the canonical
group-morphism $\mu_M\colon\pi_1(M) \to \{\pm 1\}$ obtained by
propagating local orientation around loops
(\ref{indicatrix-orient-covering}). Orientability ($a=0$)
amounts to the triviality of $\mu_M$. Since we naturally have
$\mu_M \varphi = \mu_S$,  equality of the indicatrix follows
since $\varphi$ of (\ref{soul:def}) is isomorphic.

\def\endus{ends-number}

It remains only to check the equality of the ends-number. We
recall its:

\begin{defn} {\rm The {\it\endus}{ }of a space $X$ is the maximal
cardinality of non-relatively-compact residual components of a
compactum $K$ of $X$:
%. In formula:
$$
\varepsilon(X)=\sup_{K cpct\subset X } {\rm card} \{ C\in
\pi_0(X-K) \colon \overline{C} \textrm{ is non-compact} \}
$$
}
\end{defn}

{\small
---{\it Example: } Consider a letter ``Y'' with 3
branches going to infinity. Choose as compactum a point right
below the branching, we count 2 residual components, but
enlarging it we get 3 residual components (and never more!).
The space has 3 ends.

}

\smallskip

{\bf Step~1 (Deriving from a soul a weak bag-pipe
decomposition).} Given a soul $S$ of $M$ (hence of
finite-connectivity), we know that it is homeomorphic to
$F_{n*}$ a finitely $n$ times punctured closed surface $F$
(\ref{Kerkjarto:end}). Thus there is a compact bordered
subsurface $B \subset S$ obtained by removing from $F$ the
interior of little discs centered at the punctures. We call
$B$ a bag. It is
%%clearly homotopy equivalent to
a retract-by-deformation of the soul $S$, thus having the same
$\chi$, a number of contours equal to $n$
%(=\endus{ }or
%puncture number)
and the same indicatrix.

If we remove the interior of the bag $B$ from $S$ we
%%clearly
have $n$ residual components. Thus removing ${\rm int} B$ from
$M$ gives $k\le n$ components $P_1, \dots, P_k$ which are
bordered surfaces with $d_i\ge 1$ contours. Note that $\sum_i
d_i=n$.

First we claim that $k=n$ and that all $P_i$ are non-compact,
for otherwise
%two contours of $B$ are connected outside $B$,
%and
%it must then be easy to show
%that there is a pant subregion of $M$,
%then
arguing as in
Claim~\ref{claim:incompressible-implies-countour_equal-ends}
violates the isomorphy of $\varphi'\colon \pi_1(B)\to
\pi_1(M)$ (incompressibility condition). It follows that
$d_i=1$ for all $i$ (all $P_i$  have a single contour).

Next using the Mayer-Vietoris sequence we have
%(as if all
%manifolds were finite-combinatorial)
the following additivity relation
%(since
(intuitively  the overlapping occurs along circles not
contributing to $\chi$):
\begin{equation}\label{Mayer-Vietoris:char}
\chi(M)=\chi(B)+\textstyle\sum_{i=1}^n \chi(P_i).
\end{equation}
Since each $P_i$ is bordered (and connected), $b_2(P_i)=0$,
and so $\chi(P_i)=1-b_1(P_i)\le 1$. If $b_1(P_i)=0$, then
%since
as $P_i$ has a single contour it follows that $P_i$ is compact
(cf. \cite{GaGa2010}, Lemma 10), an absurdity.
%
%OLD ARGUMENT CENSURED
%
\iffalse and therefore the disc (by
 classification (\ref{Moebius-Klein-classification})). But
then the corresponding contour $\Gamma_i$ of $B$ is
trivialized in $\pi_1(M)$, violating the incompressibility,
excepted when $B$ is a disc, in which case $M$ would be the
sphere (which was ruled out at the start). \fi
Hence
$\chi(P_i)\le 0$, and since $\chi(M)=\chi(B)$ it follows from
\eqref{Mayer-Vietoris:char} that $\chi(P_i)=0$ for all $i$.
Thus
%%%%%$b_1(P_i)=1$, and
the filled $P_i$, denoted $P_{i,filled}$ (defined by gluing a
disc
%%%or collapsing
to its unique contour),
%%%$\Gamma_i$ to a point),
is a
$1$-connected surface (as its $\chi=\chi(P_i)+1=1$, so its
$b_1=0$, and as $\pi_1$ is free
(\ref{freeness-for-open-surfaces:prop})
%%%%it follows that
its $\pi_1=0$). (In Nyikos' jargon the $P_i$ now truly deserve
the name of {\it pipes}, yet not necessarily {\it long pipes}
which terminology might
%suggest
be reserved to the $\omega$-bounded case).

{\bf Step 2 (Computing the \endus)} Since $M$ contains the bag
$B$ (as a compactum)
%having
leaving $n$ residual (non-relatively-compact) components (cf.
Step~1), we have $\varepsilon(M)\ge n$. Conversely given a
compactum $K\subset M$, we may decompose it according to the
bagpipe decomposition $M=B\cup \bigcup_{i=1}^n P_i$ to obtain
a fragmentation $K=K_B\cup \bigcup_{i=1}^n K_i$, where
$K_B=K\cap B$ and $K_i=K\cap  P_i$ which are all compacta
(recall the bag and the pipes to be bordered hence closed as
point-sets). Regarding each $K_i$ in the filled pipe
$P_{i,filled}$,
%and
we can trace  a Jordan curve $J_i$ containing $K_i$ in its
interior, cf. (\ref{enclosing-trick:lemma}) below. Since the
interior $U_i$ of $J_i$ is homeomorphic to an open-cell (or
the plane) which is one-ended $U_i-K_i$ has exactly one
component which is not relatively-compact. Reconstructing the
manifold $M$ from its bagpipe structure it follows that $M-K$
has {\it at most} $n$ components, which are not
relatively-compact. (Some ``percolation'' of the connectedness
may of course occur within the bag.) This shows that
$\varepsilon(M)\le n$, completing the proof.
\end{proof}

\begin{lemma}\label{enclosing-trick:lemma}
Any compactum $K$ of an open simply-connected surface $M$ can
be ``enclosed'' in a Jordan curve $J$, in the sense that the
bounding disc for $J$ (given by Schoenflies {\rm
(\ref{Schoenflies-Baer})}) contains $K$ in its interior.
\end{lemma}

\begin{proof}
%The proof remains to be worked out.
By
%the
calibration
%method
(\ref{calibrated-exhaustions:lemma}), we have an exhaustion of
$M$ by Lindel\"of subregions $M_{\alpha}$ with trivial groups
$\pi_1(M_{\alpha})\approx\pi_1(M)=0$. Thus $M_{\alpha}$ is
$S^2$ or ${\Bbb R}^2$ (\ref{uniformization}). In the sphere
case $M_\alpha$ is both closed and open, hence equal to $M$
(connectedness), violating the openness assumption. Thus
$M_{\alpha}$ is the plane for all $\alpha$. Since $K$ is
compact, there is $\beta<\omega_1$ such that $M_{\beta}\supset
K$. By Heine-Borel, $K$ is closed and bounded, hence contained
in a ball of large radius.
\end{proof}

\section{Algebraic distractions (Freiheitss\"atze)}

\subsection{Freeness of the fundamental group of open surfaces
(Ahlfors-Sario)}

It is well known that the fundamental group of an open metric
surface is free on countably many generators
(\ref{Ahlfors-Sario:freeness:lemma}).
%%%%%%%%
\iffalse (cf.
%e.g.
Ahlfors-Sario, 1960 \cite[\S 44A.,
p.\,102]{Ahlfors-Sario_1960} or Massey's book 1967
\cite{Massey_1967}). \fi
%One can also argue with
Using Whitehead's spine
%(as in the proof of
%(\ref{Jordan-representant}))
(\ref{Whitehead-spine}) we get the stronger assertion that
%any
%open metric
such surfaces
%is
are homotopy equivalent to a countable graph. In general, a
{\it non-metric} (Hausdorff) surface may well
%have
deliver a free fundamental group requiring uncountably many
generators, as for the doubled Pr\"ufer surface $2P$ (cf.
Calabi-Rosenlicht \cite[p.\,339--40]{Calabi-Rosenlicht_1953}
for a complicated(?) proof of the non-denumerability of
$\pi_1(2P)$ or Gabard 2008 \cite[Prop.\,3]{Gabard_2008} for an
easy computation via Seifert-van Kampen, which was suggested
by M. Baillif). It puzzled us, over a long period of time,
%wondering if
whether the fundamental group of an arbitrary (non-metric)
open surface is free (e.g., both
%Gabard, 2008
%in
\cite[p.\,272]{Gabard_2008} and Baillif 2011
\cite{Baillif_2011} raise this question), yet it
%seems to be
%perhaps
is probably a trivial exercise. The basic idea
%being
is that if there is a relation in the $\pi_1$ of the (big)
non-metric surface then, covering by charts the range of a
null-homotopy materializing this relation, we get a Lindel\"of
subregion where this relation holds already, violating the
freeness in the metric case. The trick looks theological,
 as it does {\it not} exhibit a basis for the
fundamental group.
Let us look if this naive idea can be completed to a serious
argument. \iffalse The following is rather loose in this
respect and should be seriously improved (if possible?).
%
Let us take a closer look. \fi

\begin{prop}\label{freeness-for-open-surfaces:prop} The
fundamental group of any open surface is a free group.
\end{prop}

\begin{proof} Let $M$ be any open surface.
If $G:=\pi_1(M)$ is not free then there is a reduced non-empty
word $w=w(x_1,\dots,x_k)$ in some variables $x_i$ which purely
specialized to elements $g_i\in G$
%not
%all trivial
yields the equation $w(g_1,\dots,g_k)=1$
%holding
in $G$ (cf. Lemma~\ref{non-free} below and the definition
after it for the meaning of pureness). Choose $c_i$ some loops
representing the $g_i$. Cover the range of a null-homotopy
shrinking the concatenation $w(c_1,\dots,c_k)$ to the
basepoint by a finite number of charts to obtain a Lindel\"of
subregion $L$. Of course $L$ contains the $c_i$ (their ranges
to be accurate), and so the $c_i$ define elements in
$\pi_1(L)$, say $\gamma_i$. Of course the relation
$w(\gamma_1,\dots,\gamma_k)=1\in \pi_1(L)$ continues to hold
and the $\gamma_i$ are all non-trivial since they
 map to the $g_i\neq 1$ under the
 %natural
 morphism $\pi_1(L)\to
 \pi_1(M)$ induced-by-inclusion. Notice
 %also
 that
 the specialisation of $w$ via the assignment
 $x_i\to \gamma_i$ is pure. By the reverse
 implication of Lemma~\ref{non-free} we deduce that $\pi_1(L)$ is
 not free, violating the classical (metric) case of
 the proposition (\ref{Ahlfors-Sario:freeness:lemma}).
\end{proof}

The next lemma
%sounds like
%is almost the tautology
sounds tautological: {\it a group is not free iff there is a
relation},
%except that one must keep in mind
%%modulo that one can
%%%playing as usual with both
and just amounts to the interplay between the universal
%property
description of free groups with the more concrete model in
terms of words spelled in an alphabet:

\iffalse yet since it was formulated by the great Alexander
with very roosted algebraic reminiscences, it is not
impossible that it is quite false. \fi

\begin{lemma}\label{non-free}
A group $G$ is \emph{not} free if and only if there is
%some integer $k\ge 1$ and
a non-empty reduced word $w=w(x_1,\dots,x_k)$ in $k\ge 1$
letters $x_1, \dots, x_k$ and a pure specialization $x_i\to
g_i$ to elements $g_i\in G$
%not all $1\in G$ (neutral
%element)
(cf. definition below) such that $w(g_1, \dots, g_k)=1$.
\end{lemma}

\begin{defn} {\rm A specialisation of a word $w$ in a group
$G$ is {\it pure} if whenever two letters $x,y$ of $w$ are
adjacent they do not specialize on $g\in G$ and $g^{-1}$, and
if $xy^{-1}$ or $x^{-1}y$ appears in the word $w$, $x$ and $y$
do not specialize on the same element $g\in G$. We also demand
that no letter of $w$ specialize to $1\in G$.}
\end{defn}

\begin{proof} [{\bf Proof}] [$\Rightarrow$]
%This is the easy implication.
If $G$ is not free, then $G$ is still the quotient of a free
group $\varphi\colon F \to G$ with non-trivial kernel
$\ker\varphi$. Pick a non-trivial element $1\neq x \in
\ker\varphi$. Let the set $X$ be a basis for the free group
$F$.
%We may assume that for all $y\in X$,
%$\varphi(y)\neq 1$ for otherwise we can remove this element
%without affecting the ontoness of $\varphi$.
As is well-known $X$ generates $F$ and
%so
$x\in F$ can be written as a non-empty reduced word
$w=w(x_1,\dots,x_k)$ involving finitely many $x_i\in X$. Let
$g_i=\varphi (x_i)$. As $\varphi(x)=1$, we have
$w(g_1,\dots,g_k)=1$.
%Notice that all $g_i$ are not $1$ by the way we have chosen $X$.
Furthermore pureness of the specialisation $x_i\to g_i$
follows, if we take care assuming
%that
the word $x$ to have minimal length among all those
non-trivial elements of the kernel $\ker\varphi$.

[$\Leftarrow$]
%This implication looks harder.
Assume that $G$ is free, say with basis $X\subset G$. Let
$w=w(x_1,\dots,x_k)$ be a non-empty reduced word in some
abstract symbols $x_i$ and let $x_i\to g_i\in G$ be a pure
specialization;
%to some $g_i\in G$ (not all $1$);
to show $w(g_1,\dots,g_k)\neq 1$. Each $g_i$ can be written as
a non-empty reduced word $w_i$ involving finitely many letters
of the alphabet $X$. Substitute these expressions in $w$ to
obtain the big word $w(w_1,\dots,w_k)$. \iffalse , a priori
not reduced. The game is over if we are able to show that this
word is non-empty after reduction. Since not all $g_i$ are
$1$, one at least say $g_j$ corresponds to a non-empty word
$w_j$. Since the specialization is pure it is clear that there
is no complete cancellation in the big word $W$ which is thus
non-empty after reduction. \fi If the latter collapses
completely under reduction then this forces two adjacent words
$w_i$, $w_j$ to cancel out, violating the pureness of the
specialization.
\end{proof}

\subsection{Freedom for non-Hausdorff 1-manifolds
by reduction to
%open
surfaces
%(Haefliger, Thurston, Penner)
}

Another ``mystical'' question
%in the same spirit
(also eluding us for a long time, and indeed still eluding us
slightly) is whether the fundamental group of a
(non-Hausdorff) $1$-manifold is always free. (For Hausdorff
$1$-manifolds, we have a classification in 4 specimens, which
all have trivial groups, except the circle.)
 Of
course the same Lindel\"of reduction as we just did for
Hausdorff surfaces is formally possible, yet not very
effective unless the Lindel\"of case is
%%%known.
settled.
%So is it easy to show that a
%Lindel\"of one-manifold has free fundamental group. If true
%then freeness holds in full generality.

Maybe
%a somewhat biased
the royal road (suggested by a discussion with A. Haefliger)
to this freeness
%problem
curiosity is a geometric
%approach
construction
%to this problem would
%be to
%showing that
exhibiting any (non-Hausdorff) $1$-manifold $M^1$
%is covered
as the base of a fibration of a Hausdorff surface $M^2$
%foliated
by real-lines (recovering via the leaf-space the given $M^1$).
%Then applying
The exact homotopy sequence of a fibration\footnote{This
foolhardy idea
%of using the exact homotopy sequence
was suggested
%to us
orally by Haefliger (circa 2006). One
%needs
has to convince that the classical proof does not use the
Hausdorffness of the base; compare Hopf-Eckmann,
Ehresmann-Feldbau, Steenrod, etc., yet a detailed redaction is
maybe desirable.}
%we have
reads (denoting by $F\approx {\Bbb R}$ the fibre):
\begin{equation}\label{homotopy-sequence}
\{1\}=\pi_1(F)\to \pi_1(M^2)\to \pi_1(M^1) \to \pi_0(F),
\end{equation}
from which we deduce the required freeness of $\pi_1(M^1)$ via
(\ref{freeness-for-open-surfaces:prop}). In the simplest case
where $M^1$ is a branching line or a line with two origins it
is clear how to construct such a ``thickened'' fibration
$M^2\to M^1$, essentially like for {\it train-tracks} (\`a la
Thurston-Penner), cf. Figure~\ref{Train:fig}.
%
%
%
%Of course
%
%
%
%
The above idea originates in Haefliger, 1955 \cite[p.\,8,
point {\bf 2.}]{Haefliger_1955}, where
%it is written
we read: {\small
\begin{quote}
On peut montrer que toute vari\'et\'e \`a une dimension avec
un nombre fini de bord et dont le groupe fondamental a un
nombre fini de g\'en\'erateurs\footnote{Of course we may
wonder if this proviso is really required. We believe it is
not.}, est l'espace des feuilles d'une structure feuillet\'ee,
et m\^eme la base d'une fibration par des droites d\'efinie
sur une vari\'et\'e s\'epar\'ee \`a deux dimensions.
\end{quote}}
\noindent It also reappears in Haefliger-Reeb 1957
\cite[p.\,125, last sentence]{Haefliger_Reeb_1957}, where
%near
%at the end of the paper they claim
it is asserted (again without proof) that any second-countable
$1$-connected (non-Hausdorff) \hbox{$1$-manifold} can be
realised as the leaf-space of a suitable foliation of the
plane.

\begin{figure}[h]
\centering
    \epsfig{figure=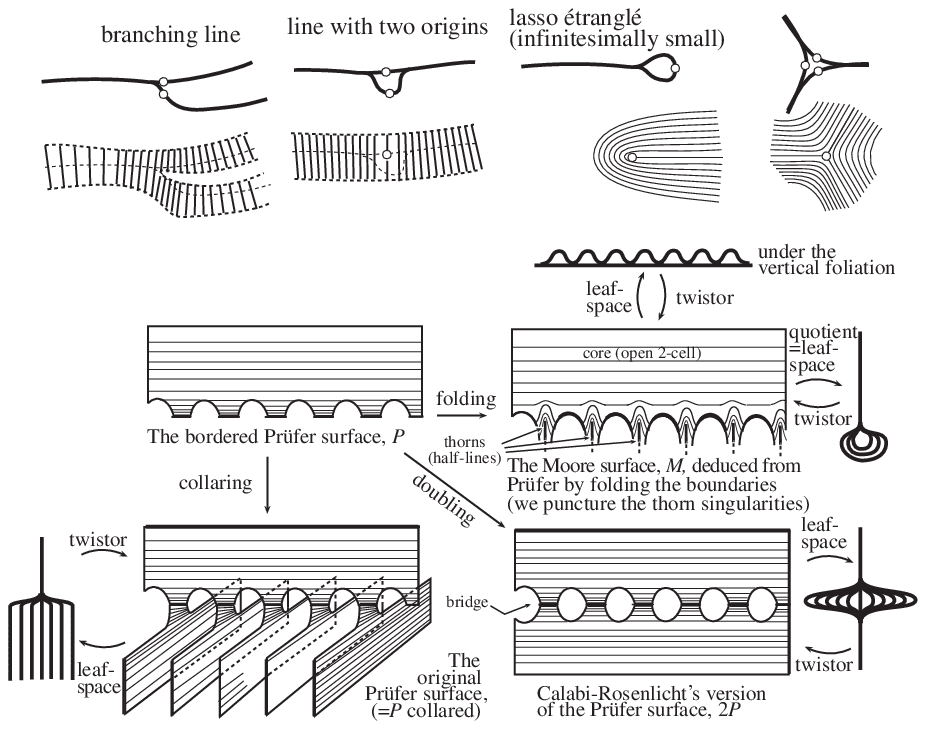,width=122mm}
  \caption{\label{Train:fig}
  %Artist view of
  The Haefliger
  %(Hausdorff)
  twistor of a
  (non-Hausdorff) 1-manifold}
\vskip-5pt\penalty0
\end{figure}

\subsection{Twistor or train-tracks (Haefliger, Thurston, Penner)}

Can somebody prove the following {\it hypothetical} lemma
(strongly inspired from Haefliger and  Thurston, Penner's
train-tracks):

\begin{lemma} \label{twistor} (Twistor trick or train-tracks.)
Any (non-Hausdorff) $1$-manifold can (non-canonically) be
materialized
%(more-or-less
as the base of a
%(locally trivial)
fibration $p\colon M^2\to
M^1$
%whose total space is
of a Hausdorff surface
%and whose fibre are
fibred by real-lines ${\Bbb R}$. In particular, the projection
$p$ induces an isomorphism on the $\pi_1$.
\end{lemma}

We hope the result is true, being implicitly used in
Haefliger-Reeb \cite{Haefliger_Reeb_1957} at least when the
manifold $M^1$ is $1$-connected and second-countable
(equivalently Lindel\"of, because manifolds are locally
second-countable). Another little piece of evidence is that a
(Morse theoretical) reverse engineering seems to hold
metrically: {\it any open metric surface can be fibred by
lines so that the quotient is a non-Hausdorff curve}
(\ref{Morse-Thom:surfaces}). Even if
%the train-track lemma
(\ref{twistor}) should work only in the Lindel\"of case, this
would be
%strong
punchy enough to settle the general freeness
%problem
question (in view of the Lindel\"of reduction trick used in
(\ref{freeness-for-open-surfaces:prop})).

\begin{exam}[Non-Lindel\"of twistors] {\rm It is worth noticing
that Pr\"ufer's construction (and its
%variations
derived products like Moore or Calabi-Rosenlicht) provides
twistors for several toy-examples of non-Lindel\"of
$1$-manifolds (cf. Figure~\ref{Train:fig}, bottom part). For
instance the horizontally foliated (classical) Pr\"ufer
surface
%%%is a twistor for
twistorize the line with continuously ($\frak c={\rm card
({\Bbb R})}$) many branches. (Hence there is at least no
%categorical
%%%%absolute
%%%structural
visceral
%obstruction to
incompatibility between the twistor desideratum
(\ref{twistor})
%in
%with
and the non-Lindel\"of context.) Likewise the leaf-space of
the horizontally-foliated doubled Pr\"ufer surface $2P$ is the
line with $\frak c$-many origins. The vertical foliation on
the Moore surface punctured along the folded points admits as
leaf-space (and therefore is a twistor for) the {\it
everywhere doubled line} (described in Baillif-Gabard 2008
\cite[\S 3]{BG_2008_PAMS}). Finally, the horizontally-foliated
Moore surface punctured at all thorns singularities is a
twistor for the ``{\it lasso \'etrangl\'e}'' with continuously
many (infinitesimal) loops. }
\end{exam}

We now
%recall
tabulate
%for convenience,
%some applications
formal consequences of this geometric construction
(\ref{twistor}):

\begin{cor} (Haefliger-Reeb 1957
{\rm \cite{Haefliger_Reeb_1957}}) All simply-connected
second-countable (non-Hausdorff) $1$-manifold occur as the
leaf-space of a suitable foliation of the plane. Relaxing
second-countability of $M^1$, the same is true for a suitably
foliated simply-connected surface.
\end{cor}

\begin{proof} Given such a $1$-manifold $M^1$, we consider its
twistor $\tau M^1\to M^1$ given by (\ref{twistor}). By the
isomorphism \eqref{homotopy-sequence}, the total space
$M^2:=\tau M^1$ is simply-connected and Lindel\"of (by the
general topology version of Poincar\'e-Volterra, cf. Bourbaki
or Guenot-Narasimhan as referenced in \cite{GabGa_2011}).
Since $M^2$ is non-compact (containing lines as closed
subsets), it is homeomorphic to ${\Bbb R}^2$ by classification
of 1-connected surfaces
%  (compare for
%instance the combinatorial proof of van der Waerden-Reichardt
%as referenced in \cite{GabGa_2011}).
(\ref{uniformization}).
\end{proof}

More generally we have the following (with relaxable
%optional
%omitable
parenthetical provisos):

\begin{cor}
Any (second-countable) $1$-manifold is the leaf-space of a
foliation by lines of a (metric) surface with the same
fundamental group.
\end{cor}

Finally regarding the  fundamental group structure  we have:

\begin{cor}
All
%(second-countable,
(non-Hausdorff) $1$-manifolds have free fundamental groups.
\end{cor}

\begin{proof}
Consider again the twistor $M^2\to M^1$ of the $1$-manifold.
By \eqref{homotopy-sequence} again, $\pi_1(M^2)$ is isomorphic
to $\pi_1(M^1)$, and the former is free by
(\ref{freeness-for-open-surfaces:prop}). In case the twistor
trick (\ref{twistor}) should hold only for Lindel\"of
$1$-manifolds, then first establish the corollary in that
case, and next extend universally by the Lindel\"of reduction
trick used in the proof of
(\ref{freeness-for-open-surfaces:prop}).
\end{proof}

\section{Foliated foundations}

Before penetrating truly to our main object, we recall some
classical facts for later references. Below, {\it
$1$-foliation}
%is just an abbreviation for
abbreviates ``one-dimensional foliation''.

\subsection{Orienting
%%%$2$-fold
double cover (Haefliger, Hector-Hirsch, etc.)}
\begin{prop}\label{orienting:2-fold-covering}
Given a $1$-foliation of a manifold, there is  a
%%%two-fold
double cover such that the lifted foliation is orientable. In
particular, any $1$-foliation of a simply-connected manifold
is orientable.
\end{prop}

\begin{proof} If no smoothness is postulated,
%one must
%employ
some tricks with germs act as a substitute to the tangent line
bundle of the foliation (cf. Haefliger 1962 \cite{Haefliger62}
or Hector-Hirsch 1981-83 \cite{HectorHirschA, HectorHirschB}).
Since the construction is purely local, there is no hindrance
in implementing it in the globalized world of non-metric
manifolds.
\end{proof}

\subsection{Compatible flows (\Kerekjarto, Whitney)}

In contrast the following paradigm is much more {\it
metric}\,-sensitive
%to the metric
(indeed false without one, as amply discussed in
\cite{GabGa_2011}):

\begin{theorem}\label{Kerek-Whitney:thm}
(\Kerekjarto{ }1925, Whitney 1933) Given an oriented
$1$-foliation of a metric manifold, there is a compatible
flow, whose trajectories are the leaves.
%%%of the foliation.
\end{theorem}

\begin{proof} The 2D-case is due to
\Kerekjarto{ }1925 \cite{Kerekjarto_1925}, and the general one
to Whitney 1933 \cite{Whitney33}.
\end{proof}

\begin{cor}\label{disc-cannot-be-foliated}
The $2$-disc cannot be foliated (tangentially).
\end{cor}

\begin{proof} Recall two classical arguments:

$\bullet$ {\it Via Brouwer.} Assuming it could, then as the
disc is 1-connected the foliation is orientable
(\ref{orienting:2-fold-covering}), hence admits a compatible
flow (\ref{Kerek-Whitney:thm}). Passing to dyadic times
$t_n=1/2^n$ of the flow, gives a nested sequence of non-empty
closed sets (Brouwer's fixed point theorem) whose common
intersection is non-void by compactness. Thus a rest-point for
all times is created, violating the flow compatibility  with
the foliation.

$\bullet$ {\it Via H. Kneser.} Double the foliated disc to get
a foliated sphere with $\chi=2$, violating Kneser's
combinatorial proof of the Euler obstruction
(\ref{Kneser:Poincare-Dyck:Euler obstruction}) below.
\end{proof}

\begin{cor}\label{Euler:classical-obstruction}
More generally, a closed topological manifold
%(finite-dimensionality)
foliated by curves has zero Euler characteristic.
\end{cor}

\begin{proof} If not, then passing to the orienting cover
(\ref{orienting:2-fold-covering}), we may assume the foliation
oriented. Consider a compatible flow
(\ref{Kerek-Whitney:thm}), which by Lefschetz's fixed point
theorem \cite{Lefschetz_1937} (version for ANR's) has a fixed
point, an absurdity. In the surface case one can also argue
elementary \`a la Kneser via (\ref{Kneser:Poincare-Dyck:Euler
obstruction}).
\end{proof}

\subsection{Beck's technique (plasticity of flows)}
\label{Beck's_technique:section}

Albeit we are primarily interested in foliations, some  facts
concerning  flows will be useful in the sequel. A basic
desideratum, when dealing with flows, is a two-fold yoga of
``restriction'' and ``extension'':

(1) {\it Given a flow on a space $X$ and an open subset
$U\subset X$, find a flow on $U$ whose phase-portrait is the
trace of the original one}; and conversely:

(2) {\it Given a flow on $U$, find a flow on $X\supset U$
whose phase-portrait restricts to the given one.}

\smallskip
Thus, one expects that any open set of a
brushing\footnote{That is a space admitting a fixed-point free
flow.} is a brushing, and that any separable super-space of a
transitive space is
%likewise
transitive, provided the sub-space is dense (or becomes so,
after a suitable inflation).

Problem (1) is solved in Beck~\cite{Beck_1958}, when $X$ is
metric (via passage to the induced foliation this also derives
from \Kerekjarto-Whitney (\ref{Kerek-Whitney:thm})). (An
example in \cite{GabGa_2011} indicates a non-metric
disruption.) The same technique of Beck (clever time-changes
afforded by suitable integrations), solves Problem~(2) in the
metric case (compare \cite[Lemma 2.3]{Jimenez_2004}):

\begin{lemma}\label{Beck's_technique:extension:lemma}
Let $X$ be a locally compact metric
space and $U$ and open set of $X$. Given a flow $f$ on $U$,
there is a new flow $f^{\star}$ on $X$ whose orbits in $U$ are
identic to the one under $f$.
\end{lemma}

\subsection{Foliated triangulations (H. Kneser)}

It is hard to resist recalling Hellmuth Kneser's combinatorial
approach to the Euler obstruction. We admit the following
referring for clean proofs to Kneser 1924 \cite{Kneser24} or
Hector-Hirsch \cite{HectorHirschA}.

\begin{lemma}\label{Kneser:triangulation-compatible-foliation}
(Kneser 1924) Any metric foliated surface has a
%compatible
``generic''  triangulation where each $2$-simplex is
%canonically
transversely foliated as depicted on Fig.\,\ref{Kneser:fig}.b.

\iffalse (no edge lyes in a leaf), say as the standard simplex
$\Delta=\{(x_1,x_2)\in {\Bbb R}^2: x_1+x_2\le 1, x_i\ge 0 \}$
foliated by horizontal lines $x_1-x_2=\text{constant}$. \fi
\end{lemma}

\begin{proof} (Dirty outline)
By definition of a foliated structure it is rather clear that
we have a tessellation by foliated boxes, which are squares.
Then we add diagonals to get triangles and whenever two of
them
%have a common face which is leaved
are adjacent along a piece of leaf, we perform Kneser's flip
depicted on Fig.\,\ref{Kneser:fig}.e (gaining transversality).
\end{proof}

%\vskip-10pt\penalty0
\begin{figure}[h]
\centering
    \epsfig{figure=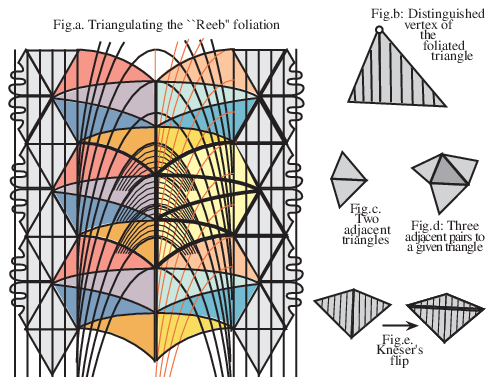,width=122mm}
\vskip-105pt\penalty0
  \caption{\label{Kneser:fig}
  %Artist view of
  Kneser's proof of the Euler-Poincar\'e-Dyck obstruction}

\end{figure}

Since any open metric surface can be foliated
(\ref{Morse-Thom:surfaces}) (=Morse theoretical trick), this
%implies
suggests another proof of Rad\'o's triangulation theorem
(\ref{Rado:triangulation}) at least for open surfaces. (Of
course Rad\'o was well aware of
%this
Kneser's paper, cf. his article \cite{Rado_1925}, but probably
not of the Morse theoretical trick.)

\begin{cor}\label{Kneser:Poincare-Dyck:Euler obstruction}
(Poincar\'e 1885, Dyck 1888, Kneser 1924) A closed surface
which is foliated has vanishing Euler characteristic $\chi=0$.
\end{cor}

\begin{proof} (Kneser). By
(\ref{Kneser:triangulation-compatible-foliation}) there is a
triangulation transverse  to the foliated structure, which
 is finite by compactness. So we
may compute the characteristic as the alternating sum of the
cardinalities $e_i$ of the set $\sigma_i$ of simplices of
dimensionality $i=0,1,2$:
\begin{equation}\label{Kneser:char}
\chi=e_0-e_1+e_2.
\end{equation}
First ignoring the foliation, recall
%that in any triangulation
%of a surface
the relation $2e_1=3e_2$, cf.
(\ref{Kneser-Descartes-etc:lemma}) right  below. Besides, any
(transversely foliated) $2$-simplex has a distinguished vertex
through which a piece of leaf
%is going
traverses the 2-simplex (Fig.\,\ref{Kneser:fig}.b). So we have
a map $\sigma_2\to \sigma_0$ which is onto and
%%%indeed
2-to-1 as the leaf extends in 2 directions. Hence $e_2=2e_0$.
Plugging those relations in \eqref{Kneser:char} gives:
%%%the conclusion:
%\begin{equation}\label{Kneser:char:detail}\notag
$
\chi=e_0-e_1+e_2=\textstyle\frac{1}{2}e_2-\frac{3}{2}e_2+e_2=0.
$
%\end{equation}
\end{proof}

\begin{lemma}\label{Kneser-Descartes-etc:lemma} (Descartes, Euler 1750, L'Huilier 1811, who else?)
In any finite triangulation of a closed (compact non-bordered)
surface the relation $2e_1=3e_2$ holds true between the
numbers
 $e_1$, $e_2$ of edges, respectively  triangles.
\end{lemma}

\begin{proof}
Let
%again
$\sigma_i$ be the set of simplices of dimension $i$ and
consider the incidence relation {\it two triangles have a
common edge} (adjacent triangles, Fig.\,\ref{Kneser:fig}.c):
$$
I=\{(\Delta_1,\Delta_2)\in \sigma_2 \times \sigma_2: \Delta_1
\cap \Delta_2=\text{one edge} \}
$$
Mapping such a pair to its common edge yields a map $I\to
\sigma_1$ which is onto (the surface being non-bordered) and
2-to-1 as the pair-order  is permutable. Thus the cardinality
of $I$ is $\# I = 2 e_1$.
%On the other hand,
Besides, projecting
%$I$
on the first factor (say) gives a map
$I\to \sigma_2$ which is onto and 3-to-1
(Fig.\,\ref{Kneser:fig}.d), whence
%Thus
$\# I = 3 e_2$.
\end{proof}

\subsection{Open metric surfaces fibrates (Morse, Thom, etc.)}

As well-known, metric differentiable manifolds admit {\it
Morse functions}, which as a reaction to the complicated
%underlying
topology (or rather the compactness)
%present
deliver generally critical points. (In the late 60's, Morse as
well as Kirby-Siebenmann explained how to get rid off the
differentiable proviso
%proving the existence of
using so-called {\it topological} Morse functions.)
When the manifold is open, one can (in principle)
%find a critical point free
%Morse function:
eliminate critical points by rejection to $\infty$:

\begin{theorem}\label{Morse-Thom-Whitehead-Hirsch}
Any open metric (topological) manifold has a critical point
free Morse function. The latter, being a submersion,  defines
a codimension-one foliation, whose transverse ``line-field''
gives a $1$-foliation whose leaves are lines.
\end{theorem}

\begin{proof} {\it (Heuristic outline).} Choose any Morse
function $f$, i.e.  locally resembling a quadratic
non-degenerate form $x_1^2+\dots+x_p^2-x_{p+1}^2-\dots-x_n^2$.
In particular critical points are isolated. In every open
metric manifolds $M$, one can---starting from any
point---trace a {\it ventilator} (jargon borrowed from L.
Siebenmann), i.e., an arc $A$ homeomorphic to a semi-line
$[0,\infty)$ such that $M$ slitted along $A$ is homeomorphic
to $M$, i.e., $M-A\approx M$.  Since $M$ is metric, the
critical-set of $f$, being discrete,  is countable. Thus by a
%appropriate
(hazardous?)
%induction,
infinite-repetition, we may remove inductively ventilators
emanating from the critical points, to reach the first claim.
The addendum follows by aping ``the'' gradient flow of the
Morse function via a technique of
Siebenmann~\cite{Siebenmann72}. (In the smooth case just
integrate the transverse line-field, w.r.t. an auxiliary
Riemannian metric.)

{\it (Variant of proof in the PL-case)} Compare Hirsch
1961~\cite{Hirsch_1961}{}, using the theory of Whitehead's
spine.
\end{proof}

\begin{prop}\label{Morse-Thom:surfaces} Any metric open
surface
%has a foliation.
foliates. More is true it can be foliated by lines, probably
even in the following hygienical way:

{\it (Hypothetical addendum).}---Any such surface can  be
regarded as the total space of a fibration by lines whose base
is a (non-Hausdorff) $1$-manifold.
\end{prop}

\begin{proof} {\it (High-brow proof)} This follows by
specializing  the above (\ref{Morse-Thom-Whitehead-Hirsch}),
taking (optionally) advantage of the smoothability of metric
surfaces. Smoothing(s) can be deduced from Rad\'o
(\ref{Rado:triangulation}) using
%optionally
eventually the trick of Stoilow-Heins to introduce a
(stronger) Riemann surface (${\Bbb C}$-analytic) structure in
the orientable case (and by adapting a Klein(=di-analytic)
surface structure) in the non-orientable case.

{\it (Elementary proof?)} Start from a triangulation given by
Rad\'o (\ref{Rado:triangulation}) and try to find some clever
combinatorial procedure to
%construct
%
propagate a
%compatible
foliated texture \`a la Kneser.
(Details left to the imaginative
%of the
readers.)

---{\it Outlined addendum.} Integrating the vector
field orthogonal to the level curves of a critical-point-free
Morse function $f$, and having speed-one w.r.t. a complete
Riemannian metric), we obtain a (fixed-point-free) flow
$\varphi\colon {\Bbb R}\times M \to M$ without
``recurrences''. Let $\cal F$ be the underlying foliation.
%Thus the underlying foliation $\cal F$ (trajectories of
%$\varphi$) has a leaf-space which is a (non-Hausdorff)
%1-manifold (because each leaf appears at most once in a
%foliated-box). Moreover
The projection on the leaf-space $p\colon M \to M/{\cal F}$ is
a fibration (by lines). Indeed, given
%an orthogonal
a trajectory of $\varphi$ (say that of the point $x$), one can
let evolve in time a small 1D-chart $V$ (selected) in the
$f$-level-curve,  $f^{-1}(f(x))$, through $x$
 to manufacture a
2-cell $U:=\varphi({\Bbb R}\times V)\approx {\Bbb R}\times V$
(via $\varphi(t,v)\leftarrowtail(t,v)$) which is (trivially)
fibred by lines (the trajectories). The projection of $U$ in
the leaf-space is open (its inverse image being $U$ which is
open) and
%\iffalse $h\colon p^{-1}(p(U))=U\approx V \times
%{\Bbb R}$ is fibre-wise a product with trivialisation $h$
%afforded by the flow, set for $u\in U$, $h(u)=(s,t)\in V\times
%{\Bbb R}$ uniquely defined by $\varphi(t,s)=u$. \fi
homeomorphic to $V\approx {\Bbb R}$ (restrict $p$ to $V$).
Hence the leaf-space is a $1$-manifold, and $p$ is a fibration
(trivial over $p(V)$).
\end{proof}

This addendum looks
%to be a reverse procedure
somewhat dual to the engineering of Haefliger (\ref{twistor})
permitting to
%see
conceive any (non-Hausdorff) 1-manifold as the base of a
fibration by lines of a Hausdorff surface.

---{\it Baby example.} Consider a ``Y''-shaped surface
resembling a tree in usual 3-space ${\Bbb R}^3$ with three
trunks going to infinity (Fig.\,\ref{ababy:fig}.a). The height
function ``$z$'' (third coordinate) has a critical point where
the two branches of the tree ``Y'' bifurcate. The latter can
be eliminated just by
%resorbing
deforming one of the branch horizontally and letting it
disappear to $\infty$ like a ``cusp''
(Fig.\,\ref{ababy:fig}.b).  Those surfaces are just
diffeomorphic to a punctured cylinder that can be imagined
endowed with a complete Riemannian
%metric of growing intensity as the
%puncture is approached. i.e.
whose line-elements diminish in size (w.r.t. to the Euclidean
element) as the puncture is approached
(Fig.\,\ref{ababy:fig}.c). The height-function is
critical-point-free and the orthogonal trajectories are
vertical lines on the cylinder-model (with a sole interruption
at the puncture), yet drastically slowed-down (w.r.t. the
Euclidean perception) as we approach the dark-matter
concentrated near the puncture. The leaf-space (of the
transverse foliation) is a circle-with two origins (just
identify two transverse circles lying above resp. below the
puncture, whenever they are intercepted by a same leaf)
(Fig.\,\ref{ababy:fig}.d).

\begin{figure}[h]
\centering
    \epsfig{figure=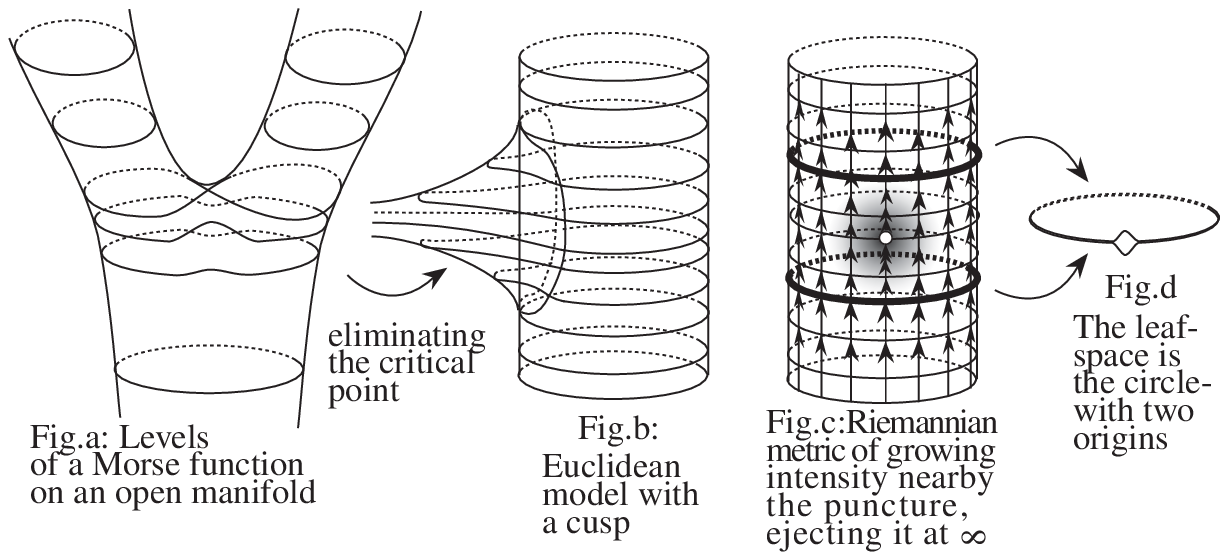,width=108mm}
\vskip-10pt\penalty0
  \caption{\label{ababy:fig}
  %Artist view of
  Foliating open surfaces by lines}
\vskip-5pt\penalty0
\end{figure}

%%%Non-metric
\section{Haefliger-Reeb theory for non-metric
simply-connected
%foliated
surfaces}

The starting point for this section (and actually of the whole
paper) was the observation by D. Gauld that a foliated avatar
of the ``closing lemma'' (Figure~\ref{David's_trick}, Case~1)
shows that a simply-connected surface (even when non-metric)
cannot be transitively foliated. Exploiting this remark allows
one to extend the Haefliger-Reeb theory to all 1-connected
surfaces, pushing its validity outside any metrical
predisposition,  evidencing a rather robust character of their
theory.

\subsection{Haefliger-Reeb derives from Schoenflies (Gauld)}

The game in this section is to ape
%as closely as possible
non-metrically the Haefliger-Reeb theory 1957
\cite{Haefliger_Reeb_1957} describing
%the qualitative (and to some extent the
%quantitative) features of
foliations on the plane ${\Bbb R}^2$.
(This was also studied earlier in 1940 by W. Kaplan.)
%[a student
%of Whitney, and visitor of Kerekjarto]).
Replacing the plane by an arbitrary (non-metric)
simply-connected surface, we passively observe that most of
the classical theory remains valid in this broader context.
(The only minor divergence is that the projection on the
leaf-space can now cease to be a fibration, an issue only
emphasised in subsequent papers, e.g. of Godbillon and Reeb.)
The {\it raison d'\^etre} for this extension is
%simply
%%%%that what lies at the hearth of the matter is the
%validity
the
%pivotal r\^ole and
non-metric availability of the Schoenflies theorem which is
implied by, and indeed equivalent to, simple-connectivity (see
Gabard-Gauld 2010 \cite{GaGa2010} or
(\ref{Schoenflies-Baer})).

\begin{prop} \label{Alex_separation} A
foliated simply-connected surface satisfies:

{\rm (a)} Any leaf is open as a manifold (i.e., no compact
circle leaf).

{\rm (b)} A leaf
%can
appears at most once in any foliated chart. More precisely if
a leaf intersects a foliated chart then this intersection
reduces to a single line
%, in other
%words a single
(plaque).

{\rm (c)} From {\rm (b)}, it follows that the leaf-space is a
$1$-manifold, in particular any leaf is closed as a point-set.
Also leaves are proper, i.e. the leaf topology matches with
the relative
%/induced
topology. Still from {\rm (b)} leaves cannot be dense.
%that any foliated chart can
%meet any given leaf only along along a single plaque. In other
%other words

{\rm (d)} By {\rm (a)} any leaf has two ends and runs to
infinity in both directions while dividing the surface in two
components (called halves).
\end{prop}

\begin{rem} {\rm This statement
may well be empty when the surface lacks any foliation. This
is the case of the 2-sphere, but can also occur to non-compact
surfaces, e.g. the long glass ${\Bbb S}^1\times {\Bbb L}_{\ge
0}$ capped off by a $2$-disc (compare \cite{BGG1}). Point (b)
is exactly Th\'eor\`eme 1 in Haefliger-Reeb
\cite[p.\,120]{Haefliger_Reeb_1957}, and  a direct consequence
is that the leaf-space is a (generally non-Hausdorff)
1-manifold. }
\end{rem}

\begin{proof} (a) is obvious, for a circle leaf would bound a
foliated disc by Schoenflies (see \cite{GaGa2010} or
(\ref{Schoenflies-Baer})), which is an absurdity
(\ref{disc-cannot-be-foliated}).
%Recall indeed that by Whitney 1933 \cite{Whitney33} there
%would be a compatible (fixed-point-free) flow violating the
%Brouwer fixed point theorem.

The proof of (b) is a similar Schoenflies obstruction modulo
some tricks reminiscent of the Poincar\'e-Bendixson trapping
argument or rather the closing lemma (for dynamical flows).
%Indeed if a certain leaf would  not be closed, then it
%has to accumulate over itself in the sense that
%there is an appropriate foliated box about a point of the leaf $L$ which
%will be revisited by $L$.
Assume that a leaf returns to a foliated chart. Orient the
foliated box as well as the leaf. Then one distinguishes two
cases depending on whether the first-return to the box matches
or reverses the orientation (cf. Figure~\ref{David's_trick}).
In fact
%by Remark~\ref{simplification:rem} below,
since the foliation is orientable
(\ref{orienting:2-fold-covering}), only the first case needs
attention.

\begin{figure}[h]
\centering
    \epsfig{figure=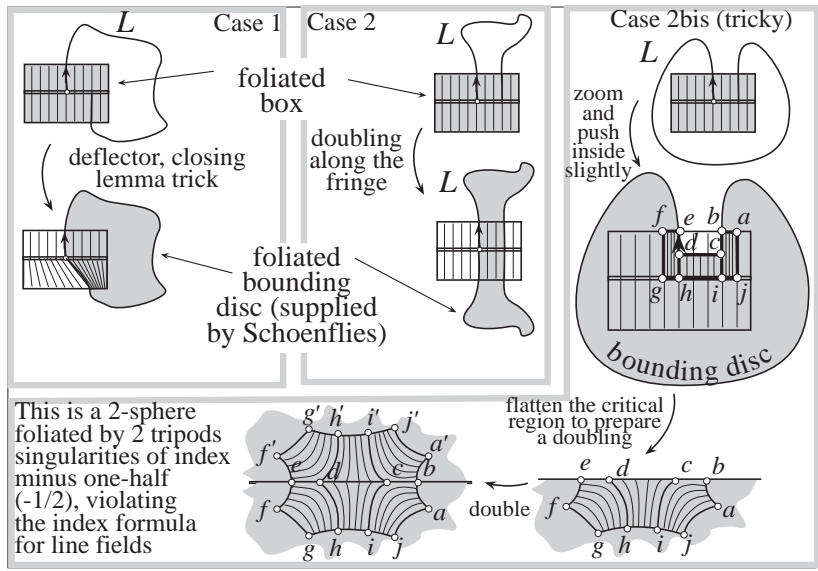,width=108mm}
  \caption{\label{David's_trick}
  %Artist view of
  Absence of recurrences for a leaf in a simply-connected surface}
\vskip-5pt\penalty0
\end{figure}

In the first case one can perturb the foliation within the box
(e.g., piecewise linearly) while creating a circle leaf (impossible
by Schoenflies).

\smallskip
{\footnotesize {\bf Very optional over-exhaustive case
distinctions.} In the other case one has a ``tongue shape''
whose double produces a foliated disc.
In fact
%it should be noticed that
in the second case there is
%also
a tricky subcase (Case 2bis on
Figure~\ref{David's_trick}) corresponding to the situation,
where
%%%%%%%%%%%
\iffalse the bounding disc for the Jordan curve starting from
$h$, say, and extended until its first impact, $i$, on a
cross-section of the foliated-box and closed-up by the unique
arc in the cross-section form $i$ back to $h$ \fi
%%%%%%%%%%%%%%%%
the bounding disc for the Jordan curve starting from the first
escapement $e$, say, of the oriented leaf $L$ from the
foliated-box $B$ and extended until its first impact, say $b$,
on the box $B$ and closed-up by the unique arc in $\partial B$
%in the cross-section
from $b$ back to $e$ (transverse to ${\cal F}$) contains the
foliated-box $B$. In this case we start by pushing the side
$\overline{eb}$ into the box up to position $\overline{cd}$.
Then we flatten the boundary of the bounding disc $D$ for the
Jordan curve $J$ (through $e,b,c,d,e$) near the critical
region (compare the figure), and finally we double $D$ with a
replica $D'$ yielding a foliated $2$-sphere having as unique
singularities two ``tripods'' singularities located at the
points $c$ and $d$. Such tripods singularities have an index
$j=-\frac{1}{2}$ each, yielding a total sum of $-1$,
disagreeing with the Euler characteristic of $S^2$. This
violates the Poincar\'e-{\Kerekjarto}-Hopf index formula for
line fields (cf. e.g., H. Hopf \cite[p.\,109 and Theorem II,
p.\,113]{Hopf_1946_1956}).

}

\iffalse
\begin{rem}\label{simplification:rem} {\rm This proof can be
%much
simplified, if one notices first that since the ambient
surface $S$ is simply-connected, any of its foliation is
orientable,
%and therefore
hence Cases 2 and 2bis need no examination.
%However
 Thus the above cases distinction is in
reality over-exhaustive, yet it
%may have the
%(very little)
%advantage of
avoids
%to
defining orientable foliations and the allied orienting
covers.}
\end{rem}
\fi

(c) The properness of leaves is clear in view of (b). That
leaves are closed sets can be derived from the fact that the
leaf-space is a (non-Hausdorff) \hbox{1-manifold}, which
follows directly from (b), as we shall recall later
(\ref{leaf-space-one-manifold}). Of course closedness can be
deduced also directly from (b): given a point $p$ not on the
leaf $L$, choose a foliated chart $U$ about $p$. If $U\cap L$
is empty we are done. If not then by (b) we see only a single
plaque of $L$ in $U$ so that we easily find an open set
$V\subset U$ containing $p$ but not intersecting $L$.

(d) As we shall not really need it, we leave as an exercise
the task of clarifying the meaning of running to infinity
(probably in terms of evasion from any compactum).
%Yet we maybe do not need it formally. Let us wait.
%
The last claim of (d) is somewhat harder to establish
(especially if one tries to delineate the broadest generality
in which such a separation holds true). Thus we reserve the
next section to a detailed discussion.
\end{proof}

\smallskip
{\small {\bf Optional semi-historical digression.} The sequel
may  lead to an interpretation of the following prose of
Haefliger-Reeb~\cite[p.\,120]{Haefliger_Reeb_1957}: {\it ``Le
th\'eor\`eme 1 qui suit est classique; sa d\'emonstration
repose sur le th\'eor\`eme de Jordan (dans une version
particuli\`erement facile \`a \'etablir); ...''} Of course
Jordan is here somehow blended with Schoenflies. Recall
incidentally that the nomenclature ``Schoenflies theorem'' for
the bounding disc property is a rather recent coinage (perhaps
first appearing in Wilder 1949, as noticed in Siebenmann 2005
\cite[p.\,651]{Siebenmann_2005}). At any rate what is relevant
to the sequel is that point (b) of Prop.~\ref{Alex_separation}
provides a local flatness allowing one to prove a version of
Jordan separation using only covering space theory. Thus we
match slightly with the version particularly easy to establish
mentioned in Haefliger-Reeb, albeit they probably rather had
in mind a mod 2 homology argument, as shows the sequel of
their text {\it ``...elle utilise donc essentiellement le fait
que le plan ${\Bbb R}^2$ est simplement connexe (ou plus
pr\'ecis\'ement que son premier nombre de Betti modulo 2 est
nul).''} However their sketched proof of their  Th\'eor\`eme 1
uses in fact Schoenflies and not merely Jordan separation
(recall Dubois-Violette's example).

}

\subsection{Polarized covering
%%trick
\`a la Riemann
and Jordan separation in the large}

%{\bf (D) Riemann electrical surgery}.

Given a hypersurface $H$ in a simply-connected manifold $M$,
it is intuitively clear that $H$ divides $M$, provided the
hypersurface is closed as a point-set. A possible strategy is
that any such hypersurface in a manifold
%(one is almost tempted to use the ``divisor'' jargon
%of algebraic geometry)
induces naturally a
%two-sheeted
double cover of $M$. When $H$ does not divides $M$ this
covering is connected, violating the simple-connectivity of
$M$.
This section details the above idea. First a:

\begin{defn}\label{hypersurface} {\rm A {\it (locally flat)
hypersurface} in a manifold is a (non-empty) subset $H$ such
that for any point $p\in H$ there is an open neighbourhood $U$
in $M$ and a homeomorphism of triad $h:(U, U\cap H,
p)\approx({\Bbb R}^n, {\Bbb R}^{n-1}\times\{0\},0)$. Since
$U\cap H$ divides $U$, we call $U$ a polarised chart. One has
a splitting $U=U_+\cup U_{-}$ in two local halves defined as
the closures in $U$ of the components of $U-(U\cap H)$. In the
sequel we shall refer to $U_{\pm}$ as being semi-charts.}
\end{defn}

For instance the ``open'' straight line $H=]-1,+1[\times\{0\}$
in ${\Bbb R}^2$ is a hypersurface,
%This set $H$ is however not
%closed, and incidentally
but does not separate the plane. This is why we restrict
attention to hypersurfaces, which are closed as point-sets. As
the terminology ``closed hypersurfaces''
%would be quite
%unfortunate
conflicts with the classical nomenclature ``closed manifolds''
%as meaning
(referring to compact borderless manifolds), some {\it ad hoc}
jargon is coined to disambiguate the double usage of
``closed'' in point-set vs. combinatorial topology:
%of
%manifolds:

\begin{defn}\label{divisor:def} {\rm A {\it divisor} in a manifold
is a hypersurface in the sense of (\ref{hypersurface}), whose
underlying set is closed as a point-set.}
\end{defn}

Now ``our'' polarization trick is the following mechanism:

\begin{prop}\label{Rieman-polarized-cover}
Given a divisor $H$ in a manifold $M$, there is a naturally
defined
%two-fold covering
double cover $M_H \to M$ (called the polarization of $M$ along
$H$), with the distinctive property that $M_H$ is disconnected
if and only if $H$ divides $M$.
\end{prop}

\begin{proof}
%We first describe the intuitive idea, and then
%describe the formal construction.
%
{\bf (1) Intuitive idea.}
%The intuitive idea is also
%two-folds.
First we can imagine that we cut $M$ along $H$ to obtain a
bordered manifolds $W$ with an involution $\sigma$ on the
boundary $\partial W$ telling one how to reglue the points to
remanufacture the manifold $M$ out of $W$. (We use here the
magic scissor of combinatorial topology, which instead of
%annihilating
deleting points rather duplicate them!) In particular one has
an {\it assembly} map $\alpha \colon W \to M$, which is
one-to-one except over $H$ where the fibers are two points
exchanged by $\sigma$. (Call $\sigma p$ the opposite of $p$.)
Then take $W'$ a replica of $W$, and denote by $p'\in W'$ the
twin copy of the point $p\in W$. In the disjoint union
$W\sqcup W'$ identify the point $p\in
\partial W$ with the opposite of its twin, i.e. $\sigma p'$
(where for simplicity we still denote by $\sigma$ the
involution on $\partial W'$). We define $M_H$ as the resulting
quotient space. It is not hard to show that the assembly maps
$\alpha  \cup \alpha' \colon W\sqcup W' \to M$ induce a map
$M_H \to M$ which is a covering projection.

\begin{figure}[h]
\centering
    \epsfig{figure=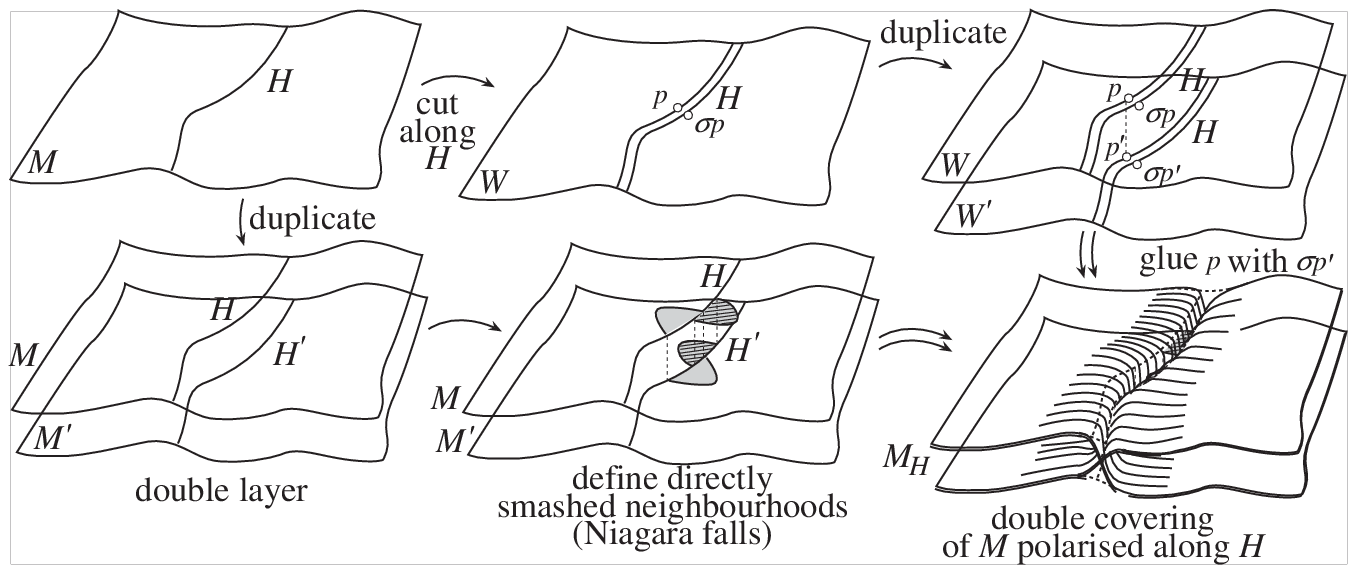,width=122mm}
  \caption{\label{Riemann's_trick}
  %Artist view of
  %Construction of
  The
  %2-sheeted covering
  double cover of a manifold $M$
  polarized along a hypersurface $H$}
\vskip-5pt\penalty0
\end{figure}

%{\bf (1) Intuitive idea.}

{\bf (2) Another viewpoint.} The above construction requires a
cleaner description of the cutting process. We can take a
slightly different approach. First take a copy $M'$ of $M$,
and define a new topology by splicing any polarized chart
$U=U_+ \cup U_{-}$ (cf. Def.~\ref{hypersurface})  into the two
``spliced'' sets $U_+\sqcup U_{-}'$ and $U_+'\sqcup U_{-}$.
The ``primes'' indicates that we push alternatively one of the
two halves of $U$ into the second layer $M'$. Further we would
like to identify the points $p$ in $U\cap H=U_+\cap U_{-}$
with their twins $p'$ so as to restore the locally Euclidean
character. Note that we are not merely redefining a new
topology on the (static) point-set $M\sqcup M'$, but really
doing a gluing on the two spliced charts which is easy
locally, yet maybe
%%difficult to do globally
problematic (at the non-metric scale). \iffalse ..Do we need
here generalised spaces in the sense of Grothendieck (yet then
unlikely to stay within the realm of classical manifolds...)
So we are in front of a little or maybe an insurmontable
difficulty. \fi

{\bf (3) Finding an issue.} Maybe the trick is as follows,
closer to the approach (1). We would like to formalise the
idea of cutting along a hypersurface. Thus we need first to
enrich $M$ by creating a replica for each point lying on $H$.
We try to think of such a point as a pair $(p,U_{\pm})$
consisting of a (classical) point $p$ of $H$ plus a preferred
half $U_{\pm}$ of a polarized chart $U$ about $p$. Since $H$
is closed (as a point-set), we may fix an atlas for $M$ such
that any chart meeting $H$ is a polarised chart (first cover
the hypersurface by polarized charts and then
%add
aggregate charts of the manifold $M-H$). Say that such an
atlas is polarised w.r.t. $H$. Given a polarised atlas $\cal
A$ (say a maximal one to kill any dependence upon anodyne
choice from the beginning) we define a new point-set $W$ as
consisting of all {\it filters}, in the following two senses:

\begin{defn}\label{filters} A filter $F$ of charts
(resp. of semi-chart) is a  nested sequences of charts
$U_i\supset U_{i+1}$ of $\cal A$ (resp. semi-charts, i.e.
halves of polarized charts of $\cal A$) whose common
intersection $\cap_{0\le i\le \omega} U_i$ is a unique point
(called the center of the filter).
\end{defn}

Declare two filters $F_1, F_2$ as {\it equivalent} if for any
member of the first $U_i\in F_1$ there is an element of the
second $V_j\in F_2$ such that $V_j\subset U_i$. It is easy to
check that this is an equivalence relation. Now define $W$ as
the set of equivalence classes of filters. Notice that there
are two equivalence classes of filters converging to a point
$p\in H$, whereas there is a unique class converging to a
point not on $H$. We have a map $\alpha\colon W \to M$
assigning to a filter its center and we endow $W$ with the
most economical topology making $\alpha$ continuous. Then it
%would only remain
looks easy to check that $W$ is a bordered manifold.
%, which looks
%rather obvious by construction.
As the map $\alpha$ is two-to-one above $H$, it gives  a
mapping $\sigma\colon
\partial W \to \partial W$ exchanging these two points. Now we
have all the necessary ingredients to conclude as in the first
step (1).

\smallskip
{\bf Proof of the distinctive property in
(\ref{Rieman-polarized-cover}).} [$\Rightarrow $] (NB: this is
the sense really needed for the corollary below). If $H$ does
not divide, then the bordered manifold $W$ is connected, and
so is a fortiori $M_H$ which is obtained by identifying $W$
with a replica $W'$. (Here and below, we use implicitly that
the interior of $W$ is naturally homeomorphic to $M-H$, and
the general fact that a bordered manifold is connected iff its
interior is.)

[$\Leftarrow$] Assume that $H$ divides $M$. Then $W$ is
disconnected, and then $M_H$ is disconnected as follows from
the construction. Indeed assume for (psychological) simplicity
that $W$ has two components $W_+$, $W_-$. Then $M_H$ results
from $W\sqcup W'=(W_+\sqcup W_-)\sqcup(W_+'\sqcup W_-')$ by
attaching $W_+$ with $W_-'$ and $W_-$ with $W_+'$ and
therefore $M_H$ has two components.
\end{proof}

This gives our sought-for:

\begin{cor}\label{Riemann-separation} A closed hypersurface $H$ (as a point-set!) in a
simply-connected manifold $M$ divides the manifold $M$.
\end{cor}

\begin{proof} If $H$ would not divide $M$, then the polarized
covering $M_H\to M$ is connected, violating the assumption
that $\pi_1(M)=0$.
\end{proof}

In particular this implies what we really wanted in
Prop.~\ref{Alex_separation}(d)

\begin{cor} Any leaf of a simply-connected
surface divides.
\end{cor}

\begin{proof} Point (b) of (\ref{Alex_separation}) implies
that the leaf is a hypersurface in the sense of
(\ref{hypersurface}), whereas point (c) of the same
(\ref{Alex_separation}) ensures that the leaf is a closed as a
point-set. Thus we conclude with (\ref{Riemann-separation}).
\end{proof}

%%%%%%%%%%%%%%%%%%%%%END OF THE CENSURE

%%%Optional
\subsection{More analogies and divergences
from Haefliger-Reeb}

Albeit we shall not use it, we can push-forward the analogy
with Haefliger-Reeb's theory. If $\cal F$ is a foliation on a
simply-connected surface $S$, then

(A) the leaf-space $S/ \cal F$
%of the foliation
is still a 1-manifold (generally non-Hausdorff), cf.
(\ref{leaf-space-one-manifold}) below. In our setting the
leaf-space needs not to be second-countable (equivalently
Lindel\"of, as manifolds are locally second-countable). So the
leaf-space is a non-Hausdorff 1-manifold with possibly long
``branches'' or also with possibly uncountably many branches
(consider e.g., the leaf-space of ${\Bbb L}^2$ slitted along
the closed set ${\Bbb L}_{\ge 0} \times  {\omega_1}$ and
foliated vertically or the Pr\"ufer type example depicted on
Figure~\ref{Train:fig}).

(B) However there is a little divergence with the metric case,
for now the projection $S\to S/ \cal F$ needs not to be a
(locally trivial) fibration. Indeed it is enough to consider
slitted long planes ${\Bbb L}^2-(\{0\}\times {\Bbb L}_{\ge
0})$ foliated vertically to see that the leaf-type can jump
erratically between the three open 1-manifolds (real-line,
long ray and long line). \iffalse Thus the issue that
%about
any leaf of a foliation of the plane admits a saturated
neighbourhood with a product structure breaks down in our
context. \fi
Another
%even more
perverse example is provided by the vertical foliation of the
Moore surface (cf. Figure~\ref{Train:fig}), where there is no
jump in the topological type of the leaves, yet the projection
$M\to M/{\cal F}$ is not a fibration. If it would then since
the base is ${\Bbb R}$ which is contractible the fibration
would be trivial (Feldbau-Ehresmann-Steenrod), and so the
total space would be ${\Bbb R}^2$ violating the non-metric
nature of the Moore surface $M$.

%%%%%%%%%%%%%%%%%%%%
\iffalse
Claim (A) requires a little justification based
by the following copied version of Th\'eor\`eme 1 in
Haefliger-Reeb \cite[p.\,120]{Haefliger_Reeb_1957}  attributed
to Poincar\'e-Bendixson and justified there (quite loosely)
via Jordan rather than Schoenflies(!), which is somewhat
abrupt in view of the Dubois-Violette example):

\begin{prop} \label{Haefli_Reeb_Thm1} ($\approx$ Haefliger-Reeb, Th\'eor\`eme~1) Given any foliated chart of a foliation of a
simply-connected surface the intersection with any leaf
reduces to the empty set or to a line.
\end{prop}

\begin{proof} The proof is the same as the one provided by
Figure~\ref{David's_trick} (plus the neighbouring
phraseology). Indeed if the leaf appears twice in the foliated
chart we can always produce a foliated disc.
\end{proof}

\fi

As in Haefliger-Reeb \cite[p.\,122]{Haefliger_Reeb_1957} the
fact that a leaf appears at most once in a foliated chart
(Prop.~\ref{Alex_separation}(b)) implies the:

\begin{cor}\label{leaf-space-one-manifold} The
leaf-space $V=S/ \cal F$ of a foliated simply-connected
surface $S$ is a one-dimensional manifold (generally
%ically
non-Hausdorff), which
%moreover
is simply-connected (i.e.,  $\pi_1(V)$ is trivial, or
equivalently $V$ is divided by any puncture).
\end{cor}

\begin{proof} (Just a translation of
Haefliger-Reeb's argument.) To show that $V=S/ \cal F$ is a
$1$-manifold, it is enough to check that any point $z\in V$
admits an open neighborhood homeomorphic to the
%numerical line
number-line ${\Bbb R}$. Let $\pi\colon S \to V$ the canonical
projection (associated to the equivalence relation $\rho$ of
appurtenance to the same leaf); the leaf $\pi^{-1}(z)$ meets
at least one foliated chart $O_i$. The equivalence relation
induced by $\rho$ on $O_i$ is, by
Prop.~\ref{Alex_separation}(b), the relation $\rho_i$
corresponding to the partition in parallel lines. Thus
$\pi(O_i)$ which is an open neighbourhood of $z$ (since $\rho$
is an open equivalence relation\footnote{This means that the
saturation of any open set is open, or what amounts to the
same that the canonical projection is open.}), is homeomorphic
to $O_i/\rho_i$, that is to the numerical line ${\Bbb R}$.

Regarding the second assertion (simple-connectivity of the
leaf-space) we  again follow Haefliger-Reeb. The complement of
each leaf $L$ (a closed subset of $S$) has two components
(Prop.~\ref{Alex_separation}(c)(d)); hence the complement of
any point of $V$ has also two components.
%(Exercise: Why
%exactly?)
This
%%property is equivalent
is equivalent to the simple-connectivity of $V$ (compare lemma
p.\,113 in Haefliger-Reeb \cite{Haefliger_Reeb_1957} which is
a special case of (\ref{Rieman-polarized-cover}), or formulate
an appropriate exercise in algebraic topology using
Seifert-van Kampen, or Mayer-Vietoris).
\end{proof}

%\section{Infinite cyclic groups}

\subsection{Hausdorffness of the leaf-space in the $\omega$-bounded case}

By the preceding section,
%we know that
the leaf-space of a foliated simply-connected surface is
%always
a 1-manifold. In the metric case, the
%one of the main catalysator responsible of the
non-Hausdorff\-ness of the quotient
%is
is mostly catalyzed
%the presence of
by Reeb components. Heuristically it is rather evident that
there is no long Reeb components. More precisely if one
assumes that there is a long transversal, then it is easy to
deduce a continuous map from the long ray ${\Bbb L}_+$ to the
reals ${\Bbb R}$ which is not eventually constant (by looking
how the leaves emanating from a point on the transversal
intercept a cross-section of a foliated chart). This gives
some weight to the:
\begin{conj}
The leaf-space of any foliated simply-connected
$\omega$-bounded surface is  Hausdorff (which is probably
always the long-line).
\end{conj}

Here is an outline of the difficulty appearing in an attempt
of proof. Given two leaves $L_1, L_2$ one would like to
separate them. Of course if there is a leaf $L$ which divides
$L_1$ from $L_2$ in the sense of ``Jordan'' that is the $L_i$
belong to two distinct components of $M-L$, then those
components (projected in the
%quotient
leaf-space) will
separate  $L_1$ from $L_2$ in the sense of Hausdorff, and we
are finished. Now we would like to show that under the
$\omega$-boundedness condition, there is such a leaf $L$.

%%%%%%%%%%%%%%%%
%%%%%ARGUMENT OF BAILLIF (censured as not completely convincing)
%%%%%%%%think with an amoeba
\iffalse
Such a leaf should exist by David's result (appearing
in BGG2) that in a Type I manifold, a doubly long leaf has a
foliated neighborhood which is a long tube $D \times L$, where
$D$ is a
%disk.
$1$-disc (interval). An $\omega$-bounded surface being Type I
and sequentially compact, if you take out one of the
neighboring long lines, then it should cut the surface in two,
using the fact that the surface is 1-connected (I guess ?).
OKAY this is true by Prop.~\ref{Alex_separation}(d), but this
does not imply that both leaves lye in different components of
the removed one. \fi

Probably more is true. Recall that a divisor is a (locally
flat) hypersurface which is closed as a point set. In a
simply-connected foliated surface any leaf is a divisor which
is not a circle (\ref{Alex_separation}). Let us call {\it
pseudo-line}
%(or just line)
a connected divisor in a simply-connected surface (say an {\it
absolute}, for short) which is not the circle. (It can be the
real-line, the long ray or the long line). Since any divisor
in a simply-connected $M^2$ divides
(\ref{Riemann-separation}), given 3 pseudo-lines in an
absolute, either one of them divides the two others or no
lines separates the remaining two. Call
%the three lines
the first configuration {\it parallel}, and the second  an
{\it amoeba}. In the latter case the 3 pseudo-lines bound a
bordered subregion namely the triple intersection of those
halves of the $L_i$ containing the remaining two pseudo-lines
$L_j, L_k$ ($j,k \neq i$). Then we have the following
strengthening of the conjecture:
\begin{conj}
Any $3$ leaves of an $\omega$-bounded foliated absolute are
parallel.
\end{conj}

%For the same purpose (of showing that between two leaves in an
%$\omega$-bounded 1-connected surface there is always another
%one separating the two others) I found this morning
Here is a somewhat more theological argument supporting this
conjecture, which is perhaps not the most elementary, yet
%suggesting
adumbrating a broader perspective. If not, then the three
lines $L_1,L_2,L_3$ are in the configuration of an amoeba.
%(i.e. two parallel lines $L_1,L_2$
%and the third $L_3$ in between like a U-shape, that is without
%dividing $L_1$ from $L_2$).
Then one can double the ``amoeba'' domain bounding the three
curves $L_i$ to get a sort of long pant. It is easy to show
that the latter pant is $\omega$-bounded and of Euler
characteristic $-1$ (for instance with Mayer-Vietoris or by
using the fact that the characteristic of a bagpipe is equal
to that of the bag, cf. \cite[Lemma 4.4]{Gab_2011_Hairiness})
Then conclude with the following conjecture
(\ref{Euler-obstruction}) which has probably some independent
interest (to be compared to the hairiness
%arXiv
%note by Alex, 2011, Hairiness of omega-bded...)
note \cite{Gab_2011_Hairiness} for flows).

\subsection{Missing Euler obstruction}

\begin{conj}\label{Euler-obstruction}
An $\omega$-bounded surface with negative Euler characteristic
$\chi<0$ cannot be foliated.
\end{conj}
This
%looks to be a rather
is an intriguing version of the Euler-Poincar\'e obstruction.
We think by experience that it must be true, yet the proof
looks more involved than in the flow case (where the
hypothesis was slightly different, namely non-zero $\chi$, cf.
\cite{Gab_2011_Hairiness}). The example of ${\Bbb L}^2$ with
$\chi=1$ shows that the condition $\chi\neq 0$ is not enough
to obstruct foliability.
%Of course let me know if you have a counter-example
%or if you find a proof.

\section{Poincar\'e-Bendixson arguments}

In this section, we derive from Poincar\'e-Bendixson's
trapping argument under dichotomy (alias Jordan separation),
several {\it universal} obstructions to transitivity
%in the sense that they are
not confined to the metric case. The complexity (of the
proofs) raises with the topology quantified by the rank of the
$\pi_1$. The method is basically a reduction to the dichotomic
case by passing to
%suitable
double covers, with Poincar\'e-Bendixson's method acquiring
more punch when combined with Riemann's branched covers.
%permitting to control the dichotomy of the total
%space.
When the total space fails to be dichotomic, some deeper
versions of Poincar\'e-Bendixson (like those of Kneser,
Markley, etc.) describing the dynamics on the Klein bottle
enter into the arena.
%(our first subsection is devoted to
%refresh our memories on those classical theorems).
Ultimately we derive an almost complete classification of
%%%those
finitely-connected metric surfaces
%of
%finite-connectivity
which are transitively-foliated. Besides, intransitivity
transfers non-metrically, being conserved to any non-metric
degeneracy of a finitely-connected metric surface provided its
invariants (Euler character, ends-number and indicatrix) are
kept unaltered.
 The
 %concept
 %jargon of the
soul concept formalizes this idea while unifying all
%individual
results under a single perspective.
%yet we did not recanted
%everything under this view as it looks desirable to give
%elementary arguments in special cases.

\subsection{Dynamics on the bottle (Kneser, Peixoto, Markley,
Aranson, Guti\'errez)}

Beside the basic Poincar\'e-Bendixson obstruction, we require
several other classic theorems describing
%for instance
the
dynamics on the Klein bottle. Those
%%requires
rely on  some
%clever
magic arguments
%%%%%(often hard to formalize)
%essentially
%extensions of
close to the Poincar\'e-Bendixson trapping, yet deviating from
it inasmuch as they exploit a
%the existence of a
global cross-section.
%We recommend to
%%%Maybe skip this section on a first reading and refer to it
%%%when really required.

\begin{lemma} (Kneser 1924 {\rm \cite{Kneser24}})\label{Kneser}
Any foliated Klein bottle has a circle leaf.
\end{lemma}

\iffalse
\begin{proof}
Compare Kn
\end{proof}
\fi

\begin{cor}\label{Klein-foliated-intransitive}
The Klein bottle $\Klein$ is foliated-intransitive.
\end{cor}

\begin{proof} By  Kneser (\ref{Kneser})
there is a circle leaf $K$. If it divides the bottle ${\Bbb
K}$ we are finished. Else cut the surface along $K$ to get a
{\it connected} compact bordered surface with $\chi=0$ and
either one or two contours. By classification
(\ref{Moebius-Klein-classification}) these are resp. a
(compact) M\"obius band or an annulus. Deleting the boundary
gives in both cases surfaces with $\pi_1\approx {\Bbb Z}$, and
conclude with (\ref{infinite-cyclic-group}) below.
\end{proof}

\begin{lemma}\label{Markley}
(Markley 1969, Aranson 1969, Guti\'errez 1977) The Klein
bottle $\Klein$ is flow-intransitive.
\end{lemma}

\begin{proof} The intransitivity of ${\Bbb K}$
was first established by Markley 1969 \cite{Markley_1969}
(independently Aranson 1969), yet the argument of Guti\'errez
1978~\cite[Thm~2, p.\,314--5]{Gutierrez_1978_TAMS} seems to be
%inter\-continentally
%universally
%%%%%recognized as
the ultimate simplification. We recall it for
completeness.

By a lemma of Peixoto there is a global cross-section $C$ to
the flow (transverse circle). This circle is two-sided (its
tubular \nbhd{ }being oriented by the flow-lines is an annulus
not a M\"obius band). Also $C$ is not dividing (a separation
impeding transitivity). Cutting $\Klein$ along $C$ yields a
connected bordered surface $W$ with 2 contours with $\chi$
unchanged equal to $0$. By classification
(\ref{Moebius-Klein-classification}), $W$ is an annulus.
Orient its 2 contours $C_1, C_2$ as the boundary of $W$, and
the original surface is recovered by an orientation-preserving
homeomorphism $h\colon C_1\to C_2$ (which we may assume, in
reference to a planar model say $W=\{z\in {\Bbb C}: 1 \le
\vert z \vert \le 2 \}$, to be a reflection about the vertical
axis on $C_1$ followed by a radial map $C_1\to C_2$). We
denote $h(p)=p'$, just by a prime.

\begin{figure}[h]
\centering
    \epsfig{figure=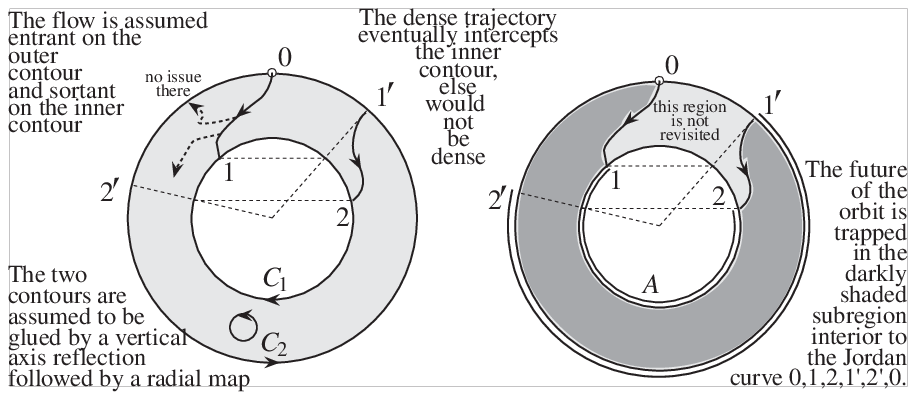,width=122mm}
\vskip-10pt\penalty0

  \caption{\label{Gutierrez:fig}
  Flow intransitivity of the Klein bottle (Guti\'errez's proof)}
\vskip-5pt\penalty0
\end{figure}

Assume the flow entrant on the outer contour $C_2$ and sortant
on the inner contour $C_1$. A dense orbit must cross $C$, and
w.l.o.g. we may suppose that the forward-orbit
%%%$F$
is dense. Let $0$ be a point on $C_2$ whose forward-orbit is
dense in $\Klein$. Since the inner contour $C_1$ has a
foliated collar where the flow is sortant,
%by denseness
the dense orbit must eventually reach this collar and so
intercepts $C_1$ at some point, say $1$. Then reflect
vertically $1$ and map it radially to get $1'\in C_2$. As
before (denseness) the subsequent trajectory must again
intercept $C_1$, at some position say $2$. Consider $h(2)=2'$,
and notice that the subsequent orbit is trapped inside the
%shaded
dark subregion of Figure~\ref{Gutierrez:fig}, violating
%its
%presupposed
denseness. Indeed, the future of $2'$ will be an interception
with the arc  $A=\overline{1,2}\subset C_1$ determined such
that the Jordan circuit $0,1,A,2,1',2',0$ [=flowing forwardly
from $0$ to $1$ , then following the arc $A$, next flowing
backwardly from 2 to $1'$ and finally moving injectively on
the circle along the orientation specified by the triple
$1',2',0$] is null-homotopic in the annulus $W$, so bounds a
disc $D$ in $W$, which is the required trapping region, since
$h(A)\subset D$.
\end{proof}

\subsection{Dichotomy obstructs oriented transitivity}

%We start by noticing
The cornerstone
%of the sequel
is an oriented foliated
%version
avatar of Poincar\'e-Bendixson:

\begin{lemma}\label{Poinc-Bendixson_many} An oriented
%(or orientable)
foliation on a dichotomic surface has no dense leaf. Further
an
%obvious
addendum is that no finite collection of leaves can
be dense.
\end{lemma}

\begin{proof}
This is the
%standard
trapping argument of Poincar\'e-Bendixson, best understood by
drawing a figure:

\begin{figure}[h]
\centering
    \epsfig{figure=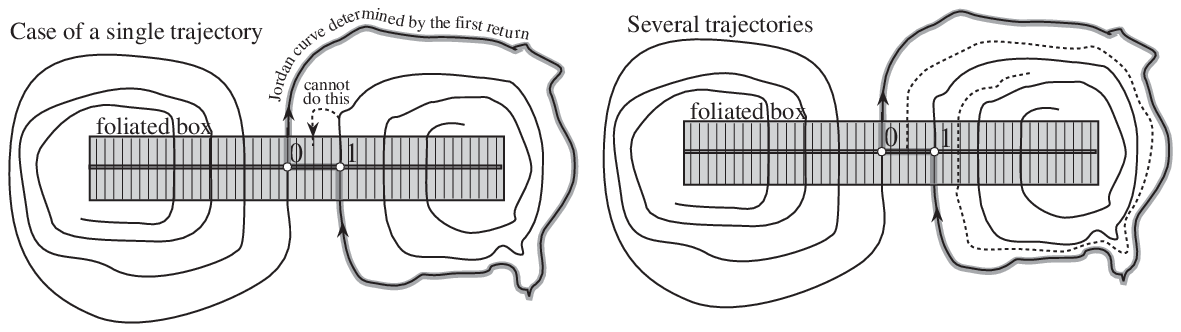,width=122mm}
\vskip-15pt\penalty0

  \caption{\label{poinben:fig}
  Foliated Poincar\'e-Bendixson argument}
\vskip-5pt\penalty0
\end{figure}

Assume that $L$ is a dense leaf. Choose on it a point (called
$0$) and about $0$ a foliated box $B$. Since $L$ is dense it
must reappear in the box $B$. W.l.o.g. assume this to be a
forward interception w.r.t. the orientation (else reverse it).
Call the first return to the cross-section $1$. The piece of
leaf from 0 to 1 closed by the cross-sectional arc from $1$
back to $0$ is a Jordan curve $J$. By  dichotomy $J$ divides
the surface, trapping the future of the trajectory. In
particular the next return to the cross-section, call it $2$,
occurs to the right of $1$. By induction it follows that the
successive returns occur in a order preserving fashion. On
that
%classical
picture on can safely superpose several trajectories
(=oriented leaves) so as to deduce the addendum.
\end{proof}

\begin{rem} {\rm Lemma~\ref{Poinc-Bendixson_many}
%provides another proof
reproves that a 1-connected surface lacks dense leaves (a
foliation on a 1-connected manifold being automatically
orientable (\ref{orienting:2-fold-covering})).

%Application: f
\subsection{Foliated surfaces with infinite
cyclic group}

Our initial intention was to apply the Haefliger-Reeb theory
to the following proposition, by passing to the universal
cover while arguing that the lifts of the dense leaf cannot be
dense. (We were not able to present a decent proof and crudely
put, met some difficulties in showing that the generating
deck-translation acts as translations or  gliding reflections
of ${\Bbb R}^2$.)

\begin{prop} \label{infinite-cyclic-group} A (non-metric) surface
whose fundamental group is infinite cyclic
%cannot support
lacks a transitive foliation (i.e., with at least one dense
leaf).
\end{prop}

\begin{proof} If orientable, the surface is
dichotomic (\ref{dichotomy:orient_plus_inf_cycl}) and we
conclude with (\ref{Poinc-Bendixson_many}) after orienting the
foliation up to passing to the
%2-fold
double cover (\ref{orienting:2-fold-covering}). As the group
is ${\Bbb Z}$, it stays so under finite covering, which also
preserves orientability. If  not orientable, the surface
constructs its orientation double cover
(\ref{indicatrix-orient-covering}) and we are reduced to the
previous
%discussion.
case.
\end{proof}

\iffalse
\begin{rem} {\rm Lemma~\ref{Poinc-Bendixson_many}
%provides another proof
reproves that a simply-connected surface cannot have a dense
leaf (a foliation on a simply-connected
%surface
manifold is automatically orientable).\fi \iffalse And then we
can either:

(1) use Jordan separation and the above Poincar\'e-Bendixson
lemma.

(2) or alternatively apply the closing argument (of BGG2 \cite{BGG2}, cf.
also the simple case $1$ of Figure~\ref{David's_trick})
without having to worry about the more complicated Cases~2 and
2bis, which are impossible for an oriented foliation. [So at
this stage I understood that David's proof in BGG2 is indeed
complete, yet one has to insist that the foliation can be
globally oriented. Thus with this preliminary we can simplify
somewhat the proof surrounding Figure~\ref{David's_trick},
while in particular avoiding the index formula for line
fields.]
\fi
}
\end{rem}

\subsection{Free groups of rank two under dichotomy}

Recall as a motivation Dubois-Violette's transitive
(indeed minimal) foliation on the thrice punctured plane with
 fundamental group  $F_3$ (free
 %group
 on three letters).
Also the Kronecker torus punctured once shows a minimal foliation
on a surface with $\pi_1\approx F_2$ (free of rank 2). The
following result
%(which is
(surely well-known in the metric case, albeit we did not
 checked
 %if it occurs in
 the literature carefully)
 shows the sharpness of  those examples as
 %materializing
 having the minimum complexity for the
fundamental group permitting a
%``labyrinthic''
dense leaf.
%trajectories.
In particular 3 is the minimal number of punctures required to
the plane to manufacture a {\it labyrinth}
%(i.e.
(=foliation
with a dense leaf).

\begin{prop}\label{dichotomic-free-of-rank-2:prop}
A dichotomic surface
%whose $\pi_1$ is free of rank $2$
with fundamental group $\pi_1$  free of rank $2$ is
foliated-intransitive.
%(under
%a foliation).
\end{prop}

This applies
%for instance
to the twice-punctured Moore
surface, and more generally to any
%two-points deletion in a
twice-punctured  simply-connected surface,
%%%%$S$,
in view of
(\ref{puncturing}).

\smallskip

\begin{proof}
If the foliation is orientable, then  the foliated version of
Poincar\'e-Bendixson (\ref{Poinc-Bendixson_many}) concludes.
Otherwise, pass to the 2-fold orienting cover
%ing rendering the foliation orientable
(\ref{orienting:2-fold-covering}), which by the branched
covering argument of (\ref{dicho-covering-is-dicho}) is still
dichotomic, reducing again to the Poincar\'e-Bendixson
obstruction (\ref{Poinc-Bendixson_many}).
\end{proof}

\subsection{Free groups of rank two (non-orientable cases)}

The above (\ref{dichotomic-free-of-rank-2:prop}) does not
apply to $\Moe_*$ (punctured M\"obius band), which is ${\Bbb
R}P^2_{**}$ (twice punctured projective plane), which has
$\pi_1\approx F_2$ (\ref{puncturing}), but
%%is
not dichotomic
%. Notice indeed the:
by (\ref{dicho-implies-orientable}).
Yet the
%same
method of branched covers still applies:

\begin{prop}\label{Klein-Weichold} The twice-punctured
projective plane
${\Bbb R}P^2_{**}=\Moe_*$ is foliated-intransitive.
\end{prop}

\begin{proof} {\it Orientable case.} If the foliation
is orientable, take a compatible flow on ${\Bbb R}P^2_{**}$
(\ref{Kerek-Whitney:thm}). By a standard method \`a la Beck
(\ref{Beck's_technique:extension:lemma}) the flow extends to
${\Bbb R}P^2$. Passing to the universal cover $S^2$,
Poincar\'e-Bendixson obstructs transitivity.

{\it Non-orientable case.} If not, the foliation determines a
%2-fold
double orienting cover $p\colon \Sigma \to {\Bbb R}P^2_{**}$
(\ref{orienting:2-fold-covering}). Looking around the
punctures we can by Riemann's trick
(\ref{Riemann:branched-cover}) compactify this map to a
branched covering $\Sigma^*\to {\Bbb R}P^2$ (ramified at $R$ a
sublocus of the punctures). Thus
$$
\chi(\Sigma^*)=2\chi({\Bbb R}P^2)-\deg(R).
$$
Since  $\deg(R)\le 2$, $\chi(\Sigma^*)\ge 0$. Since $\Sigma^*$
is connected we have also $\chi(\Sigma^*)\le 2$.

If $\chi(\Sigma^*)=2$, then $\Sigma^*\approx S^2$ and
%we have
%the
Poincar\'e-Bendixson
%obstruction
%prompted by dichotomy.
concludes.

If $\chi(\Sigma^*)=1$, we
%are concerned  with
have ${\Bbb R}P^2$ and argue as above (orientable case).

If $\chi(\Sigma^*)=0$, then we have either the Klein bottle
${\Bbb K}$ or the torus ${\Bbb T}^2$. In the first case, lift
the foliation to $\Sigma$, take a compatible flow
(\ref{Kerek-Whitney:thm}) and ``extend'' it to $\Sigma^{*}$
(\ref{Beck's_technique:extension:lemma}), violating  the
flow-intransitivity of Klein ${\Bbb K}$
%(Markley (1968),
%ref. as in \cite{GabGa_2011}).
(\ref{Markley}).

The toric case requires a separate argument.
%ation.
We have the
branched covering ${\Bbb T}^2=\Sigma^*\to \proj$ ramified at
%at most
two places, as $\deg(R)= 2$.
%(the value $0$ being trivially
%ruled out).
Exchanging sheets induces an
%allied
%``hyperelliptic''
involution  $\sigma$ of the torus which is orientation
reversing (the quotient being non-orientable) with
%at most
two fixed points.
%Such a bird has never been observed
%in nature, and can be ruled out by varied ways.
%%%
One obstruction to this is geometric
%and amounts
amounting essentially to the Klein-Weichold (1876--1883)
classification of orientation reversing involutions on
oriented closed surfaces (relevant to the real algebraic
geometry of curves). In fact we merely need the very
%%foundational
basic fact,
%%%of that theory,
asserting that the linearization in the small of such an
involution near a fixed point is a symmetry about a line,
violating the isolated nature of the above fixed points.
(Formal proof of this in the topological case came slightly
later with the era of Schoenflies, Brouwer, \Kerekjarto.)
%
%
%%%%%%%%%%%%%%%%CENSURED AS LOOKS STUPID
\iffalse
{\footnotesize Alternatively there is
%a somewhat here artificial
a clumsy algebraic way using the Lefschetz trace formula:
$$
2=\chi({\rm Fix}(\sigma))=\sum_{i} (-1)^i Tr (H_i(\sigma)).
$$
The second trace is $-1$, as $\sigma$ reverses orientation.
Writing the matrix of $H_1(\sigma)$ w.r.t. the canonical basis
$e_1, e_2$ of $H_1(T^2)$ as $M=\begin{pmatrix} a &b \cr c & d
\end{pmatrix}$, with coefficients in ${\Bbb Z}$. The Lefschetz
relation writes $a+d=-2$ (L). Since $\sigma $ is involutive
$M^2$ is the identity; hence (1) $a^2+bc=1$, (2)
$ab+bd=0=b(a+d)$, (3) $ac+cd=0=c(a+d)$, (4) $cb+d^2=1$.

From (2),(3) and (L) it follows that $b=c=0$. So by (1),(4)
$a^2=1$ and $d^2=1$. By (L) it follows $a=d=-1$

But that is not all. Using the Lie group structure of the
torus we have the so-called Pontrjagin product. The latter can
be used to show that what $\sigma$ does on the $H_1$
determines its action on the $H_2$. In our situation we have
$$
H_2(\sigma)(e_1\star e_2)=H_1(\sigma) (e_1) \star H_1(\sigma)
(e_2)=(-e_1)\star(-e_2)=e_1\star e_2,
$$
by bi-linearity. Since $e_1\star e_2=[T^2]$ is the fundamental
class, this violates the orientation reversion. } \fi
\end{proof}

Beside ${\Bbb R}P^2_{**}=\Moe_*$, another
%%%noteworthy
specimen
with $\pi_1=F_2$ is $\Klein_*$ (punctured {\it Klein bottle},
i.e. non-orientable closed surface with $\chi=0$, so ${\Bbb
K}=S_{2c}$ is also the sphere with $2$ cross-caps). Again we
expect a similar result:

\begin{lemma}
The punctured Klein bottle $\Klein_*$ is
foliated-intransitive.
\end{lemma}

\begin{proof} We break the argument in two parts; the first
being a repetition of the method employed so far, which
in the present case seems sterile:

\smallskip
{\footnotesize

{\bf The usual method foils.} If the foliation is orientable
we are done by the flow-intransitivity of Klein $\Klein$
%(due to Markley).
(\ref{Markley}). If not the foliation defines its oriented
2-fold covering $\Sigma \to \Klein_*$
(\ref{orienting:2-fold-covering}), which we compactify into a
branched covering $\Sigma^* \to \Klein$. In particular:
$$
\chi(\Sigma^*)=2\chi(\Klein)-\deg(R).
$$
As $0\le \deg(R)\le 1$, it follows $-1\le \chi(\Sigma^*)\le
0$. Hence $\Sigma^*$ is either $\Klein$ or $T^2$ if $\chi=0$
or $N_3$ the sphere with $3$ cross-caps when $\chi=-1$.

The case of $\Klein$ is easily ruled out by Markley's
intransitivity (\ref{Markley}).

The other cases are harder. In the toric case $\deg(R)=0$, so
that the compactifying covering is unramified. This means that
about the puncture the foliation is oriented, etc.

}

\smallskip
{\bf A new trick is required.} Maybe a more efficient argument
is to use the index formula for line-fields (Poincar\'e,
Bendixson, \Kerekjarto, Hopf, etc.). Since there is a unique
singularity, the index at the puncture is zero. Thus the
foliation extends to $\Klein$. (If the singularity looks like
a letter ``X'' with opposite hyperbolic sectors and opposite
focus-type sectors with leaves converging to the puncture,
then there is no such extension! Yet such a scenario impedes
transitivity due to the focusing sectors.) Once the foliation
is extended, conclude with Kneser 1924 (\ref{Kneser}) or
rather its corollary (\ref{Klein-foliated-intransitive}).
\end{proof}

Uniting
%Unifying
%together
the forces of the two previous propositions we
%may
deduce:

\begin{lemma}\label{rank_two_non-orientable:intransitive}
A non-orientable metric surface with $\pi_1=F_2$ is
intransitive.
\end{lemma}

\begin{proof} Such a surface is homeomorphic
either to $\proj_{**}$ or $\Klein_*$ by
(\ref{Kerekjarto:non-orient}).
\end{proof}

The ultimate generality is to
%drop
relax the metric proviso:

\begin{theorem}\label{intransitivity:non-orient_rank_2}
A non-orientable surface with $\pi_1=F_2$ is intransitive.
\end{theorem}

\begin{proof} The idea is
%again
to reduce to the metric case via an appropriate exhaustion. If
$L$ is a dense leaf then $L$ is either ${\Bbb R}$ or the
long-ray ${\Bbb L}_+$ (cf. (\ref{separability}) below). Up to
deleting the long-side of $L$ (which does not affect the
assumption made on our surface $M$ by
(\ref{deleting:a_closed_long_ray})) we may assume that $L$ is
the real-line. Choose now a $\pi_1$-calibrated exhaustion
(\ref{calibrated-exhaustions:lemma})
$M=\bigcup_{\alpha<\omega_1} M_{\alpha}$ by Lindel\"of
subregions  with $\pi_1(M_{\alpha})\approx\pi_1(M)=F_2$. Since
$L$ is Lindel\"of, there is $\beta<\omega_1$ such that
$M_{\beta}\supset L$; and $L$ being dense in $M$ it is a
fortiori so in $M_{\beta}$. Since $M$ is non-orientable, it
contains  a one-sided Jordan curve $J$
(\ref{orientability-in-terms-of-Jordan-curves}) (whose tubular
neighbourhood $T$ is a M\"obius band). Since $T\supset J$ is
Lindel\"of we may assume that $M_{\beta}\supset T$ as well,
violating the metric case
(\ref{rank_two_non-orientable:intransitive}) of the theorem.
\end{proof}

This applies to the Moore surface $M$ with two cross caps,
denoted $M_{2c}$, as well as to $M_{*,c}$ the Moore surface
punctured once and cross-capped once (apply
(\ref{cross-capping}) and (\ref{puncturing})).
%
%
%In the next section we study up to which extent those results
%are sharp, and how they can sometimes be sharpened.
The next section studies the sharpness of those results, while
giving some sporadic extensions to groups of rank 3.

\iffalse It is easy to check that the above results are sharp
by modifying the Dubois-Violette example. More precisely one
can inflate the punctured point of the thorn-singularity to a
punctured disc and then cross-cap the border (cf.
Figure~\ref{cross-cap:fig} below). (Note that a tripod
singularity is generated, which has a saddle connection.) Thus
we see that the following non-orientable surfaces with
$\pi_1=F_3$ are transitive: $S^2_{3*,1c}$ (3 punctures, 1
cross-cap), $S^2_{2*,2c}$ (2 punctures, 2 cross-caps),
$S^2_{1*,3c}$ (1 punctures, 3 cross-caps).
%Loosely speaking, the moral seems to
%be that
Intuitively, a cross-cap has exactly the same dynamical effect
as a puncture.
%Finally
Doing more punctures or cross-caps only improves the
transitivity issue (just puncture outside a dense leaf, which
cannot fill the whole manifold).

\begin{figure}[h]
\centering
    \epsfig{figure=crosscap.eps,width=82mm}
  \caption{\label{cross-cap:fig}
  Non-orientable avatars of Dubois-Violette}
\vskip-5pt\penalty0
\end{figure}

\fi

%(In)transitivity
\subsection{The monolith of finitely-connected metric surfaces}

This section aims to classify
%completely
those
finitely-connected metric surfaces which are transitive (and
those which are not). In the subsequent section we deduce
non-metrical transfers of intransitivity.

In view of \Kerekjarto{ }(\ref{Kerkjarto:end}) any
finitely-connected metric surface is homeomorphic to a
finitely-punctured closed surface. Hence by the M\"obius
%-Jordan-Klein-von Dyck-Heegaard-Dehn
{\it et al.} classification
(\ref{Moebius-Klein-classification}) any such surface derives
from the sphere $S:=S^2$ through iteration of the three
operations (1) handle surgery, (2) cross-capping (3)
puncturing. Thus we can tabulate a ``monolith'' for all such
surfaces  (Figure~\ref{monolith:fig} below), where
right-arrows are {\it puncturing}
%puncturation
(denoted by a stared subscripts ``$_{*}$'', e.g. $S_*={\Bbb
R}^2$), up double-arrows are {\it handle attachments}
($\Sigma_g$ denoting the orientable closed surface of genus
$g$), and left-squig-arrows are {\it cross-caps} (denoted by
subscripts ``$_{c}$''). Boldface fonts denote the rank of the
(fundamental) group when it is free. (Given a rank there are
only finitely many
%(finitely-connected
metric surfaces with the prescribed group.) The exotic arrows
(not fitting with the hexagonal lattice) arise from the
well-known relations in the monoid of closed surfaces
(under-connected sum) inherent to the classification theorem
(\ref{Moebius-Klein-classification}) in term of $\chi$, and
the orientability character. (For instance attaching a handle
to a non-orientable surface amounts to 2 cross-caps, both
decreasing $\chi$ by 2 units.)

\begin{figure}[h]
\vskip-35pt\penalty0 \centering
    \epsfig{figure=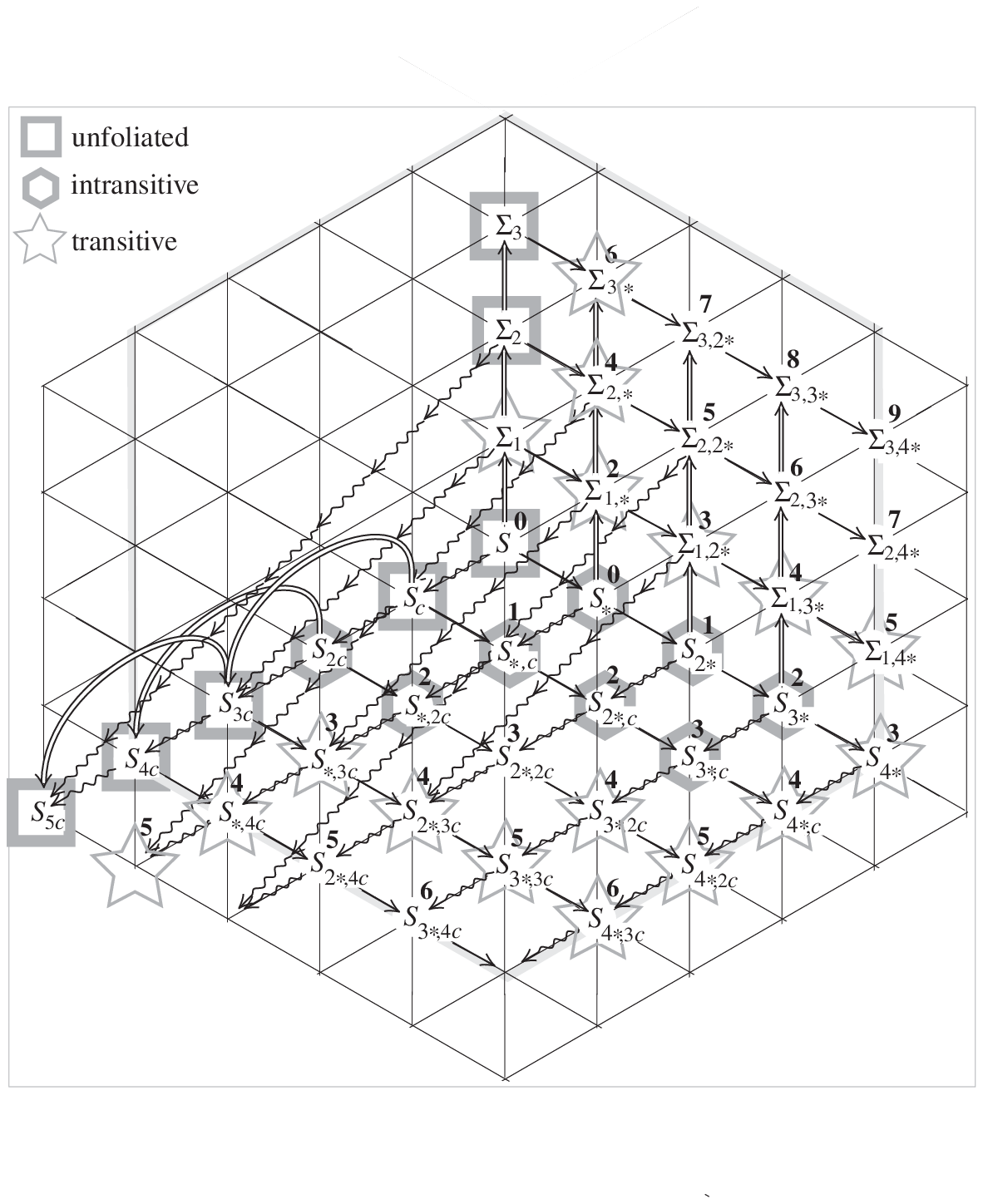,width=132mm}
\vskip-65pt\penalty0
  \caption{\label{monolith:fig}
  The monolith of finitely-connected metric surfaces}
\end{figure}

{\it Squares} indicate those surfaces which cannot be foliated
(Euler-Poincar\'e obstruction $\chi\neq 0$). {\it Hexagons}
show surfaces which are foliated-intransitive in view of
previously listed obstructions (\ref{Poinc-Bendixson_many}),
(\ref{infinite-cyclic-group}),
(\ref{dichotomic-free-of-rank-2:prop}),
(\ref{rank_two_non-orientable:intransitive}) and
(\ref{intransitivity:new-obstruction}) below. {\it Stars} show
surfaces which are foliated-transitive
%(by varying the
%Dubois-Violette construction, as in the previous section).
(as discussed below).

\def\wood{woodpecker}

\begin{lemma} If a surface is transitive, then so is its punctured
version (just puncture outside a dense leaf).
\end{lemma}

Thus we need only to establish the transitivity of ``minimal''
models with respect to puncturing. For instance $\Sigma_1$ the
torus is transitive (Kronecker foliation) and this propagates
right-down (on Figure~\ref{monolith:fig}).

\subsection{Transitive examples via surgery (Peixoto, Blohin)}

To construct transitive foliations, we can use a surgical
device (due e.g., to Peixoto 1962 \cite{Peixoto_1962}, Blohin
1972 \cite{Blohin_1972}):%
%which we term the \wood-surgery.
%Using this it is immediate to show:

\begin{lemma}\label{Peixoto-Blohin:lemma} The following surfaces are transitively foliated:

{\rm (1)} $\Sigma_{g,*}$ the once punctured orientable surface
of genus $g\ge 1$;

{\rm (2)} $S_{*,gc}$ the once punctured
%non-orientable surface
sphere with $g$ cross-caps $g\ge 3$.
\end{lemma}

\begin{proof} Start  with a (Kronecker) irrational
foliation of the torus $\torus$. Pick two foliated boxes and
apply the \wood-surgery (Fig.\,\ref{wood:fig}.a). Connecting
by a handle the two contours, and deleting the arc (saddle
connection) shows that $\Sigma_{2,*}$ is transitive (the arc
deletion amounts to a single puncturing as the handle is
thought of as infinitesimal so that the two depicted arcs are
in reality just one). For higher genuses,
%we merely need to
consider an alignment of such flow-boxes (cf.
Fig.\,\ref{wood:fig}.b), and delete the thick arc (3 pieces,
but connected!)  proving (1). Regarding (2) we cross-cap the
contours (cf. Fig.\,\ref{wood:fig}.c) and delete the thick
arc. Since the torus with one cross-cap is $\approx$ $S_{3c}$
this proves (2).
\begin{figure}[h]
\centering
    \epsfig{figure=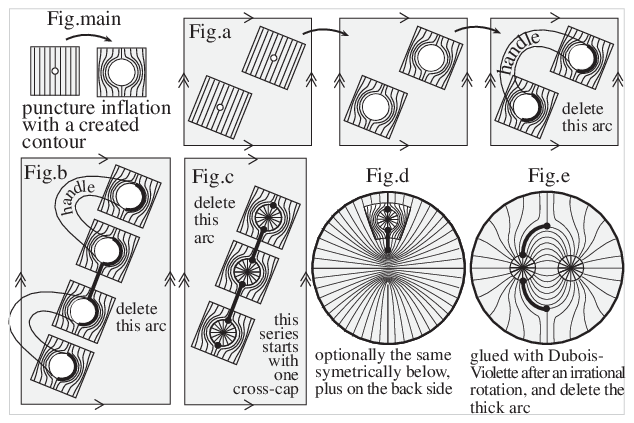,width=122mm}
  \caption{\label{wood:fig}
  The \wood-surgery of Peixoto-Blohin}
\vskip-5pt\penalty0
\end{figure}
\end{proof}

%Thus all surfaces on the ground floor of the monolith are
%intransitive, i.e.:

Doing the same surgery in the Dubois-Violette foliation
(Fig.\,\ref{wood:fig}.d), shows:

\begin{lemma} The sphere $S_{n*,k c }$ with $n$ punctures and
$k$ cross-caps is transitive for $n= 4$ and $0\le k\le 4$.
\end{lemma}

Together with (\ref{Peixoto-Blohin:lemma}) (and keeping a view
over the monolithic Figure~\ref{monolith:fig}), this gives a
complete knowledge of which $S_{n*,k c }$ are transitive,
except when $(n,k)$ takes the values $(3,2)$, $(2,2)$,
$(3,1)$.

The first case $(3,2)$ is transitive by gluing
Fig.\,\ref{wood:fig}.e with Dubois-Violette's disc
(Figure~\ref{Dubois:fig}, Rosenberg's version). This piece of
the puzzle suffices to establish in view of the combinatorics
of Figure~\ref{monolith:fig} the:

\begin{prop}\label{rank-4-or-more-is-volatile-transitive} An open metric surface of finite-connectivity
(=rank of the $\pi_1$) $\ge 4$ is transitive.
\end{prop}

\subsection{Sporadic obstruction in rank $3$}

The last case $(n,k)=(3,1)$ (of the previous section) is
intransitive by the following (using again Riemann's branched
coverings conjointly with the Poincar\'e-\Kerekjarto-Hopf
index formula):

\begin{lemma}\label{intransitivity:new-obstruction}
The thrice-punctured projective plane $S_{3*,c}$ is
foliated-intransitive.
\end{lemma}

\begin{proof} If the foliation is orientable, then we are
reduced to the flow-intransitivity of the projective plane
$\proj$ (which boils down to that of $S^2$ prompted by
Poincar\'e-Bendixson). Otherwise we construct the
%%$2$-fold
double cover $\Si\to M:=S_{3*,c}$ rendering the foliation
oriented (\ref{orienting:2-fold-covering}). Compactify this to
a branched covering $\Si^{*}\to M^*=\proj$ via Riemann's trick
(\ref{Riemann:branched-cover}). By Riemann-Hurwitz
$\chi(\Si^*)=2\chi(\proj)-\deg(R)$. As there are $3$
punctures, $0\le \deg(R) \le 3$. Hence $-1\le \chi(\Si^*) \le
2$. If $\chi(\Si^*)=2$, then we have $\sphere$ or $\proj$ both
precluded by Poincar\'e-Bendixson. If $\chi(\Si^*)=0$, then we
have $\torus$ or $\Klein$. The former is excluded since the
sheet exchange involution must be orientation reversing hence
cannot fix isolated points (Klein-Weichold argument already
used in (\ref{Klein-Weichold})). The Klein option $\Klein$ is
precluded by Markley's flow-intransitivity of the Klein bottle
(\ref{Markley}).
%(Notice here that some details must be
%expanded.)
Finally when $\chi(\Si^*)=-1$, we have $\deg(R)=3$.
This means by construction that all three punctures are
%singularities
non-orientably foliated. Thus they have each
%determine
a
semi-integral index. By the index formula they sum up to
$\chi(\proj)=1$, violating arithmetics modulo one
$\frac{1}{2}+\frac{1}{2}+\frac{1}{2}=\frac{3}{2}\neq 1=0$.
\end{proof}

At this stage the only remaining bastion of resistance is the
case $S_{2*,2c}$ (twice punctured Klein bottle), of which it
is not completely trivial to decide its transitivity issue.
(We do not know yet the
answer.) %%%%%But see Alex's notes page 89 for an almost
%%%%complete obstruction...
This is surely well known yet confess that presently we failed
to present a decent proof.

\subsection{Intransitivity transfer from the metric soul}

As seen in (\ref{soul:uniqueness}) it is legitimate to
%imagine
%that
think of any (non-metric) surface of finite-connectivity (i.e.
with fundamental group of finite rank) has having  a metric
{\it soul} capturing the salient invariants $(\chi=1-b_1,
\varepsilon, a)$ of connectivity, the ends number and
indicatrix (=orientable or not) which in the metric case
constitutes a complete system of topological invariants
(\ref{Kerkjarto:end}).
%(of finite-connectivity).
A foliated application of this soul-method
%(which we recall is
(very akin to  Nyikos' bagpipes) is the transfer of
intransitivity from the metrical soul to the whole manifold:

\begin{prop}\label{soul:intrans}
If the soul of a  finitely-connected (non-metric) surface is
intransitive, then so is the whole surface.
\end{prop}

\begin{proof} The argument is similar to
(\ref{intransitivity:non-orient_rank_2}). Assume $M$
transitive with  dense leaf $L$. We know that $L$ is either
${\Bbb R}$ or the long-ray ${\Bbb L}_+$ (cf.
(\ref{separability}) below). Deleting from $M$ the long-side
of $L$ does not
%affect
change the
%characteristic
invariants $(\chi,\varepsilon, a)$
%(viz. characteristic
%equivalently connectivity, number of ends  and indicatrix
by virtue of (\ref{deleting:a_closed_long_ray}), hence
%%%leaves
keeps invariant the soul type by (\ref{soul:uniqueness}).
After this long slit, we have a new leaf $L\approx {\Bbb R}$
which is still dense. As $L$ is Lindel\"of, it is contained in
a Lindel\"of subregion $U\supset L$ (random amalgam of charts
covering $L$). By the kernel killing procedure
%in the proof of
%({\ref{calibrated-exhaustions:lemma}})
(\ref{killing:kernel}) we can enlarge $U$ into a soul
$S\supset U$ so that $\pi_1(S)\to \pi_1(M)$ is isomorphic. $L$
being dense in $M$, it is a fortiori
%so
in
$S$, violating the soul intransitivity.
\end{proof}

Applying (\ref{intransitivity:new-obstruction}) we find:

\begin{cor}
(Non-metric) surfaces with $\pi_1=F_3$ and $3$ ends are
intransitive.
\end{cor}

This applies for instance to the Moorization (cf. e.g.
\cite{GabGa_2011}) of the bordered surface given by the sphere
with 3 holes and 1 cross-cap (i.e. 3-holed $\proj$). Notice
that the full-theory of the soul is not truly required for
such simple-minded Moore type surfaces whose geometric
structure is sufficiently explicit so as to replace the soul
by a calibrated exhaustion by subregions having the same
topological type (just add successively the thorns).

\subsection{Biminimal foliations (bidirectional denseness)}

In this section we address a somewhat specialized question,
%yet
involving
%some respectable
%classical
methods of Bendixson,
\Kerekjarto{ }and Mather.
%Our point of departure is that
Albeit Dubois-Violette's foliation of $S^2_{4*}$ (4 punctures
in $S^2$) is minimal, the 4 leaves converging to the punctures
(separatrices) fails to be dense in one direction.
%Hence it is not as perfect as
In contrast, all leaves of the Kronecker irrational foliation
of the torus are bidirectionally dense.
%in both
%directions.

\begin{defn}
{\rm A
%one-dimensional
$1$-foliation is {\it biminimal} if all leaves are
bidirectionally dense, i.e. both semi-leaves emanating from
any point are dense.
% in both directions.
}
\end{defn}

By the sequential-compactness argument
%%%of
%%%(\ref{separability}),
(\ref{long-semi-leaf}) all leaves of a biminimal foliation are
real-lines ${\Bbb R}$ (provided ambient dimension $\ge 2$).
%A naive question is to
We may  wonder if $S^2_{4*}$ (fourth-punctured sphere) is
biminimal. The negative answer is
%easily
supplied by the following
%, which
%we call the
Bendixson alternative:
%(compare also Mather 1982
%\cite{Mather_1982}):

\begin{lemma} In a  foliated punctured plane ${\Bbb
R}^2_*$, there is either a circle leaf enclosing the puncture
or there is a leaf converging to the puncture. Moreover both
alternatives
%become mutually exclusive
exclude mutually if the first option occurs countably many
times with a nested collection of circle leaves shrinking to
the puncture.
\end{lemma}

\begin{proof} Mather 1982 \cite[p.\,246, \S 7]{Mather_1982},
refers to Bendixson original paper or to \Kerekjarto{ }1923
\cite[p.\,256]{Kerekjarto_1923}. Here is a brief outline of
the argument as modernized by Birkhoff, Nemytskii-Stepanov,
Reeb, etc.
%Yet nowadays it is perhaps better to argue via
The classical Poincar\'e-Bendixson argument shows that
dichotomy implies properness in the flow case, to which we may
reduce as $\pi_1={\Bbb Z}$ by passing to the double cover
(\ref{orienting:2-fold-covering}). Then
%there is
a standard Zorn
lemma argument
%to get
creates a compact leaf, provided there is a leaf with compact
semi-leaf closure. The compact circle leaf encloses the
puncture,  since otherwise it bounds a disc (recall that a
proper power of the generator of $\pi_1$ cannot be realised by
a Jordan curve).
\end{proof}

\begin{cor}\label{biminimal-impeded-by-puncture} Any punctured surface (metric or not)
lacks a biminimal foliation.
\end{cor}

Since any finitely-connected open metric surface possesses an
end
%whose
neighbourhood
%is
homeomorphic to a punctured plane
%(cf. \Kerekjarto, Richards and Mather, p.\, 237)
(\ref{Kerkjarto:end}), we have the special case:

\begin{cor} A finitely-connected metric surface
%lacks a
cannot be biminimally foliated, except the torus.
\end{cor}

\begin{proof} If compact, the surface has $\chi=0$
(\ref{Euler:classical-obstruction}), so is the torus or the
Klein bottle $\Klein$ (\ref{Moebius-Klein-classification}).
The latter option is precluded by Kneser 1924 (\ref{Kneser}).
%), as argued in the proof of
%(\ref{separability:cor}).

In the open case, by \Kerekjarto's end theorem
(\ref{Kerkjarto:end}), the surface is a punctured one
%(consider the one-point compactification of the end).
and (\ref{biminimal-impeded-by-puncture}) concludes.
\end{proof}

Those results fails to
%answer
tell if the infinite connected sum of tori have a (bi)minimal
foliation.
%(cf. maybe Beni\`ere's
%thesis, also Nikolaev's book 2002).
Also which metric surfaces are biminimally foliated? Besides,
does the last corollary extend to non-metric surfaces? It
seems that the Lindel\"of
%reduction
exhaustion trick does not work well. Note that the doubled
Pr\"ufer surface $2P$ is not biminimal (indeed not even
minimally foliated by (\ref{Baillif:nano-black-holes})).
Baillif's example in BGG2 \cite{BGG2} of a minimally foliated
non-metric surface foliated by short leaves is
%presumably
not biminimal, raising the:

\begin{ques} Can we
%manufacture
find a non-metric biminimally foliated surface?
\end{ques}

%%%Geometric-gravitational
\section{Gravitational
%obstructions
effects (quantum radiation at the microscopic scale)}
\label{Gravitation:sec}

\subsection{Long semi-leaves are tame}

Given a point in a manifold foliated by curves, then after
fixing one of the two possible directions there is a unique
motion starting from the point prescribed by the foliation.
Such a ``trajectory'' referred to as a {\it semi-leaf}, is a
bordered $1$-manifold (under the leaf-topology).
%emanating from the point.

\begin{prop}\label{long-semi-leaf}
Long semi-leaves in
%$1$-foliation of a manifold
$1$-foliated manifolds are properly embedded.
%in $M$.
\end{prop}

\begin{proof} By
%the
classification of bordered 1-manifolds (in $3$ species:
$[0,1]$,
%${\Bbb R}_{\ge 0}=
$[0,\infty)$ and ${\Bbb L}_{\ge 0}=[0, \omega_1)$ {\it closed
long ray}) our non-metric semi-leaf is the closed long ray.
%%%${\Bbb L}_{\ge 0}$.
The latter being
sequentially-compact, we conclude with (\ref{sekt}) below.
\end{proof}

\begin{lemma}\label{sekt} (i) Let
$f\colon X \to Y$ be a continuous map, where $X$ is
sequentially-compact (sekt for short), and $Y$ Hausdorff and
first-countable. Then the mapping $f$ is closed.

(ii) In particular $f(X)$ is closed and if furthermore $f$ is
injective, then $f\colon X \to f(X)$ is a homeomorphism onto
its image.
\end{lemma}

%{\footnotesize
\begin{proof} First recall that a closed subset
of a space is sequentially-closed, and conversely if the space
is first-countable.

Thus for (i), it is enough to show that $f(F)$ is
sequentially-closed, whenever $F$ is closed. So let $y_n\in
f(F)$ be a sequence converging to $y\in Y$. Choose $x_n\in F$
such that $f(x_n)=y_n$. Since $X$ is sekt, we may extract a
subsequence $x_{n_k} \to x \in F$, converging to $x$, say. By
continuity $f(x_{n_k})\to f(x)$. Since $Y$ is Hausdorff, it
follows $y=f(x)\in f(F)$. q.e.d.

(ii)
%When $f$ is injective
The continuity of the inverse follows from $f$ being a closed
mapping.
\end{proof}

%}

\subsection{Transitivity implies separability}

When a manifold (indeed a space) has a transitive flow (one
dense orbit) then the phase-space is separable (rational times
of a dense orbit).
%What about transitive foliations?
If a manifold is transitively foliated (one dense leaf), then
as the latter can be long it is not obvious that the ambient
manifold has to be separable. It is even trivially false as
exemplified by a long 1-manifold trivially foliated by itself.
%
%Excluding
Ruling out this trivial exception we have:

\begin{prop}\label{separability}
A dense leaf of a $1$-foliation on a manifold $M^n$ of
dimensionality $n\ge 2$ is homeomorphic (w.r.t. the leaf
topology) to the real-line ${\Bbb R}$ or the long ray ${\Bbb
L}_+$. Furthermore $M^n$ is separable.
\end{prop}

\begin{proof} By
%the
classification of 1-manifolds the leaf belongs to one of the
following type: circle ${\Bbb S}^1$, real-line ${\Bbb R}$,
long ray ${\Bbb L}_+$ and long line ${\Bbb L}$. The two
extreme items in this list are sequentially-compact,
%(sekt for
%short),
thus always embedded and closed as point-sets by (\ref{sekt}).
Thus by (the elementary case of) invariance of
%the
dimension (Brouwer {\it et al.} in general, yet not required
presently), the dense leaf $L$ cannot be of those two types as
$n\ge 2$, whence the first assertion. Regarding the
separability clause, it is plain when $L\approx {\Bbb R}$.
Assuming $L\approx {\Bbb L}_{+}$, we may at any point $p$ of
$L$ split the leaf in a short $L_{\le p}$ and a long $L_{\ge
p}$ semi-leaf, resp. homeomorphic to ${\Bbb R}_{\ge 0}$ and
${\Bbb L}_{\ge 0}$ (closed long ray). The latter being
sequentially-compact, it fails to be dense. Thus only the
short
%%%%side
semi-leaf
%has to be
can contribute to denseness,
% in $M$,
and separability of $M$ follows.
\end{proof}

\begin{rem} {\rm For foliations of higher dimensionality,
%than one,
the above (\ref{separability}) fails drastically. An example
of Martin Kneser shows how to foliate with a {\it unique}
%(hence dense)
surface-leaf a non-metric 3-manifold of the Pr\"ufer type (cf.
for a picture, \cite{BGG1}, arXiv version). Kneser's example
can easily be
%slightly modified
`stretched' so as to
%make
render it non-separable.}
\end{rem}

\begin{cor}\label{separability:cor} An
$\omega$-bounded surface is intransitive,
except if it is the torus.
\end{cor}

\begin{proof} If transitively foliated, the surface is
separable (\ref{separability}). Being also $\omega$-bounded it
is compact. Since it is foliated, $\chi=0$
(\ref{Euler:classical-obstruction}). So by classification
(\ref{Moebius-Klein-classification}), it is either the torus
or the Klein bottle, which is intransitive
(\ref{Klein-foliated-intransitive}).
\end{proof}

\subsection{Miniature black holes (Pr\"ufer, R.\,L. Moore, Baillif)}

This and the subsequent section exemplify a
%typical
geometric obstruction to foliated-transitivity lying beyond
the algebraic obstruction encoded in the fundamental group or
better in the soul (\ref{soul:intrans}), as well as beyond the
point-set obstruction of non-separability
(\ref{separability}). Thus the present obstruction (due to M.
Baillif) is the first (within the
%modest
%realm
scope of this paper) being truly non-metrical, yet still of a
geometric nature prompted by the
%specific
granularity of particular non-metric manifolds, namely those
of the Pr\"ufer type (including the Moore and
Calabi-Rosenlicht surfaces, plus of course many other
specimens having a similar morphology). Using the
phase-transition metaphor, this amounts to a volatile-gaseous
(=transitive) configuration which embedded into the
non-metrical fridge becomes frozen-intransitive.

%%%%\subsection{Picturing}
%%%%add optionnally the nice picture on page 97 of AG's notes

We already know that Cantor's long ray is responsible of some
%macroscopic
black hole phenomenology at the
%global level
macroscopic scale  \cite{BGG1}. For instance the long cylinder
$S^1\times {\Bbb L}_{+}$ imposes to each
%of its
foliated structure to be either ultimately vertical (foliated
by straight long rays) or asymptotically horizontal (with
slices \hbox{$S^1\times \{\alpha\}$} occurring as leaf for a
{\it closed unbounded} ({\it club}) subset of $\alpha$'s
running in the long-ray factor).
%Thus
So one can essentially imagine
%that
a super-massive black hole
%is
hiddenly sitting  at the
%``ideal''
long end of the
%(long)
cylinder and dictating the destiny of any foliated structure,
thought of  as an {\it ether} ($\approx${\it substrat
physico-chimique} in R. Thom's jargon) evidencing the
gravitational features of the manifold.
%Of course
We are dealing here
%involved
with a purely naked-topological
%%%(topological-qualitative
form of
gravitation without metric (and the allied
%Riemann's
Riemann curvature tensor), yet
%that we still want
%which is
still reasonably
%to qualify
qualifiable as ``geometric''.

%Beside
Apart from Cantor's long ray (of dimension 1) the other
%key
%example
charismatic prototype of non-metric manifold (requiring two
dimensions)
%, at least under the Hausdorff proviso)
is the
Pr\"ufer construction. Especially
%important
we have, $P$, the bordered Pr\"ufer surface (constructed by a
aggregating rays to an open half-plane, very akin to
projective geometry, esp. the blow-up operation) which---by
folding the contours---produces the Moore surface. This can be
thought of as an upper half-plane with many (infinitesimal)
`teats' hanging down the boundary (horizontal line), compare
Figure~\ref{Train:fig} for a poor depiction.
%
%(The latter
%can be also thought of as given by rays reflecting along the
%boundary of the half-plane.)
Clearly something
%strange
%special (erotical)
`erotical' must happen near the `boundary'
%(precisely
(or rather what remains thereof---that is nothing!),
%and those
%Pr\"ufer and Moore surfaces realise so-to-speak a quantum
%fluctuation of the usual boundary of the closed upper-half
%plane,
and
%we shall see that
by analogy with Cantor's super-massive black hole scenario, we
now
%rather
%%have
imagine a `continuous' series of
%super-light
nano-black holes materialised by the folded contours of the
Pr\"ufer surface $P$, called the {\it thorns} of the
Moorization. (If you prefer imagine the little black holes
located at each teat's extremity.) The latter effects a
%sort of
quantum radiation at the microscopic scale as shown by the
following technique of M. Baillif ($\approx 2008-9$):

\begin{theorem}\label{Baillif:nano-black-holes}
In any foliation of the Moore surface, almost all (=all but
countably many exceptions) thorns are semi-leaves. Similarly,
in any foliation of the Calabi-Rosenlicht surface (=doubled
Pr\"ufer surface $2P$) almost all bridges are leaves. (Here
Bridges refer to  the contours of $\partial P$ viewed in
$2P$.)
\end{theorem}

\begin{proof} Let $P\to M$ be the contour-folding projection
from Pr\"ufer-to-Moore. The boundary $\partial P$ decomposes
as a continuum of real-lines contours, whose respective image
are the thorns denoted $T_x$ (homeomorphic to a semi-line
${\Bbb R}_{\ge 0}$) indexed by $x\in{\Bbb R}$. The complement
of all thorns is called the {\it core} (of the Moore surface).
The core $U$ being a cell $\approx {\Bbb R}^2$, hence
separable, let us fix $D\subset U$ a countable dense set. Let
${\cal F}_D$ be the
%saturation
set of leaves of the foliation (denoted $\cal F$) passing
through the points of $D$.

If the thorn $T_x$ is not a semi-leaf of $\cal F$, then there
is a point $y\in T_x$ such that $L_y$ (=the leaf through $y$)
%%%%moves
deviates into the core. Then we can find nearby a leaf $L$ in
the collection ${\cal F}_D$ intercepting $T_x$
(Fig.\,\ref{Baillif:fig}.a).

\begin{figure}[h]
\centering
    \epsfig{figure=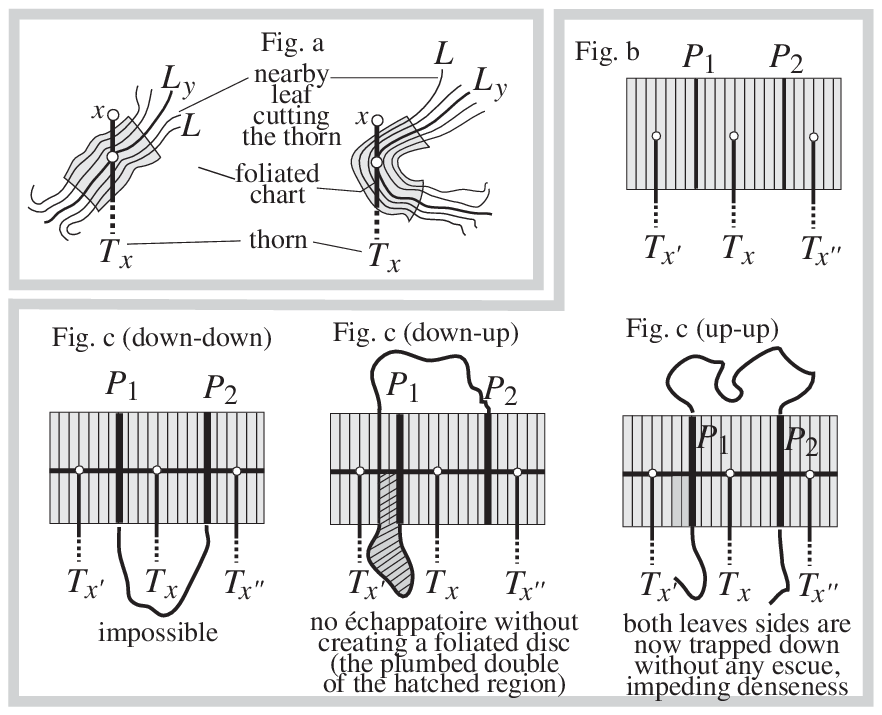,width=122mm}
  \caption{\label{Baillif:fig}
  %Artist view of
  Interception of thorns by leaves and black thorns sandwiches}
\vskip-5pt\penalty0
\end{figure}

Choosing for each such $x$
%such
an intercepting leaf defines a
map
$$
\varphi \colon \text{ White thorns}:=\{x\in {\Bbb R} : T_x
\text{ is not a semi-leaf}\} \longrightarrow {\cal F}_D .
$$
Now a leaf can intersect at most countably many thorns. This
follows from the squatness of the Moore surface $M$ (i.e. any
continuous map from the long ray to $M$ is eventually
constant) and the Pr\"ufer topology which shows that any
subset of the Moore surface having at most one point on each
thorn is discrete. Therefore long leaves are precluded and
short ones cannot contain an uncountable discrete subset.

Hence the map $\varphi$ is at most $\omega$-to-1 (denumerable
fibres), whence the countability of the set of ``White
thorns''.
The
%proof
argument for the doubled Pr\"ufer is
%very
nearly identical,  hence left as an
%mental
exercise.
\end{proof}

\subsection{Intransitivity of the thrice-punctured Moore surface}

The Moore surface $M$ itself, being simply-connected
(Seifert-van Kampen applied to the canonic open cover by
%%%%$\text{core}\cup \text{one thorn}$
individual thorns added to the core), is
%certainly
intransitive, as deduced either via Schoenflies
(\ref{Alex_separation}) or
%%by
Poincar\'e-Bendixson (\ref{Poinc-Bendixson_many}).

After one or two punctures there is an universal obstruction
dictated by the fundamental group
(\ref{infinite-cyclic-group}) resp.
(\ref{dichotomic-free-of-rank-2:prop}). In the Moore case, the
argument can be done by hand:

\medskip
{\footnotesize {\bf Optional easy  argument by hand.} One adds
successively
%or rather in block all
the thorns of $N=M_{n*}$ the Moore surface  with $0\le n\le 2$
punctures: i.e., $N=(\text{core } U) \cup \bigcup_{x\in {\Bbb
R}} T_x$ which is covered by the $M_x:=U\cup T_x$. Assuming
$M$ transitive with dense leaf $L$ we know that $L$ is short
(squatness of Moore). Thus $L$ is Lindel\"of and therefore
contained in $M_D:=\bigcup_{x\in D} M_x$ a countable union of
$M_x$ ($D\subset {\Bbb R}$ denumerable). Using either Morton
Brown's monotone union theorem (or the classification of
metric $1$-connected surfaces (\ref{uniformization}) or even
better a home-made homeomorphism) it is easy to see that $M_D$
is homeomorphic to the core $U$. Since $U$ contains $L$
densely, this violates the intransitivity of the core.

}
\medskip

Doing more than $\ge 3$ punctures in Moore there is no
universal algebraic obstruction (recall Dubois-Violette), but
a gravitational one (remind Baillif):

\begin{prop}\label{Baillif:adapted_by_Gabard}
The Moore surface with $n\ge 3$ (hence all $n$) punctures
$M_{n*}$ is foliated-intransitive.
\end{prop}

\begin{proof} (with some little bluff, but hopefully
%true!
convincing enough!) Truncate the punctured surface $N:=M_{n*}$
(foliated by ${\cal F}$) by looking at the subregion $U$ lying
below a line above which all punctures are lying, and which
therefore is homeomorphic to the Moore surface $M$. By
(\ref{Baillif:nano-black-holes}) applied to $(U,{\cal F}_U)$
almost all thorns are ``black'', i.e. semi-leaves of the
foliation. Fix any black thorn $T_x$, and about its extremity
$\partial T_x$ a foliated chart $B$. Assuming that $L$ is a
dense leaf of ${\cal F}$ it will certainly appears on both
sides of $T_x$ regarded in the foliated-box $B$ as 2 plaques
$P_1, P_2$ separated by the plaque $P$ of $B$ containing
$T_x$. Since the set of white thorns is countable, its
complementary set of black thorns is dense (Baire). Thus we
can squeeze in sandwich the 2 plaques $P_1, P_2$ by 2 new
black thorns $T_{x'}$ and $T_{x''}$
(Fig.\,\ref{Baillif:fig}.b) such that within the box $B$,
$P_1$ separates $T_{x'}$ from $T_{x}$ and $P_2$ separates
$T_{x}$ from $T_{x''}$. As the two $P_i$ belong to the same
leaf they must be somehow connected.

%To avoid a picture, we
Agreeing that the box is vertically foliated, both plaques
$P_i$ will be either connected through their bottoms
(down-down), in a mixed fashion (down-up) or via their tops
(up-up).

Since $T_x$ goes down to infinity, a down-down connection is
precluded, except of course if the leaf re-traverse the box
$B$ but then a foliated disc is created, by doubling the first
return to the cross-section of $B$
(Fig.\,\ref{Baillif:fig}.c). We use here Schoenflies
(\ref{Schoenflies-Baer}) (or a simple form thereof) in the
1-connected Moore surface $U\approx M$ (taking advantage of
the canonical open cover
%of Moore
by thorns aggregation). The
same argument excludes the possibility of a down-up
connection. Thus we have an ``up-up'' connection and then
sides of the leaf $L$ are squeezed between the three thorns
$T_{x'},T_{x}, T_{x''}$ and this without the possibility of
coming up again (else a foliated disc is created by the same
Schoenflies mechanism). This triple squeezing clearly impedes
the denseness of $L$, proving our assertion.
\end{proof}

%%%%%Euler-Venn diagram of
\subsection{Experimental data:
%surfaces
%with prescribed
prescribing topology and foliated dynamics}

Now we try to draw a
%big
picture showing
%systematically
the
%%mutual interactions
interplay between the topology and the possible foliated
dynamics for surfaces. This involves a Venn diagram with the
following topological versus
%dynamical
foliated attributes:

(1)
%Algebro-geometric (combinatorial)
Combinatorial topology: simply-connected $\Rightarrow$
dichotomic;

(2) Point-set topology: metric $\Rightarrow$ separable;

(3) %Qualitative
Foliated dynamics:
%of foliation:
minimal $\Rightarrow$ transitive $\Rightarrow$ foliated.

Recall also some
%non-trivial
%interdisciplinary
`transverse' implications: transitive implies separable
(\ref{separability}), and the mutual exclusion of 1-connected
and transitive (\ref{Alex_separation}) or
(\ref{Poinc-Bendixson_many}). By way of examples the following
diagram (Figure~\ref{Venn}) shows that this is a
%%%n essentially
reasonably exhaustive list of obstructions. A closer look
%at the diagram
aids guessing new empirical obstructions, or more neutrally to
ask the right questions. For instance
%it is interesting to note that
it looks rather hard to exhibit a separable, 1-connected,
non-metric surface lacking a foliation.
%%%%%CENSURED AS PERHAPS FALSE
\iffalse
Thus the following conjecture is a rather
metaphysical existence question (i.e., potentially at the
borderline of the usual axiomatic ZFC
(Zermelo-Fraenkel-Choice):
%
\begin{conj} Any pseudo-Moore surface (i.e.,
separable, 1-connected, but non-metric) admits a foliation.
\end{conj}
%
\begin{rem} {\rm Recall however that it is possible
(yet not very easy) to locate separable non-metric surfaces
lacking foliations (cf. the mixed Pr\"ufer-Moore surfaces in
BGG2 \cite{BGG2})}
\end{rem}
\fi
%
Recall however that it is possible (yet not very easy) to
locate separable non-metric surfaces lacking foliations (cf.
the mixed Pr\"ufer-Moore surfaces in BGG2 \cite{BGG2}).

Let us
%briefly
describe the various regions of the diagram (the following
enumeration refers to the labels (1), (2), (3), etc. fixed on
Figure~\ref{Venn} in a spiral-like fashion):

(1) and (2) contains respectively only the plane and the
sphere which are the only simply-connected metric surfaces
(\ref{uniformization}).

(3) is a region where it is difficult to exhibit a specimen.
(For a vague candidate cf. Section~\ref{long-sun})

\begin{figure}[h]
%\centering
   \hskip-25pt \epsfig{figure=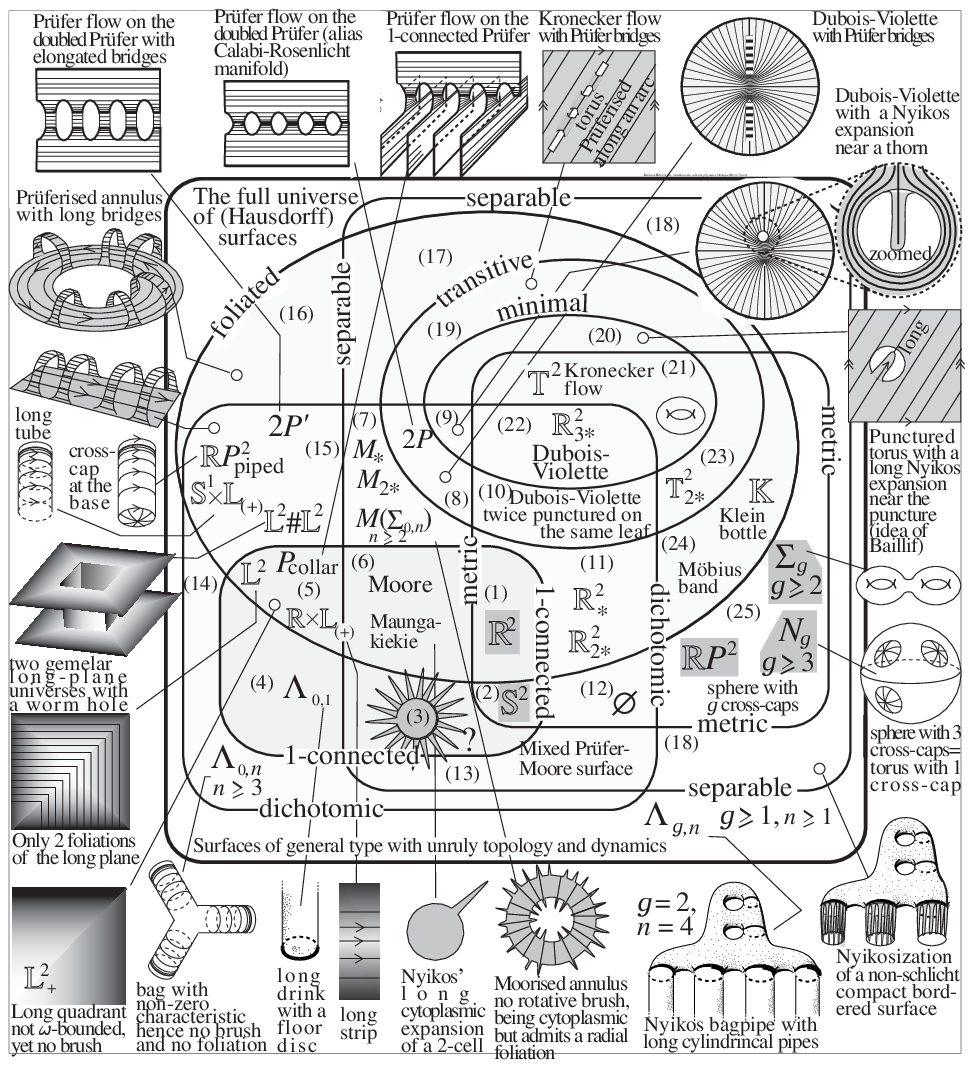,width=142mm}
\vskip-15pt\penalty0
  \caption{\label{Venn}
  %Artist view of
  Some foliated geography in the non-metric realm}
\vskip-5pt\penalty0
\end{figure}

(4) contains the long-glass $\Lambda_{0,1}$ which is the
long-cylinder ${\Bbb S}^1\times{\Bbb L}_{\ge 0}$ capped-off by
a 2-disc. This surface cannot be foliated by BGG1 \cite{BGG1}.

(5) has a plethora of examples including the long-plane ${\Bbb
L}^2$, the long-quadrant ${\Bbb L}^2_+$ and the original
1-connected Pr\"ufer surface $P_{\rm collar}$.

(6) contains the Moore surface $M$ which has a foliation
induced by the vertical foliation of the bordered Pr\"ufer
surface $P$. The latter has semi-saddle singularities
disappearing during the folding process $P\to M$. Also in (6)
we have the more exotic Maungakiekie surface, which is the
result of a long Nyikos expansion effected to an open 2-cell.

(7) contains naively speaking $2P$ the doubled Pr\"ufer (alias
Calabi-Rosenlicht manifold) which has a horizontal foliation.
This does not preclude the possibility that $2P$ endowed with
a more exotic foliation is transitive.  Thus here we ignore
the question of the sharp positioning of the manifold $2P$,
within the diagram. However, in view of (\ref{Dubois}) it
seems that $2P$ is transitive, but not minimal by
(\ref{Baillif:nano-black-holes}), so $2P$ belongs in reality
to (8). Yet, class (7) contains the Moorized annulus $M(A)$
with the radial foliation. Since $M(A)$ has a $\pi_1$
isomorphic to ${\Bbb Z}$, it is intransitive
%(modulo the truth of
by (\ref{infinite-cyclic-group}). For the same reason this
class also contains the Moore surface punctured once $M_{*}$.
The twice-punctured Moore surface $M_{2*}$ is also here, being
dichotomic with $\pi_1=F_2$, hence intransitive by
(\ref{dichotomic-free-of-rank-2:prop}).

(8) is  non-empty by Pr\"uferising Dubois-Violette's example
(\ref{Dubois}).
%This class
%also perhaps contains the thrice punctured Moore surface.
%%%%%%NO the latter is intransitive.

(9) contains  a Nyikos long expansion performed near one of
the 4 punctures of Dubois-Violette (elongate the separatrix).

(10) contains the Dubois-Violette foliation (on ${\Bbb
R}^2_{3*}$) punctured twice on the same leaf to generate
artificially a non-dense leaf. In reality the underlying
manifold ${\Bbb R}^2_{5*}$ is minimally foliated (puncture
twice on different leaves of Dubois-Violette).

(11) contains the punctured plane ${\Bbb R}^2_{*}$ (foliated
e.g. by concentric circles) which is intransitive by
(\ref{infinite-cyclic-group}). Likewise ${\Bbb R}^2_{2*}$ is
intransitive by (\ref{dichotomic-free-of-rank-2:prop}).
%%%%%%%%%%%%%%%%%%%%%%%%%%%%

(12) is an empty region, because any metric open surface has a
Morse function without critical points
(\ref{Morse-Thom:surfaces}), thus a foliation. So a surface in
(9) has to be compact, and dichotomy allows only the sphere
(by  classification (\ref{Moebius-Klein-classification})),
which has already been positioned in a deeper nest of the
diagram.
%%%%%%%%%%%%%%%%%%%%%%%%%%%%

(13) could contain the Moorization $M(G)$ of a
multiply-connected domain $G=\Sigma_{0,n}$ with at least $n\ge
3$ contours (starting with the ``pant''). A vague idea could
be that the Moorization forces a vertical behavior near the
cytoplasmic expansions present in the Moorization, thus
``shaving'' the ``hairs'' gives a compact subregion
(homeomorphic to $G$) along the boundary of which the
foliation is transverse, and then we are done by the
Euler-Poincar\'e obstruction. But this
 hasty intuition is wrong as shown in
 Section~\ref{Moorized-disc:sec}. In
particular this argument would imply that the Moorized disc
cannot be foliated,
%yet this might conflict with the dual
%intuition that the latter is homeomorphic to the Moorized
%half-plane... In fact
but  (\ref{Moorized-disc:prop}) below shows the contrary. The
same construction (cf. Fig.\,\ref{Moorized-disc}g) also shows
that the $M(\Sigma_{0,n})$ for $n\ge 3$ admit foliations, and
so belongs to region (7), but not to (8) in view of
(\ref{Baillif:adapted_by_Gabard}), the case $n=3$ following
also from (\ref{dichotomic-free-of-rank-2:prop}). So region
(13) looks rather deserted, except if we remind the
construction in BGG2 \cite[Sect.\,4.2]{BGG2} of mixed
Pr\"ufer-Moore surfaces which produces a bunch of separable,
non-simply-connected surfaces lacking foliations. Furthermore
if we accept the operation of {\it full} Nyikosization $N$,
which produces long hairs at all point of the boundary, then
$N(\Sigma_{0,n})$ for $n \ge 2$ would belong to (13), compare
Section~\ref{long-sun}.
%%%%%%%%%%%%%%%%%%%%%%%%%%%%%%%%%%%%%%%%

(14) has the surfaces $\Lambda_{0,n}$ of genus $0$ with $n\ge
3$ long cylinder-pipes, which lack foliations by BGG1
\cite{BGG1} (super-massive black hole scenario).

(15) admits  a plethora of examples with most of them arising
indeed from a non-singular flow (cf. \cite{GabGa_2011} and
recall optionally the theorem of Whitney \cite{Whitney33}
(building over Hausdorff) telling that non-singular 2D-flows
induce foliations). Of our pictured examples the only one
which is not induced by a flow is the ``wormhole'' double
long-plane, i.e. the connected sum ${\Bbb L}^2 \# {\Bbb L}^2$
which is however
%easily
foliated by circles.

(16) contains simple examples using variants of the Pr\"ufer
construction. For instance we can Pr\"uferize an annulus and
glue radially opposite boundaries by long bridges (cf.
figure).

(17) The same construction as in (16) with short bridges
yields a specimen.

(18) We lack a serious example.
%(do not thrust here the picture which
%are old stuff from the flow case!).
However with the operation of full Nyikosization we can take
$N(\Sigma_{g,n})$ where the genus $g\ge 1$ and with $n\ge 1$
contours.

(19)  Take a Kronecker torus Pr\"uferized along an arc.
%Note
%that
By the scenario of nano-black holes
(\ref{Baillif:nano-black-holes}) this surface is not minimal,
giving a sharp positioning.

(20) Puncture the Kronecker torus, and making a long Nyikos
expansion
%near the puncture in the direction of the foliation
of one of the 2 separatrices converging to the puncture gives
a minimal foliation on a non-metric surface (trick due to M.
Baillif, more details in BGG2 \cite{BGG2}).

(21) Take a Kronecker torus. This is the unique compact
specimen as follows from Kneser (\ref{Kneser}), but there is
of course a menagerie of non-compact examples (e.g., Kronecker
torus punctured once).

(22) Take the example of Dubois-Violette on $S^2_{4*}={\Bbb
R}^2_{3*}$.

(23) contains as fake specimen the Kronecker torus punctured
twice on the same leaf. (In reality the twice punctured torus
also carries a minimal foliation, so it is not a sharp
example.)

(24) Take the torus with the trivial foliation by circles. Of
course this is fake, since the torus really lives in (21). Yet
(24) contains the M\"obius band (=twisted ${\Bbb R}$-bundle
over $S^1$) which is intransitive by
(\ref{infinite-cyclic-group}). Doing one (or even 2) punctures
in M\"obius
%the $\pi_1$ grows to $F_2$, yet it is not dichotomic,
%so we cannot draw intransitivity.
the surface is still intransitive by
%%%(\ref{intransitivity:non-orient_rank_2})
(\ref{Klein-Weichold}) resp.
(\ref{intransitivity:new-obstruction}). For a compact example
we have
 the Klein
bottle ${\Bbb K}$, which by Kneser (\ref{Kneser}) is not
minimal, and in fact  foliated-intransitive
(\ref{Klein-foliated-intransitive}).

\iffalse We would like to show that Klein is intransitive.
This is well-known for flows (Markley, 1968). For foliations
one can cut the Klein bottle along the Kneser circle, which is
not null-homotopic (else it would bound  a disc) and we get a
bordered surface of characteristic zero which is either an
annulus or a M\"obius band. Both have $\pi_1={\Bbb Z}$ so
there is an obstruction to a dense leaf by
(\ref{infinite-cyclic-group}). \fi

(25) contains
%the bunch of
all closed surfaces except those having already been
positioned, namely the sphere, and the 2 surfaces with Euler
characteristic zero. These split into orientable surfaces of
genus $\ge 2$ and non-orientable surfaces (spheres with $g\ge
1$ cross-caps) for all values of $g$ except $g=2$ which is the
Klein bottle. Since open metric surfaces always foliate (Morse
function argument (\ref{Morse-Thom:surfaces})), this is a
complete
%classification in
tabulation of the birds in class (25) in view of the
classification (\ref{Moebius-Klein-classification}).

\subsection{Razor principle foiled}\label{Moorized-disc:sec}

Given the Moorized disc $M(D^2)$, which looks like a hairy
disc with hairs emanating transverse to the boundary
(Fig.\,\ref{Moorized-disc}a), one could expect that a foliated
structure has to be compatible with the hairs, and thus
`shaving' the hairs gives an impossible foliation of the
(compact) disc. This would imply that $M(D^2)$ lacks a
foliation. This naive principle is erroneous.

Basically, it is
%erroneous
faulty because the {\it semi-saddle} $xy$ (level curves of
that function) restricted to $y\ge 0$ is not the unique one
%%%descending to
inducing a regular foliation after Moorization. Indeed the
{\it half-saddle}  defined  by $x^2-y^2$ (cf.
Fig.\,\ref{Moorized-disc}c) behaves also well under folding.
This suggests how to foliate $M(D)$, compare
Figure~\ref{Moorized-disc}, which is commented upon in the
picture-assisted proof below.

\begin{figure}[h]
\centering
    \epsfig{figure=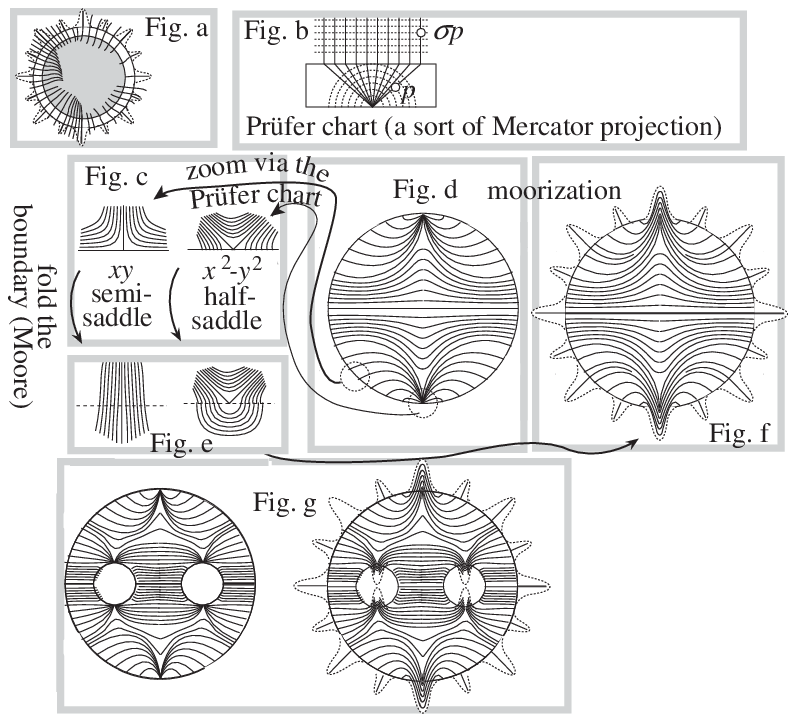,width=122mm}
%\vskip-30pt\penalty0
  \caption{\label{Moorized-disc}
  Foliating the Moorized disc}
\vskip-5pt\penalty0
\end{figure}

\begin{prop}\label{Moorized-disc:prop}
The Moorized disc can be foliated.
\end{prop}

\begin{proof} Consider the singular foliation of the disc
$D^2$ depicted on Fig.\,\ref{Moorized-disc}d: this is
everywhere orthogonal to the boundary except at 2 singular
points resembling a `spiderman'. Recall the
%definition
interpretation of Pr\"ufer's construction in terms of rays,
and the synthetic description of the (Pr\"ufer)-charts. This
can be thought of as the map
%$\varphi$
$\sigma$ pictured above (Fig.\,\ref{Moorized-disc}b). It is
easy to
%%%check and visualise
see that if the foliation is vertical then its image by the
Pr\"ufer chart is a semi-saddle, whereas the spiderman
singularity transforms to a half-saddle. Since both types of
singularities disappear during the folding process
(Fig.\,\ref{Moorized-disc}e), we get a genuine foliation on
the Moorization $M(D^2)$ (depicted on
Fig.\,\ref{Moorized-disc}f).
\end{proof}

The same method shows that the Moorization of the
multiply-connected domains (any number of contours) can be
foliated (Fig.\,\ref{Moorized-disc}g).
%%%%%%For slightly more details AG's notes page N.37 and N.39.

\subsection{Razor principle for long hairs, supernovas
and science fiction}\label{long-sun}

%As we noticed
It is rather hard to find  non-metric separable surfaces
lacking foliations. However, in \cite[Sect.\,4.2]{BGG2},
Baillif showed that suitable mixed Pr\"ufer-Moore surfaces
(which are separable) lack foliations. The idea is the
following. Start with $P$ the bordered Pr\"ufer surface. Each
contour (=boundary component) of $P$ can either be folded
(Moorization) or left unaltered. Since Moorization is well
behaved w.r.t. the vertical foliation
(Fig.\,\ref{Moorized-disc}c-e), but not against the horizontal
foliation developing thorn singularities when folded, Baillif
showed that a violent mixture of both processes (folding and
`nothing') produces  bordered surfaces whose double are
separable, yet without foliations. Being doubled such surfaces
fails to be 1-connected, and it would be interesting to answer
the:

%%%%%%CENSURED BECAUSE LOOSE
\iffalse if only countably many collars are required. However
it can be observed:

(1) the precise argument of BGG2 is somewhat involved, and

(2) it uses the double $2 P_{A,B}$ (notation of loc. cit.),
thereby not answering the question of the existence of a
simply-connected example lacking a foliation.

Thus it is natural (at least legitimate) to ask: \fi

\begin{ques}\label{corona-sun:question} Is there a simply-connected non-metric separable  surface without
foliations? If yes, then more subjectively, what is the
simplest
%non-metric separable surface
 %lacking foliations?
such example?
\end{ques}

In the second subjective question, ``simplest'' seems to have
a double fragrance, namely
%%the
simplest to construct or
%%%the
simplest to show its non-foliability. Along the second
interpretation, we have perhaps the following (very
hypothetical) answer, involving another black hole
scenario---this time at the `mesoscopic' scale! {\it Warning:}
the next paragraph is maybe only pure science fiction.

Recall Nyikos' long (cytoplasmic) expansion of a cell (cf.
\cite{BGG2}). \iffalse As explained in BGG2 \cite{BGG2}, we
can think of the plane trivially foliated by lines, and
elongated one leaf to get a surface $N$ which is separable,
$1$-connected and foliated. If we can show a sort of
black-hole behaviour, namely that foliations of this surface
must contain ultimately the long hair (only available as
Gauld's hand-written notes of the parc Bertrand, Geneva 2010).
\fi
%%%%%%%%%%%
Assume such long expansions can be performed as often as we
please, doing them in all directions of a disc yields a
$1$-connected separable surface with many long hairs emanating
in all directions (naively imagined as orthogonal to the
%boundary
circumference). Call this manifold the {\it supernova} (with
long hairs and complicated corona). Another more tangible
generating mode for the supernova is to do first a Moorization
(of the disc), and then make Nyikos' expansions to the thorns
(conceived now as independent processes). It is conceivable
that any foliation of the supernova contains all hairs as
semi-leaves, by a variant of (\ref{Baillif:nano-black-holes}).
(I~think that this was verified in David Gauld's hand-written
notes from the parc Bertrand, Geneva 2010, unfortunately
unpublished as yet.) Then one might (via a razor principle)
deduce a foliation of the compact disc (transverse to the
boundary). The supernova would thereby answer positively our
naive question (\ref{corona-sun:question}). Albeit it is not
the simplest to construct, it is perhaps the easiest to show
its lack of foliated structure.

%%%%%%CENSURABLE NOW EVENTUALLY

%%%%Some naive (open?) questions
\section{Miscellaneous}

This section collects miscellaneous topics not directly
relevant to the main text:

\subsection{Jordan separation}

Maybe first a question related to Jordan separation approached
via the Riemann polarization trick. Recall that we defined a
divisor in a manifold as a
%closed proper
locally-flat hypersurface
%%%(codimension-one submanifold)
which is closed as a point-set (\ref{divisor:def}). Taking for
granted the universality (i.e., non-metrical validity) of the
Riemann trick (\ref{Rieman-polarized-cover}) we deduced that
{\it in a simply-connected manifold any divisor is dividing}
(\ref{Riemann-separation}). The converse does not hold, for
instance the Poincar\'e (homology) sphere is divided by any
divisor, but not simply-connected. The former assertion
follows from the well-known:

\begin{lemma}\label{intersection-pairing} Given a closed (topological) manifold $M$ whose first homology
modulo 2, $H_1(M,{\Bbb Z}_2)$ is trivial (equivalently the
first Betti number $\beta_1$ with mod 2 coefficients, is
zero). Assume furthermore that the manifold as a reasonable
``intersection theory'', which is certainly the case whenever
$M$ is smoothable (Weyl 1923, Lefschetz 1930, de Rham 1931,
etc.). Then $M$
%splits (i.e.
is divided by any divisor.
\end{lemma}

\begin{proof} By contradiction, let
$H$ be a non-dividing divisor. Since $H$ is closed as a
point-set in $M$ compact, it is compact. Thus it carries a
fundamental class mod 2, denoted $[H]$. Choose any point $p\in
H$ and a locally flat chart $U$ with $(U,U\cap
H,p)\approx({\Bbb R}^n,{\Bbb R}^{n-1}\times \{0 \},0)$. Fix a
little arc $a$ transverse to $H$ and meeting it in exactly one
point. Since $M-H$ is connected, choose a path $b$ in $M-H$
joining both extremities of $a$. Thus $a+b=:c$ defines a
1-cycle (mod 2) intersecting $H$ in exactly one point
transversally. Thus the intersection number $[c]\cdot [H]=1$.
This implies that $[c]\neq 0$, and so $\beta_1\neq 0$.
\end{proof}

Consequently the Poincar\'e sphere, $M^3$, splits because
$H_1(M^3, {\Bbb Z})=0$ implies $H_1(M^3, {\Bbb Z}_2)=0$ (true
in any space, for a 1-cycle mod 2 can always be oriented). We
do not know if the lemma holds in full generality:

\begin{ques}
Assume the manifold $M$ to have $\beta_1=0$, then is $M$
separated by any divisor? And what about the converse?
\end{ques}

\subsection{Gr\"otzsch-Teichm\"uller theory for
non-metric surfaces?}

Since Rad\'o 1925 \cite{Rado_1925}, it is known that a
complex-analytic structure on a $2$-manifold implies second
countability, hence metrizability (Urysohn). Now such a
structure also known as a Riemann surface structure is
essentially the same as a conformal structure allowing one to
measure angles. In fact this can be achieved in the more
general category di-analytic or Klein surfaces, where
non-orientable surfaces are permitted. Rad\'o's theorem
generalizes directly to Klein surfaces, by passing to the
orientation double cover which is a Riemann surface (with an
anti-holomorphic involution):

\begin{lemma} Any Klein surface is
%metrizable.
metric.
\end{lemma}

\begin{proof} Above the Klein surface there is a Riemann
surface, which by Rad\'o \cite{Rado_1925} is metric, thus
Lindel\"of and the latter property pushes down to the Klein
surface, which being locally second countable is then second
countable, and  Urysohn concludes.
\end{proof}

Since conformal structures are lacking on non-metric surfaces,
and reminding the quasi-conformal trend (initiated in the now
classical works of Gr\"otzsch, 1928, Lavrentief 1929, Ahlfors
1935, Teichm\"uller 1938, etc.) it is rather natural to wonder
about quasi-conformal structures. This was addressed by R.\,J.
Cannon 1969 \cite{Cannon_1969}, who found rather surprising
answer(s). More on this soon, yet let us first dream a little.

If a diffeomorphism (between regions) of the plane (say of
class $(C^1)$) is not conformal, then it will distort
infinitesimal circle into ellipses, whose eccentricity $Q\ge
1$ (long axes divided by the short axes) provides a dilation
quotient measuring the deviation from conformality. The
diffeomorphism is {\it $K$-quasiconformal} if the distortion
$Q$ is bounded by a finite constant $1\le K<\infty$ throughout
its domain of definition.
% of the diffeomorphism.

A {\it $K$-quasiconformal structure} is defined by an atlas
with transition maps being $K$-qc. Of course a $1$-qc
structure is nothing else than a Klein structure, or a
conformal structure (non-orientable surfaces are welcome, at
least permitted).

Now the dream would be that given any (non-metric) surface
(say with a differentiable structure, albeit there is a way to
speak of quasi-conformality without regularity assumption),
then (following Gr\"otzsch's idea and phraseology) we could
look for the ``{\it m\"oglichst Konform}'' atlas minimizing
the angular distortion $Q$. Thus there is a way of assigning
to every surface $M$ a number $1\le Q(M)\le +\infty$ which is
the infimum of $K$ over all $K$-qc atlases. Of course we set
$Q(M)=+\infty$ if there is no such atlas, and $Q(M)=1$ holds
precisely when $M$ is metric (Rad\'o's theorem). This provides
a continuous numerical invariant quantifying how violently
non-metric the surface is. Unfortunately it seems that there
is no such fine quantification, for Cannon's main result is
that any $K$-quasi-conformal structure on a surface forces
metrizability so that $Q(M)$ can take only the values $1$ or
$+\infty$.

However Cannon also shows that some reasonable surfaces (of
the Pr\"ufer type) as well as surfaces deduced from the long
ray (Cantor type) allows quasi-conformal structures in the
weak sense that there is no uniform bound on the distortion
valid for all transition maps. This raises the question if
such  (weak) quasi-conformal structures exist for all
surfaces. (Maybe this would follow from the conjectural
smoothability of surfaces.)

\medskip
 {\small {\bf Acknowledgements.} The
 author
wishes to thank  Daniel Asimov,   Andr\'e Haefliger, Claude
Weber
%%%, Gerhard Wanner and Martin Gander
for conversations
 on topics lying in or outside the scope of the present
note. Last but not least, the
%non-metric
experts David Gauld and Mathieu Baillif are
%congratulated for
%their enthusiasms and
acknowledged for generous
%ideas
e-mails exchanges.

}

{\small
%%%%%%%%%%%%%%%%%%%SAXO%%%%%%%NEW BIBLIO

}

{
\hspace{+5mm} % To get a little bit of space between the figures
{\footnotesize
\begin{minipage}[b]{0.6\linewidth} Alexandre
Gabard

Universit\'e de Gen\`eve

Section de Math\'ematiques

2-4 rue du Li\`evre, CP 64

CH-1211 Gen\`eve 4

Switzerland

alexandregabard@hotmail.com
\end{minipage}
\hspace{-25mm} }


\begin{thebibliography}{30}


\bibitem{Ahlfors-Sario_1960}
L.\,V. Ahlfors, and L. Sario, \textsl{Riemann Surfaces},
Princeton Univ. Press, 1960 (Second Printing, 1965).


\bibitem{Baillif_2011}
M. Baillif, \textsl{Some steps on bridges: non-metric surfaces
and $CW$-complexes}, arXiv (2011).

\bibitem{BG_2008_PAMS}
M. Baillif, A. Gabard, \textsl{Manifolds: Hausdorffness versus
homogeneity}, Proc. Amer. Math. Soc. (2008), 1105--1111.


\bibitem{BGG1}
M. Baillif, A. Gabard and D. Gauld, \textsl{Foliations on
non-metrisable manifolds: absorption by a Cantor black hole},
arXiv (2009).

\bibitem{BGG2}
M. Baillif, A. Gabard and D. Gauld, \textsl{Foliations on
non-metrisable manifolds II: contrasted behaviours}, not yet
(but soon) available, arXiv (2011-12?).

\bibitem{Beck_1958}
A.~Beck, \textsl{On invariant sets}, Ann. of Math. (2) 67
(1958), 99--103.


\bibitem{Blohin_1972} A.\,A. Blohin, \textsl{Smooth
ergodic flows on surfaces}, Trans. Moscow Math. Soc. 27
(1972), 117--134.


\bibitem{Calabi-Rosenlicht_1953} E. Calabi and M. Rosenlicht,
\textsl{Complex analytic manifolds without countable base},
Proc. Amer. Math. Soc. 4 (1953), 335--340.


\bibitem{Cannon_1969} R.\,J. Cannon,
\textsl{Quasiconformal structures and the metrization of
$2$-manifolds}, Trans. Amer. Math. Soc. 135 (1969), 95--103.

\bibitem{Dubois-Violette_1949} Mme. Pierre-Louis Dubois-Violette,
\textsl{Sur les r\'eseaux de courbes couvrant une surface de
genre $p$}, C. R. Acad. Sci. Paris 228 (1949), 896--898.


\bibitem{Franks_1976} J.\,M. Franks,
\textsl{Two foliations in the plane}, Proc. Amer. Math. Soc.
58 (1976), 262--264.


\bibitem{Gabard_2008}
A. Gabard, \textsl{A separable manifold failing to have the
homotopy type of a $CW$-complex}, arXiv (2007) or Archiv. der
Math. 90 (2008), 267--274.



\bibitem{GaGa2010}
A. Gabard and D. Gauld, \textsl{Jordan and Schoenflies in
non-metrical analysis situs}, arXiv (2010) or New Zealand J.
Math. (2010).


\bibitem{GabGa_2011}
A. Gabard and D. Gauld, \textsl{Dynamics of non-metric
manifolds}, arXiv (2011).

\bibitem{Gab_2011_Hairiness}
A. Gabard, \textsl{Hairiness of $\omega$-bounded surfaces with
non-zero Euler characteristic}, arXiv (2011).


\bibitem{Gutierrez_1978_TAMS} C. Guti\'errez,
\textsl{Structural stability for flows on the torus with a
cross-cap}, Trans. Amer. Math. Soc. 241 (1978), 311--320.



\bibitem{Haefliger_1955}
A. Haefliger, \textsl{Sur les feuilletages des vari\'et\'es de
dimension $n$ par des feuilles ferm\'ees de dimension $n-1$},
Colloque de Topologie de Strasbourg (1955), 8 pages.

\bibitem{Haefliger_Reeb_1957}
A. Haefliger et G. Reeb, \textsl{Vari\'et\'es (non
s\'epar\'ees) \`a une dimension et structures feuillet\'ees du
plan}, L'Enseign. Math. (2) 3 (1957), 107--125.


\bibitem{Haefliger62}
A.~Haefliger, \textsl{Vari\'et\'es feuillet\'ees}, Ann. Scuola
Norm. Sup. Pisa  16 (1962), 367--397.

\bibitem{Hajek_1968}
O.~H\'ajek, \textsl{Dynamical Systems in the Plane}, Academic
Press, London and New York, 1968.


\bibitem{HectorHirschA}
G. Hector, U. Hirsch, {\em Introduction to the Geometry of
Foliations, Part A}, Aspects of Mathematics, Vieweg \& Sohn,
Braunschweig, 1981.

\bibitem{HectorHirschB}
G. Hector, U. Hirsch, {\em Introduction to the Geometry of
Foliations, Part B}, Aspects of Mathematics, Vieweg \& Sohn,
Braunschweig, 1983.


\bibitem{Hirsch_1961}
M.\,W.~Hirsch, \textsl{On imbedding differentiable manifolds
in Euclidean space}, Ann. of Math. (2)  73 (1961), 566--571.

\bibitem{Hopf_1946_1956}
H. Hopf, \textsl{Differential Geometry in the Large}, Seminar
Lectures New York University 1946 and Stanford Univ. 1956,
Lecture Notes in Math. 1000, Second edition, Springer-Verlag,
1989.


\bibitem{Jimenez_2004}
V.~Jim\'enez L\'opez and G. Soler L\'opez, \textsl{Transitive
flows on manifolds}, Rev. Mat. Iberoamericana 20 (2004),
107--130.


\bibitem{Kerekjarto_1923}
B.~von \Kerekjarto, \textsl{Vorlesungen \"uber Topologie, I,
Fl\"achen Topologie}, Die Grundlehren der mathematischen
Wissenschaften in Einzeldarstellungen 96, Band VIII, Verlag
von Julius Springer, Berlin, 1923.


\bibitem{Kerekjarto_1925}
B. de \Kerekjarto, {\em On a geometric theory of continuous
groups. I. Families of path-curves of continuous one-parameter
groups of the plane}, Ann. of Math. (2) 27 (1925), 105--117.

\bibitem{Kneser24}
H.~Kneser, \textsl{Regul\"are Kurvenscharen auf den
Ringfl\"achen}, Math. Ann. 91 (1924), 135--154.


\bibitem{Lefschetz_1937}
S. Lefschetz, \textsl{On the fixed point formula}, Ann. of
Math. (2) 38 (1937), 819--822.


\bibitem{Markley_1969} N.\,G. Markley, \textsl{The
Poincar\'e-Bendixson theorem for the Klein bottle}, Trans.
Amer. Math. Soc. 135 (1969), 159--165.



\bibitem{Massey_1967}
W.\,S. Massey,  \textsl{Algebraic Topology: An Introduction},
Graduate Texts in Mathematics~56, Springer-Verlag, New York,
Heidelberg, Berlin, 1967.


\bibitem{Mather_1982} J.\,N. Mather, \textsl{Foliations of
surfaces: I, an ideal boundary}, Ann. Inst. Fourier 32, 1
(1982), 235--261.


\bibitem{Moebius_1863}
A.\,F. M\"obius, \textsl{Theorie der elementaren
Verwandschaft}, Ber. Verhandl. K\"onigl. S\"achs. Gesell. d.
Wiss., mat.-phys. Klasse 15  (1863), 18--57. (M\"obius Werke
II).


\bibitem{Moise_1977}
E.\,E. Moise,  \textsl{Geometric Topology in Dimensions 2 and
3}, Graduate Texts in Mathematics~47, Springer-Verlag, New
York, Heidelberg, Berlin, 1977.


\bibitem{Nyikos84}
P.\,J.~Nyikos, \textsl{The theory of nonmetrisable manifolds},
In: {\em Handbook of Set-Theoretic Topology}, ed.
 Kunen and Vaughan, 633--684, North-Holland,
Amsterdam, 1984.


\bibitem{Peixoto_1962}
M.\,M.~Peixoto, \textsl{Structural stability on
two-dimensional manifolds}, Topology 1 (1962), 101--120.


\bibitem{Rado_1925}
T.~Rad\'o, \textsl{\"Uber den Begriff der Riemannschen
Fl\"ache}, Acta Szeged {2} (1925), 101--121.


\bibitem{Rosenberg_1983} H. Rosenberg, \textsl{Labyrinths
in the disc and surfaces}, Ann. of Math. (2) 117 (1983),
1--33.


\bibitem{Samelson_1965-homology} H. Samelson, \textsl{On Poincar\'e
duality}, J. Anal. Math.  14 (1965), 323--336.



\bibitem{Siebenmann72}
L.\,C.~Siebenmann, \textsl{Deformation of homeomorphisms on
stratified sets}, Comment. Math. Helv. 47 (1972), 123--163.



\bibitem{Siebenmann_2005} L. Siebenmann, \textsl{The
Osgood-Schoenflies theorem revisited}, Russian Math. Surveys {
60} (2005), 645--672. Cf. also the online version available in
the Hopf archive:
http://hopf.math.purdue.edu/cgi-bin/generate?/Siebenmann/Schoen-02Sept2005
(from which a number of misprints in the printed version have
been removed.)

\bibitem{Speiser_1927}
A.~Speiser, \textsl{Naturphilosophische Untersuchungen von
Euler und Riemann}, J. reine angew. Math. 157 (1927),
105--114.


\bibitem{Whitney33}
H. Whitney, \textsl{Regular families of curves}, Ann. of Math.
(2)  34 (1933), 244--270.


\end{thebibliography}
\end{document}